\documentclass[11pt]{amsart}

\usepackage{amsmath, amssymb, amsthm, amsfonts}

\input xy
\xyoption{all}

\usepackage{hyperref}

\swapnumbers
\numberwithin{equation}{section}

\theoremstyle{plain}
\newtheorem{theorem}[subsubsection]{Theorem}
\newtheorem{lemma}[subsubsection]{Lemma}
\newtheorem{prop}[subsubsection]{Proposition}
\newtheorem{cor}[subsubsection]{Corollary}

\theoremstyle{definition}
\newtheorem{defn}[subsubsection]{Definition}

\newtheorem{remark}[subsubsection]{Remark}
\newtheorem{exam}[subsubsection]{Example}

\def\AA{\mathbb{A}}
\def\BB{\mathbb{B}}
\def\CC{\mathbb{C}}
\def\DD{\mathbb{D}}

\def\FF{\mathbb{F}}
\def\GG{\mathbb{G}}

\def\PP{\mathbb{P}}
\def\QQ{\mathbb{Q}}
\def\RR{\mathbb{R}}

\def\ZZ{\mathbb{Z}}

\def\calE{\mathcal{E}}
\def\calF{\mathcal{F}}

\def\calH{\mathcal{H}}
\def\calI{\mathcal{I}}

\def\calK{\mathcal{K}}

\def\calO{\mathcal{O}}

\def\calS{\mathcal{S}}

\def\bI{\mathbf{I}}

\def\bR{\mathbf{R}}

\newcommand\frG{\mathfrak{G}}

\newcommand\tilD{\widetilde{D}}

\def\hatG{\widehat{G}}
\def\hatH{\widehat{H}}
\def\hatT{\widehat{T}}

\newcommand{\Bun}{\textup{Bun}}

\newcommand{\cyc}{\textup{cyc}}
\newcommand\ev{\textup{ev}}

\newcommand\Frob{\textup{Frob}}
\newcommand\Gal{\textup{Gal}}
\newcommand\geom{\textup{geom}}

\newcommand{\Gr}{\textup{Gr}}
\newcommand{\gr}{\textup{gr}}

\newcommand\Hk{\textup{Hk}}
\newcommand\IC{\textup{IC}}

\renewcommand{\Im}{\textup{Im}}
\newcommand{\Ind}{\textup{Ind}}

\newcommand\Isom{\textup{Isom}}

\newcommand\Lie{\textup{Lie}\ }
\newcommand\Loc{\textup{Loc}}

\newcommand{\Nm}{\textup{Nm}}
\newcommand\opp{\textup{opp}}
\newcommand\Out{\textup{Out}}
\newcommand\Perv{\textup{Perv}}

\newcommand\pr{\textup{pr}}
\newcommand\prim{\textup{prim}}

\newcommand\Rep{\textup{Rep}}
\newcommand{\Res}{\textup{Res}}

\newcommand\sgn{\textup{sgn}}
\newcommand\Span{\textup{Span}}
\newcommand\Spec{\textup{Spec}\ }
\newcommand\St{\textup{St}}

\newcommand\Swan{\textup{Swan}_\infty}

\newcommand\Sym{\textup{Sym}}

\newcommand{\Tr}{\textup{Tr}}

\newcommand{\val}{\textup{val}}
\newcommand{\Vect}{\textup{Vect}}

\newcommand\Aut{\textup{Aut}}
\newcommand\Hom{\textup{Hom}}

\newcommand\GL{\textup{GL}}

\newcommand\PGL{\textup{PGL}}

\newcommand\SL{\textup{SL}}

\newcommand\SO{\textup{SO}}

\newcommand\Og{\textup{O}}
\newcommand\GO{\textup{GO}}
\newcommand\Sp{\textup{Sp}}

\newcommand\GSp{\textup{GSp}}

\newcommand\GSpin{\textup{GSpin}}

\newcommand{\Gm}{\GG_m}
\def\Ga{\GG_a}

\newcommand{\ad}{\textup{ad}}
\newcommand{\Ad}{\textup{Ad}}

\newcommand\xch{\mathbb{X}^*}
\newcommand\xcoch{\mathbb{X}_*}

\newcommand{\one}{\mathbf{1}}
\newcommand{\incl}{\hookrightarrow}
\newcommand{\isom}{\stackrel{\sim}{\to}}
\newcommand{\surj}{\twoheadrightarrow}
\newcommand{\leftexp}[2]{{\vphantom{#2}}^{#1}{#2}}
\newcommand{\pH}{\leftexp{\textup{p}}{\textup{H}}}
\newcommand{\pD}{\leftexp{\textup{p}}{D}}

\newcommand{\Ql}{\QQ_\ell}

\newcommand{\twtimes}[1]{\stackrel{#1}{\times}}
\newcommand{\jiao}[1]{\langle{#1}\rangle}

\newcommand{\wt}[1]{\widetilde{#1}}
\newcommand\quash[1]{}
\newcommand{\nth}{^{\textup{th}}}
\newcommand{\un}[1]{\underline{#1}}

\newcommand{\cohog}[2]{\textup{H}^{#1}({#2})}     % plain group
\newcommand{\cohoc}[2]{\textup{H}_{c}^{#1}({#2})}     % compact support
\newcommand{\upH}{\textup{H}}

\newcommand{\Kl}{\textup{Kl}}
\newcommand{\AS}{\textup{AS}}
\newcommand{\Cox}{\textup{Cox}}
\newcommand{\dCox}{\wt{\textup{Cox}}}
\newcommand{\diag}{\textup{diag}}
\newcommand{\Wt}{\textup{Wt}}
\newcommand{\Tt}{\textup{Tt}}
\newcommand{\Sat}{\textup{Sat}}
\newcommand{\ME}{\textup{ME}}

\newcommand{\bad}{\textup{bad}}
\newcommand{\Hod}{\textup{Hod}}
\newcommand{\dR}{\textup{dR}}

\newcommand{\ep}{\epsilon}
\renewcommand{\l}{\lambda}

\newcommand{\Fpbar}{\overline{\FF}_p}
\newcommand{\Qlbar}{\overline{\QQ}_\ell}
\newcommand{\Qbar}{\overline{\QQ}}
\newcommand{\Qpbar}{\overline{\QQ}_{p}}
\newcommand{\kbar}{\overline{k}}
\newcommand{\rhobar}{\overline{\rho}}
\newcommand{\rhowtbar}{\overline{\wt{\rho}}}
\newcommand{\ebar}{\overline{e}}
\newcommand{\Dbar}{\overline{D}}
\newcommand{\tile}{\widetilde{e}}
\newcommand{\bpr}{\overline{\textup{pr}}}
\newcommand{\barl}{\overline{\lambda}}
\newcommand{\barF}{\overline{F}}
\newcommand{\barM}{\overline{M}}
\newcommand{\barv}{\overline{v}}
\newcommand{\tilv}{\widetilde{v}}
\newcommand{\barZ}{\overline{Z}}
\newcommand{\bara}{\overline{a}}

\newcommand{\Grot}{\Gm^{\textup{rot}}}
\newcommand{\GR}{\textup{GR}}
\renewcommand{\c}{\circ}
\newcommand{\GRc}{\textup{GR}^\c}
\newcommand{\Grc}{\textup{Gr}^\c}
\newcommand{\ICc}{\textup{IC}^\c}

\newcommand{\pic}{\pi^\c}
\newcommand{\Four}{\textup{Four}}
\newcommand{\HF}{\textup{HFour}}
\newcommand{\GmS}{\GG_{m,S}}
\newcommand{\bunF}{\underline{\overline{F}}}

\newcommand{\Gk}{\Gal(\kbar/k)}
\newcommand{\GQ}{\Gal(\overline{\QQ}/\QQ)}
\newcommand{\GQp}{\Gal(\overline{\QQ}_{p}/\QQ_{p})}
\newcommand{\GQl}{\Gal(\overline{\QQ}_{\ell}/\QQ_{\ell})}

\newcommand{\unF}{\underline{F}}
\newcommand{\unG}{\underline{G}}
\newcommand{\unK}{\underline{K}}
\newcommand{\unM}{\underline{M}}
\newcommand{\unN}{\underline{N}}
\newcommand{\unP}{\underline{P}}

\newcommand{\unGm}{\underline{\Gm}}
\newcommand{\unGr}{\underline{\Gr}}
\newcommand{\unGR}{\underline{\GR}}
\newcommand{\unGRc}{\underline{\GR}^{\c}}
\newcommand{\unfrG}{\underline{\frG}}
\newcommand{\unBun}{\underline{\Bun}}
\newcommand{\unIC}{\underline{\IC}}
\newcommand{\unICc}{\underline{\IC}^{\c}}
\newcommand{\unS}{\underline{\calS}}

\newcommand{\Zl}{\ZZ[\ell^{-1}]}
\newcommand{\SZl}{\Spec\Zl}
\newcommand{\Zp}{\ZZ_{p}}
\newcommand{\SZp}{\Spec\Zp}
\title[Galois representations and conjectures of Evans]{Galois representations attached to moments of Kloosterman sums and conjectures of Evans}

\author[Zhiwei Yun]{Zhiwei Yun \\ With an appendix by Christelle Vincent}
\thanks{Supported by the NSF grants  DMS-1261660 and DMS-1302071.}
\address{Department of Mathematics, Stanford University, 450 Serra Mall, Building 380, Stanford, CA 94305}
\email{zwyun@stanford.edu}
\address{Department of Mathematics, Stanford University, 450 Serra Mall, Building 380, Stanford, CA 94305}
\email{cvincent@stanford.edu}

\date{}

\begin{document}

\begin{abstract}
Kloosterman sums for a finite field $\FF_{p}$ arise as Frobenius trace functions of certain local systems defined over $\GG_{m,\FF_{p}}$. The moments of Kloosterman sums calculate the Frobenius traces on the cohomology of tensor powers (or symmetric powers, exterior powers, etc.) of these local systems. We show that when $p$ ranges over all primes, the moments of the corresponding Kloosterman sums for $\FF_{p}$ arise as Frobenius traces on a continuous $\ell$-adic representation of $\GQ$ that comes from geometry. We also give bounds on the ramification of these Galois representations. All this is done in the generality of Kloosterman sheaves attached to reductive groups introduced in \cite{HNY}. As an application, we give proofs of conjectures of R. Evans (\cite{Evans0}, \cite{Evans}) expressing the seventh and eighth symmetric power moments of the classical Kloosterman sum in terms of Fourier coefficients of explicit modular forms. The proof for the eighth symmetric power moment conjecture relies on the computation done in Appendix B by C.Vincent.
 
\end{abstract}

\maketitle
\tableofcontents
%%% intro %%%%equation*}

\section{Introduction}

\subsection{Kloosterman sums and their moments}\label{ss:intro Kl}
We first recall the definition of Kloosterman sums.
\begin{defn}\label{def:Kl sum} Let $p$ be a prime number. Fix a nontrivial additive character $\psi:\FF_p\to\CC^{\times}$. Let $n\geq2$ be an integer. Then the $n$-variable Kloosterman sum over $\FF_p$ is a complex valued function on $\FF^{\times}_{p}$ whose value at $a\in\FF_{p}^{\times}$ is
\begin{equation*}
\Kl_n(p;a)=\sum_{x_1,\cdots,x_n\in\FF^\times_p;x_1x_{2}\cdots x_{n}=a}\psi(x_1+\cdots+x_n).
\end{equation*}
\end{defn}
These exponential sums arise naturally in the study of automorphic forms for $\GL_{n}$.

\subsubsection{Deligne's Kloosterman sheaf}\label{Deligne Kl} Deligne \cite{Del} has given a geometric interpretation of the Kloosterman sum. He considers the following diagram consisting of schemes over $\FF_{p}$
\begin{equation*}
\xymatrix{& \Gm^n\ar[dl]_{\pi}\ar[dr]^{F}\\ \Gm & & \AA^1}
\end{equation*}
Here the morphism $\pi$ is taking the product and $F$ is the morphism of taking the sum. Deligne defines
\begin{equation*}
\Kl_n:=\bR^{n-1}\pi_!F^*\AS_\psi,
\end{equation*}
where $\AS_{\psi}$ is the rank one local system on $\AA^{1}_{\FF_{p}}$ corresponding to the character $\psi:\FF_{p}\isom\mu_{p}(\CC)\cong\mu_{p}(\Qlbar)$ (we make a choice of the latter isomorphism). He then shows that $\Kl_n$ is a $\Ql(\mu_{p})$-local system on $\GG_{m,\FF_p}$ of rank $n$. The relationship between the local system $\Kl_{n}$ and the Kloosterman sum $\Kl_{n}(p;a)$ is explained by the following identity
\begin{equation*}
\Kl_n(p;a)=(-1)^{n-1}\Tr(\Frob_a,(\Kl_{n})_a).
\end{equation*}
Here $\Frob_{a}$ is the geometric Frobenius operator acting on the geometric stalk $(\Kl_{n})_{a}$ of the local system $\Kl_{n}$ at $a\in\Gm(\FF_{p})=\FF_{p}^{\times}$. The Kloosterman sheaf has been studied in depths by Deligne \cite{Del} and Katz \cite{Katz}. 

For each representation $V$ of $\GL_{n}$, we may form the corresponding local system $\Kl^{V}_{n}$. For example, if $V=\Sym^{d}$, the $d\nth$ symmetric power of the standard representation of $\GL_{n}$, then $\Kl^{V}_{n}=\Sym^{d}(\Kl_{n})$. We define
\begin{equation*}
\Kl^V_n(p;a):=\Tr(\Frob_a,(\Kl^V_{n})_a).
\end{equation*}
For example, if $n=2$, Deligne's interpretation gives an expression $\Kl_{n}(p;a)=-(\alpha_{a}+\beta_{a})$ where $\alpha_{a},\beta_{a}\in\Qlbar$ are the eigenvalues of the operator $\Frob_{a}$ on $(\Kl_{n})_{a}$. Taking $V=\Sym^{d}$, we have
\begin{equation*}
\Kl^{\Sym^{d}}_{n}(p;a)=\sum_{i=0}^{d}\alpha_{a}^{i}\beta_{a}^{d-i}.
\end{equation*}

\begin{defn}\label{def:intro moments} Let $V$ be a representation of $\GL_{n}$. The {\em $V$-moment} of $\Kl_n$ is the sum
\begin{equation*}
m^V_{n}(p):=\sum_{a\in\FF_p^\times}\Kl^V_n(p;a).
\end{equation*}
\end{defn}

One may ask whether there is a closed formula for $m^{V}_{n}(p)$ in terms of more familiar arithmetic functions of the prime $p$. A related question is whether the behavior of the sequence of numbers $\{m^{V}_{n}(p)\}$ exhibits certain ``motivic'' nature. A typical example of a motivic sequence is the Fourier coefficients $\{a_{f}(p)\}$ of a holomorphic Hecke eigenform $f$. Another typical example is $\{\#X(\FF_{p})\}$ where $X$ is an algebraic variety defined over $\ZZ$. In this direction, Fu and Wan \cite{FW KlooZ} showed that for $V=\Sym^{d}, \wedge^{d}$ or $\otimes^{d}$, the moments $m^{V}_{n}(p)$ are the Frobenius traces on a {\em virtual} motivic Galois representation of $\GQ$.

Ron Evans has made very precise conjectures about the moments $m^{V}_{2}(p)$ when $V=\Sym^{d}$ and $d=5,6,7$ or $8$, which we now recall.

\subsection{Conjectures of Evans}\label{conj} Let us set $n=2$, and take $V=\Sym^{d}$. We denote $m^{\Sym^{d}}_{2}(p)$ by $m^{d}_{2}(p)$. Evans made the following conjectures, see \cite[p.523]{Evans0} and \cite[p.350]{Evans}.
\begin{enumerate}
\item When $d=5$, $(-m^{5}_{2}(p)-1)/p^{2}$ is the $p\nth$ Fourier coefficient of a holomorphic cuspdial Hecke eigenform of weight 3, level $\Gamma_{0}(15)$ and quadratic nebentypus $\left(\frac{\cdot}{15}\right)$. This is proved by Livn\'e \cite{Livne} and Peters-Top-van der Vlugt \cite{PTV}. % check with Evans.

\item When $d=6$, $(-m^{6}_{2}(p)-1)/p^{2}$ is the $p\nth$ Fourier coefficient of a holomorphic Hecke eigenform of weight 4 and level $\Gamma_{0}(6)$. This is proved by Hulek-Spandaw-van Geemen-van Straten \cite{HSVV} %\footnote{Check with Evans}.

\item\label{intro:7} (See \cite[Conjecture 1.1]{Evans}) When $d=7$, we should have for all primes $p>7$
\begin{equation}\label{m72}
\left(\frac{p}{105}\right)(-m^{7}_{2}(p)-1)/p^{2} = a_{f}(p)^{2}\ep_{f}(p)^{-1}-p^{2},
\end{equation}
where $a_{f}(p)$ is the $p\nth$ Fourier coefficient of a holomorphic cuspdial Hecke eigenform $f$ of weight 3, level $\Gamma_{0}(525)$ and nebentypus $\ep_{f}=\left(\frac{\cdot}{21}\right)\ep_{5}$, where $\ep_{5}$ is a quartic character of conductor 5. We will give a proof of this conjecture.

\item\label{intro:8} When $d=8$, we should have for all primes $p\geq3$
\begin{equation*}
-m^{8}_{2}(p)-1-p^{4}=p^{2}a_{f}(p),
\end{equation*}
where $a_{f}(p)$ is the $p\nth$ Fourier coefficient of a holomorphic Hecke eigenform $f$ of weight 6,  level $\Gamma_{0}(6)$. We will give a proof of this conjecture.
\end{enumerate}

\begin{theorem}[proved in \S\ref{pf Sym7}]\label{th:Sym7} Evans's conjecture \S\ref{conj}\eqref{intro:7} on the $\Sym^{7}$-moment of $\Kl_{2}$ is true.
\end{theorem}

Conjecture \S\ref{conj}\eqref{intro:8} will be reduced to a finite calculation in the main body of the paper, and will be verified by computational softwares in the Appendix by Christelle Vincent.

\begin{theorem}[proved in \S\ref{pf Sym8} and Appendix \ref{app:Vincent}]\label{th:Sym8} 
\begin{enumerate}
\item There exist integers $1\leq k\leq 4$, $0\leq e\leq 8$ and a holomorphic cuspdial Hecke eigenform $f$ of weight $2k$, level $\Gamma_{0}(2^{e}\cdot 3)$ and rational Fourier coefficients such that for all primes $p\geq3$, we have
\begin{equation}\label{m82}
-m^{8}_{2}(p)-1-p^{4}=p^{5-k}a_{f}(p).
\end{equation}
Here $a_{f}(p)$ is the $p\nth$ normalized Fourier coefficient of $f$. 
\item Evans's conjecture \S\ref{conj}\eqref{intro:8} on the $\Sym^{8}$-moment of $\Kl_{2}$ is true.
\end{enumerate}
\end{theorem}

The proof of these results are based on two key ingredients. One is Serre's modularity conjecture \cite{Serre} proved by Khare and Wintenberger \cite{KW}; the other is the following existence result on the Galois representations (not just a virtual one) underlying the moments $m^{d}_{2}(p)$, which is a very special case of our main result (to be stated later).

\begin{theorem}[proved in \S\ref{pf Kl2}]\label{th:intro} 
\begin{enumerate}
\item[]
\item Let $d\geq3$ be an odd integer. Then for each prime $\ell$, there exists an orthogonal $\Ql$-vector space $M_{\ell}$ of dimension $(d-1)/2$ and a continuous Galois representation
\begin{equation*}
\rho_{\ell}:\GQ\to\Og(M_{\ell})
\end{equation*}
with the following properties
\begin{itemize}
\item The determinant of $\rho_{\ell}$ is the quadratic character given by the Jacobi symbol $\left(\frac{\cdot}{d!!}\right)$, where $d!!=d(d-2)(d-4)\cdots1$.
\item The representation $\rho_{\ell}$ comes from geometry. More precisely, there is a smooth projective algebraic variety  $X$ of dimension $d-1$ over $\QQ$ (independent of $\ell$) such that the $\GQ$-module $M_{\ell}$ appears as a subquotient of $\cohog{d-1}{X_{\Qbar},\Ql}(\frac{d-1}{2})$.
\item For all primes $p\neq\ell$ satisfying $p>d$ or $p=2$, $\rho_{\ell}$ is unramified at $p$, and we have
\begin{equation}\label{md2 odd}
-m^{d}_{2}(p)-1=p^{\frac{d+1}{2}}\Tr(\Frob_{p},M_{\ell}).
\end{equation}
\end{itemize}

\item Let $d\geq4$ be an even integer. Then for each prime $\ell$, there exists a symplectic $\Ql$-vector space $M_{\ell}$ of dimension $2[(d+2)/4]-2$ and a continuous Galois representation
\begin{equation*}
\rho_{\ell}:\GQ\to\GSp(M_{\ell})
\end{equation*}
with the following properties
\begin{itemize}
\item The similitude character of $\rho_{\ell}$ is the $(-d-1)\nth$ power of the $\ell$-adic cyclotomic character.
\item The representation $\rho_{\ell}$ comes from geometry. More precisely, there is a smooth projective algebraic variety  $X$ of dimension $d-1$ over $\QQ$ (independent of $\ell$) such that the $\GQ$-module $M_{\ell}$ appears as a subquotient of $\cohog{d-1}{X_{\Qbar},\Ql}(-1)$.
%\item For all primes $p\neq\ell$ and $p>d/2$, $\rho_{\ell}$ is unramified at $p$, and we have
%\begin{equation}\label{md2 even}
%-m^{d}_{2}(p)-1=\Tr(\Frob_{p},M_{\ell})+\begin{cases}p^{d/2} & d\equiv0\mod4;\\0&d\equiv2\mod4.\end{cases}
%\end{equation}
\item For a prime $p\neq2,p\neq\ell$, $M_{\ell}$ as a $\GQp$-module can be decomposed into a direct sum of symplectic $\GQp$-modules
\begin{equation}\label{md2 even ram}
M_{\ell}|_{\GQp}\cong J_{2}(-d/2)^{[\frac{d}{2p}]}\oplus U_{p}
\end{equation}
Here the action of $\GQp$ on $U_{p}$ is unramified, and $J_{2}$ is the unique two-dimensional representation of $\GQp$ which is an extension of $\Ql(-1)$ by $\Ql$ and on which the action of the inertia is unipotent but nontrivial (i.e., $J_{2}$ corresponds to the Steinberg representation under the local Langlands correspondence for $\GL_{2}(\QQ_{p})$). Moreover we have
\begin{equation}\label{md2 even}
-m^{d}_{2}(p)-1=\Tr(\Frob_{p},M_{\ell}^{\calI_{p}})+\begin{cases}p^{d/2} & d\equiv0\mod4;\\0&d \equiv2\mod4.\end{cases}
\end{equation}
In particular, $\rho_{\ell}$ is unramified at primes $p>d/2, p\neq\ell$.

%\begin{equation}\label{md2 even}
%-m^{d}_{2}(p)=\begin{cases} 1+([\frac{d}{2p}]+1)p^{d/2}+\Tr(\Frob_{p},\Im(M^{\calI_{p}}_{\ell}\to (M_{\ell})_{\calI_{p}})) & d\equiv0\mod4, p>2;\\
%1+(-1)^{d/4}2^{d/2}+\Tr(\Frob_{2},\Im(M^{\calI_{2}}_{\ell}\to (M_{\ell})_{\calI_{2}})) & d\equiv0\mod4, p=2;\\
%1+[\frac{d}{2p}]p^{d/2}+\Tr(\Frob_{p},\Im(M^{\calI_{p}}_{\ell}\to (M_{\ell})_{\calI_{p}})) & d\equiv2\mod4, p>2;\\
%1+\Tr(\Frob_{2},\Im(M^{\calI_{2}}_{\ell}\to (M_{\ell})_{\calI_{2}})) & d\equiv2\mod4, p=2.\end{cases}
%\end{equation}
\end{itemize}
\end{enumerate}
\end{theorem}

We mention that when $d=7$, Katz \cite{KE} proposed that there should exist an Galois representation into $\Og_{3}$ as in the above theorem that underlies the modular form predicted by Evans.

\subsection{General case of the main result} In the main body of the paper, we work with more general Kloosterman sheaves than Deligne's. Motivated by the work of Frenkel and Gross \cite{FG}, Heinloth, Ng\^o and the author \cite{HNY} construct a Kloosterman sheaf $\Kl_{\hatG}$ for each almost simple split algebraic group $\hatG$ over $\Ql$. It has the following properties.
\begin{enumerate}
\item $\Kl_{\hatG}$ is a $\hatG^{\Tt}(\Ql(\mu_{p}))$-local system over $\GG_{m,\FF_{p}}=\PP^{1}_{\FF_{p}}-\{0,\infty\}$. Here $\hatG^{\Tt}$ is a slight modification of $\hatG$ (after Deligne) in order to avoid half Tate-twists, see \S\ref{modified dual}. In other words, for every representation $V$ of $\hatG^{\Tt}$ there is a local system $\Kl^{V}_{\hatG}$ over $\GG_{m,\FF_{p}}$ of rank $\dim V$, such that the assignment $V\mapsto \Kl^{V}_{\hatG}$ is compatible with the formation of tensor products.
\item The local geometric monodromy of $\Kl_{\hatG}$ at the puncture $0\in\PP^{1}$ is a regular unipotent element in $\hatG$.
\item The local geometric monodromy of $\Kl_{\hatG}$ at the puncture $\infty\in\PP^{1}$ is a ``simple wild parameter'' \`a la Gross and Reeder \cite{GR}. This roughly means that the Swan conductor at $\infty$ of the adjoint local system $\Kl^{\Ad}_{\hatG}$ takes the smallest possible nonzero value.
\item When $\hatG=\SL_{n}$, $\Kl_{\hatG}$ is the Kloosterman sheaf $\Kl_{n}$ defined by Deligne.
\end{enumerate}

Let $V$ be a representation of $\hatG^{\Tt}$. We define a ``balanced'' version of the $V$-moment for $\Kl_{\hatG}$ to be the Frobenius module
\begin{equation}\label{intro M!*}
M^{V}_{!*,\FF_{p}}=\Im(\cohoc{1}{\GG_{m,\Fpbar},\Kl^{V}_{\hatG}}\to\cohog{1}{\GG_{m,\Fpbar},\Kl^{V}_{\hatG}}).
\end{equation}
When $\hatG=\SL_{n}$, we have defined the $V$-moment of $\Kl_{n}$ in Definition \ref{def:intro moments}, which is equal to the alternating trace of Frobenius on the cohomology $\cohoc{*}{\Gm,\Kl^{V}_{n}}$ according to the Lefschetz trace formula. Therefore $\Tr(\Frob, M^{V}_{!*,\FF_{p}})$ does {\em not} directly give the moment $m^{V}_{n}(p)$. However, the difference between $\Tr(\Frob, M^{V}_{!*,\FF_{p}})$ and the moment $m^{V}_{n}(p)$ is explicitly computable in many examples. It turns out that $M^{V}_{!*,\FF_{p}}$ is a better object to work with.

Our main result is a generalization of Theorem \ref{th:intro} to all moments of all Kloosterman sheaves $\Kl_{\hatG}$.

\begin{theorem}[proved in \S\ref{Pf mot} and \S\ref{Pf ram}]\label{th:main} Let $V$ be the irreducible representation of $\hatG$ with highest weight $\l$ and we form the $V$-moment $M^{V}_{!*,\FF_{p}}$ of $\Kl_{\hatG}$ as in \eqref{intro M!*}\footnote{We will see in \S\ref{irr} that $V$ can be viewed as a $\hatG^{\Tt}$-representation in a natural way, which is pure of weight $w=\jiao{2\rho,\l}$. Therefore $\Kl^{V}_{\hatG}$ and $M^{V}_{!*,\FF_{p}}$ are defined.}. Suppose $V^{\hatG^{\geom}}=0$ (here $\hatG^{\geom}\subset\hatG$ is the Zariski closure of the geometric monodromy of $\Kl_{\hatG}$ for large $p$, to be recalled in Theorem \ref{th:global mono}). Let $w=\jiao{2\rho,\l}$ where $2\rho$ denotes the sum of positive coroots for $\hatG$. Then for each prime $\ell$ there exists a finite-dimensional $\Ql$-vector space $M_{\ell}$ with a continuous $\GQ$-action and a $\GQ$-equivariant $(-1)^{w+1}$-symmetric perfect pairing
\begin{equation*}
M_{\ell}\otimes M_{\ell}\to \Ql(-w-1)
\end{equation*}
so that the Galois representation
\begin{eqnarray*}
&&\rho^{V}_{\ell}:\GQ\to\GO(M_{\ell}) \mbox{ when $w$ is odd}\\
\textup{or} &&\rho^{V}_{\ell}:\GQ\to\GSp(M_{\ell}) \mbox{ when $w$ is even}
\end{eqnarray*}
has the following properties.
\begin{enumerate}
\item\label{motivic} It comes from geometry. More precisely,  there exists a smooth projective algebraic variety $X$ over $\QQ$ (independent of $\ell$), of dimension $w-1$, such that $\rho^{V}_{\ell}$ is a subquotient of the $\GQ$-module $\cohog{w-1}{X,\Ql}(-1)$.
%% In particular, $\rho^{V}_{\ell}|_{\Gal(\Qlbar/\Ql)}$ is de Rham with Hodge-Tate weights in $\{1,\cdots,\jiao{2\rho,\l}\}$.
\item\label{ram} Let $p\neq\ell$ be a prime such that $\Kl^{V}_{\hatG}$ does not have geometrically trivial subquotients. Then we have a canonical inclusion of $\Frob_{p}$-modules
\begin{equation}\label{inv coinv}
M^{V}_{!*,\FF_{p}}\incl \Ql(\mu_{p})\otimes\Im(M_{\ell}^{\calI_{p}}\to(M_{\ell})_{\calI_{p}}).
\end{equation}
\end{enumerate}
\end{theorem}

The proof of the above theorem again has two key ingredients. 

First is the interpretation of the moments $M^{V}_{!*, \FF_{p}}$ in terms of {\em homogeneous Fourier transform} (Proposition \ref{p:MV}). Unlike Fourier-Deligne transform which relies on working with a single characteristic $p$, homogeneous Fourier transform (introduced by Laumon) is a ``characteristic-free'' version of Fourier-Deligne transform, and it serves as the bridge between moments in characteristic $p$ and motives over $\QQ$. We will give a quick review of Laumon's homogeneous Fourier transform in Appendix A. The algebraic variety $X$ in Theorem \ref{th:main}\eqref{motivic}  is closely related to affine Schubert varieties in the affine Grassmannian. 

Another ingredient is that in order to study the ramification behavior of the Galois representations $\rho^{V}_{\ell}$ we need to work with $\Zp$-models of various geometric objects such as affine Grassmannian and perverse sheaves on them. Here the results of G. Pappas and X. Zhu \cite{PZ} play a crucial role.

\subsection{Towards more conjectures of Evans type}
In view of Theorem \ref{th:main} and the Langlands correspondence, Evans's conjectures in \S\ref{conj} should be the first few examples of a large list of modularity problems relating moments of generalized Kloosterman sums and automorphic forms. For example, we may start with a simply-connected almost simple group $\hatG$ and consider $V$-moments of the Kloosterman sheaf $\Kl_{\hatG}$, where $V$ is the irreducible representation of $\hatG$ of highest weight $\l$. We then compute the dimension $d_{!*}^{V}$ of $M^{V}_{!*,\FF_{p}}$ for large $p$.

\begin{itemize}
\item If $d_{!*}^{V}=2$, and $\jiao{2\rho,\l}$ is odd, then we get Galois representations into $\GO_{2}$. The $V$-moment of $\Kl_{\hatG}$ should be expressible in terms of Fourier coefficients of a CM modular form. 

Examples: $(\hatG,V)=(\SL_{2}, \Sym^{5})$, $(\SL_{4}$ or $\Sp_{4}, \Sym^{3})$.

\item If $d_{!*}^{V}=2$, and $\jiao{2\rho,\l}$ is even, then we get 2-dimensional Galois representations, and the  $V$-moment of $\Kl_{\hatG}$ should be expressible in terms of Fourier coefficients of a modular form.

Examples: 
\begin{eqnarray*}
(\hatG,V)&=& (\SL_{2}, \Sym^{6}), (\SL_{2}, \Sym^{8}), \\
&& (\SL_{3}, \Sym^{4}), (\SL_{3}, V_{2\omega_{1}+\omega_{2}}), (\SL_{3}, V_{2\omega_{1}+2\omega_{2}}), \\
&& (\SL_{4}\textup{ or }\Sp_{4}, \Sym^{4}) \textup{ and its dual}, (\Sp_{4}, V_{3\omega_{2}}). 
\end{eqnarray*}
Here $\{\omega_{1},\omega_{2}\}$ denote the fundamental weights of $\SL_{3}$ or $\Sp_{4}$, with $\omega_{1}$ corresponding to the standard representation.

\item If $d_{!*}^{V}=3$ (in this case $\jiao{2\rho,\l}$ is necessarily odd),  then we get Galois representations into $\Og_{3}\cong\PGL_{2}\times\{\pm1\}$. The $\PGL_{2}$-part of these representations should be liftable to $\GL_{2}$, and we should again expect that the $V$-moment of $\Kl_{\hatG}$ be expressible in terms of a quadratic Dirichlet character and the Fourier coefficients of a modular form. 

Example: $(\hatG, V)=(\SL_{2}, \Sym^{7})$.

\item If $d_{!*}^{V}=4$ and $\jiao{2\rho,\l}$ is odd, then we get Galois representations into $\Og_{4}$.  If, moreover, the determinant character corresponds to a {\em real} quadratic field $F/\QQ$ and if $\Gal(\Qbar/F)\to\SO_{4}$ lifts to $\GSpin_{4}\incl\GL_{2}\times\GL_{2}$, then the $V$-moment of $\Kl_{\hatG}$ should be expressible in terms of Fourier coefficients of two Hilbert modular forms for $\GL_{2}$ over $F$.

Example: $(\hatG, V)=(\SL_{2}, \Sym^{9})$.

\item If $d_{!*}^{V}=4$ and $\jiao{2\rho,\l}$ is even, then we get Galois representations into $\GSp_{4}$.  We should then expect that the $V$-moment of $\Kl_{\hatG}$ be expressible in terms of a Siegel modular form for $\GSp_{4}$ over $\QQ$.

Examples: 
\begin{eqnarray*}
(\hatG,V)&=& (\SL_{2}, \Sym^{10}), (\SL_{2}, \Sym^{12})\\
&& (\SL_{3}, \Sym^{5}), (\SL_{3}, V_{3\omega_{1}+\omega_{2}}), (\SL_{3}, \Sym^{6}) \textup{ and their duals}, \\
&& (G_{2}, V_{\omega_{1}+\omega_{2}}), (G_{2}, V_{2\omega_{2}}), (G_{2}, V_{3\omega_{1}}).
\end{eqnarray*}
Here $\omega_{1}$ (resp. $\omega_{2}$) are the dominant short (resp. long) root of $G_{2}$.
\end{itemize}
And the list continues. We plan to investigate some of these examples in details in the future.

\subsection{Convention}
Throughout the paper, $G$ will be a connected, split, simple (hence {\em adjoint}) group over a field $k$.  We fix a maximal torus $T\subset G$ and a Borel subgroup $B$ containing $T$. This way we get a based root system and the Weyl group $W$.

We also fix a prime $\ell$ which is different from $\textup{char}(k)$. All sheaves will be $\Ql$-sheaves in the \'etale topology of schemes or stacks. {\bf Caution:} all sheaf-theoretic functors are derived functors unless otherwise stated.

When $k$ is a finite field, $\Frob_{k}$ will denote the {\em geometric} Frobenius element in $\Gk$. Similarly, for a prime $p$, $\Frob_{p}$ denotes the {\em geometric} Frobenius element in $\Gal(\Fpbar/\FF_{p})$.

%%% thanks %%%

\noindent\textbf{Acknowledgement} The author would like to thank Ken Ono for suggesting the problem. He also appreciates discussions with Brian Conrad, Lei Fu, Dick Gross, Ruochuan Liu and Christelle Vincent.

%%%%%%%%%%%

\section{Geometric setup for the moments}
In this section, we give geometric interpretation of moments of Kloosterman sums.  We will work with the more general Kloosterman sheaves constructed in our earlier work \cite{HNY}. After some preparation on the geometric Satake equivalence in \S\ref{Sat}, we review the construction of generalized Kloosterman sheaves and their monodromy properties in \S\ref{ss:Kl}-\S\ref{ss:classical}. Finally in \S\ref{ss:moments}-\S\ref{ss:autoduality} we define moments of Kloosterman sheaves and study them via homogeneous Fourier transform.

\subsection{Deligne's modification of the Langlands dual group and the Satake-Tate category}\label{Sat}
In this subsection, $k$ is either a finite field or a number field. Deligne has introduced a modification of the Langlands dual group in the setup of the global Langlands correspondence, in order to avoid making half Tate twists. We reformulate his construction from the viewpoint of the geometric Satake equivalence. For a quick algebraic account, see \cite[\S2]{FG}.

\subsubsection{Tate sheaves} The Tate sheaf $\Ql(1)$ is defined as an $\Ql$-sheaf on $\Spec k$, using the inverse system $\varprojlim_{n}\mu_{\ell^{n}}$ of $\ell$-power roots of unity in $\kbar$. Let $\delta^{\Tt}$ be the full subcategory of $\Ql$-sheaves on $\Spec k$ consisting of finite direct sums of $\Ql(n)$ for $n\in\ZZ$. Then $\delta^{\Tt}$ admits a structure of a tensor category using tensor product of sheaves. Fixing an algebraic closure $\kbar$ of $k$, the functor of taking the $\kbar$-stalk of a sheaf gives a fiber functor $\omega:\delta^{\Tt}\to\Vect_{\Ql}$, and realizes  a tensor equivalence
\begin{equation*}
\delta^{\Tt}\cong\Rep(\Gm^{\Tt}).
\end{equation*}
Here $\Gm^{\Tt}$ is simply the multiplicative group $\Gm$. Under this equivalence, the tautological 1-dimensional representation of $\Gm^{\Tt}$ corresponds to $\Ql(1)$.

\subsubsection{The geometric Satake equivalence} Let $LG$ be the loop group of $G$: this is an ind-scheme representing the functor $R\mapsto G(R[[t]])$. Let $L^{+}G$ the positive loops of $G$: this is a scheme of infinite type representing the functor $R\mapsto G(R[[t]])$. The fppf quotient $\Gr=LG/L^{+}G$ is called the {\em affine Grassmannian} of $G$.  Then $L^{+}G$ acts on $\Gr$ via left translation. The $L^{+}G$-orbits on $\Gr$ are indexed by  dominant coweights $\l\in\xcoch(T)^+$. The orbit containing the element $t^{\l}\in T(k((t)))$ is denoted by $\Gr_{\l}$ and its closure is denoted $\Gr_{\leq\l}$. We have $\dim\Gr_{\l}=\jiao{2\rho,\l}$, where $2\rho$ denotes the sum of positive roots in $G$. The reduced scheme structure on $\Gr_{\leq\l}$ is a projective variety called the affine Schubert variety attached to $\l$. We denote the intersection complex of $\Gr_{\leq\l}$ by $\IC_{\l}$: this is the middle extension of the shifted constant sheaf $\Ql[\jiao{2\rho,\l}]$ on $\Gr_{\l}$.

The {\em Satake category} $\Sat=\Perv_{L^{+}G}(\Gr)$ is the category of $L^{+}G$-equivariant perverse \'etale sheaves on $\Gr$ (over $k$) whose support is of finite type. Similarly, we define $\Sat^\geom$ by considering the base change of the situation to $\kbar$. Then we have a pullback functor $\Sat\to\Sat^{\geom}$.

In \cite{Lu}, \cite{Ginz} and \cite{MV}, it was shown that $\Sat^\geom$ carries a natural tensor structure (which is also defined for $\Sat$), such that the global cohomology functor $h=H^*(\Gr,-):\Sat^\geom\to\Vect$ is a fiber functor. It is also shown that the Tannakian group of the tensor category $\Sat^{\geom}$ is a connected reductive split group over $\Ql$ whose root system is dual to  that of $G$. We denote this group by $\hatG$. The Tannakian formalism gives the {\em geometric Satake equivalence} of tensor categories
\begin{equation*}
\Sat^\geom\cong\Rep(\hatG,\Ql).
\end{equation*}

\subsubsection{The Satake-Tate category}
Let  $\Sat^{\Tt}$ be the full subcategory of $\Sat$ whose objects are finite direct sums of the form $\oplus_{\l_{i},n_{i}}\IC_{\l_{i}}(n_{i})$, where $\l_{i}\in\xcoch(T)^{+}$ and $n_{i}\in\ZZ$. By \cite[\S3.5]{AB}, the category $\Sat^{\Tt}$ is closed under the tensor product in $\Sat$, and $\Sat^{\Tt}$ is itself a rigid tensor category. Alternatively, one can define $\Sat^{\Tt}$ as the smallest tensor subcategory of $\Sat$ containing the objects $\IC_{\lambda}$ which is also stable under Tate twists. We have an embedding of tensor categories $\delta^{\Tt}\subset\Sat^{\Tt}$ sending $\Ql(n)$ to $\IC_{0}(n)$, here $\IC_{0}$ is the skyscraper sheaf supported at the singleton orbit $\Gr_{0}$.

We have tensor functors
\begin{equation}\label{Sat Tt}
\xymatrix{\Sat^{w=0}\ar[dr] & & \delta^{\Tt}\ar[dl] \\
& \Sat^{\Tt}\ar[dr]^{\upH^{\Wt}}\ar[dl]\\
\Sat^{\geom} & & \Vect^{\gr}}
\end{equation}
We explain the notation. The functor $\upH^{\Wt}$ is the global section functor $\upH^{*}$ with the grading by weights. If we write an object $\calK\in\Sat^{\Tt}$ as $\calK=\oplus_{i}\calK_{i}$, where $\calK_{i}$ is pure of weight $i$, then $\upH^{\Wt}(\calK)$ is the graded vector space $\oplus_{i}h(\calK_{i})$. The category $\Sat^{w=0}$ is the full subcategory of $\Sat^{\Tt}$ consisting of perverse sheaves which are pure of weight zero.

\subsubsection{Modified Langlands dual group}\label{modified dual}
The global section functor $h^{\Tt}=\upH^{*}(\Gr,-):\Sat^{\Tt}\to \Vect_{\Ql}$ is a fiber functor, which gives the Tannakian group $\hatG^{\Tt}:=\Aut^{\otimes}(h^{\Tt})$ so that we have a tensor equivalence
\begin{equation*}
\Sat^{\Tt}\cong\Rep(\hatG^{\Tt},\Ql).
\end{equation*}
For $V\in\Rep(\hatG^{\Tt})$, we shall denote the corresponding object in $\Sat^{\Tt}$ by $\IC_{V}$.

By Tannakian duality,  the diagram \eqref{Sat Tt} gives a diagram of reductive groups over $\Ql$
\begin{equation*}
\xymatrix{\hatG^{\ev} & & \Gm^{\Tt} \\
& \hatG^{\Tt}\ar[ur]^{\Tt}\ar[ul]\\
\hatG\ar[ur]^{\iota}\ar[uu]^{\pi} & & \Gm^{\Wt}\ar[ul]^{\Wt}\ar[uu]^{[-2]}}
\end{equation*}
where the two diagonal sequences are short exact. The reductive group $\hatG^{\ev}$ is defined as the Tannakian dual of $\Sat^{w=0}$, and $\Gm^{\Wt}$ is a one-dimensional torus.

\subsubsection{Irreducible objects}\label{irr} Let $\l$ be a dominant coweight of $T$,  then   $\IC_{\l}$ is a complex pure of weight $\jiao{2\rho,\l}$. The corresponding object in $\hatG^{\Tt}$ is an irreducible representation $V^{\Tt}_{\l}$ of $\hatG^{\Tt}$. What is this representation? Its restriction to $\hatG$ is the irreducible representation of highest weight $\l$. Since $\IC_{\l}$ is pure of weight $\jiao{2\rho,\l}$, $\Gm^{\Wt}$ acts on $V^{\Tt}_{\l}$ by scalar multiplication via the $\jiao{2\rho,\l}\nth$ power of the tautological character. If we look at the action of $\hatG\times\Gm^{\Wt}$ on $V^{\Tt}_{\l}$, then the element $((-1)^{2\rho},-1)\in\hatG\times\Gm^{\Wt}$ acts trivially. Here $2\rho$ is viewed as a coweight of $\hatG$. Note that $(-1)^{\2rho}$ always lies in the center of $\hatG$.

We use $\Ql(n)$ to denote the one-dimensional representation of $\hatG^{\Tt}$ given by the $n\nth$ power of the character $\Tt:\hatG^{\Tt}\to\Gm^{\Tt}$. We also use $V(n)$ to mean $V\otimes\Ql(n)$, as we do for sheaves.

Since every irreducible object in $\Sat^{\Tt}$ is of the form $\IC_{\l}(n)$ for some  dominant coweight $\l$ and $n\in\ZZ$, every irreducible representation of $\hatG^{\Tt}$ is of the form $V^{\Tt}_{\l}(n)$. The action of $\hatG$ on $V^{\Tt}_{\l}(n)$ is still the same as its action on $V_{\l}$, and the action of $\Gm^{\Wt}$ on it is through the $(\jiao{2\rho,\l}-2n)\nth$ power of the tautological character.

\begin{lemma} The map $\Wt:\Gm^{\Wt}\to \hatG^{\Tt}$ is central. The homomorphism $\iota\times\Wt:\hatG\times\Gm^{\Wt}\to \hatG^{\Tt}$ is an isogeny whose kernel is isomorphic to $\mu_{2}$ generated by $((-1)^{2\rho}, -1)$ (which is central). 
\end{lemma}
\begin{proof}
Every irreducible object in $\Sat^{\Tt}$ is pure of some weight, hence $\Gm^{\Wt}$ acts on the corresponding representation by scalars. Therefore $\Gm^{\Wt}$ is central.

Let $\hatG_{1}:=\hatG\times\Gm^{\Wt}/\mu_{2}$ where the $\mu_{2}$ is generated by $((-1)^{2\rho}, -1)$. 
We have seen from the discussion in \S\ref{irr} that the action of $\hatG\times\Gm^{\Wt}$ on every irreducible representations of $\hatG^{\Tt}$ factors through $\hatG_{1}$. Therefore $\iota\times\Wt$ factors through a homomorphism $\phi:\hatG_{1}\to\hatG^{\Tt}$. We then have a commutative diagram
\begin{equation*}
\xymatrix{1\ar[r] & \hatG\ar@{=}[d]\ar[r] & \hatG_{1}\ar[d]^{\phi}\ar[r] & \Gm^{\Wt}/\mu_{2}\ar[d]^{\wr}\ar[r] & 1\\
1\ar[r] & \hatG\ar[r] & \hatG^{\Tt}\ar[r]^{\Tt} & \Gm^{\Tt}\ar[r] & 1}
\end{equation*}
where the right-most map is the isomorphism induced from $[-2]:\Gm^{\Wt}\to\Gm^{\Tt}$. Both the upper and lower sequences are exact, we conclude that $\phi$ is an isomorphism.
\end{proof}

\begin{defn}\label{def:pure} An object  $V\in\Rep(\hatG^{\Tt})$ is {\em pure of weight $w$}, if the action of the weight torus $\Gm^{\Wt}$ on $V$ (via $\Gm^{\Wt}\xrightarrow{\Wt}\hatG^{\Tt}\to\GL(V)$) is by the $w\nth$ power (as scalar multiplication).
\end{defn}
If $V\in \Rep(\hatG^{\Tt})$ is pure of weight $w$, the corresponding perverse sheaf $\IC_{V}$ is pure of the same weight $w$.

The following two examples will be important for our study of classical Kloosterman sums.
\begin{exam}\label{ex:n odd} Let $n\geq2$ be an integer and $G=\PGL_{n}$. We have $\hatG=\SL_{n}$. In this case, $(-1)^{2\rho}=(-1)^{n-1}I_{n}\in\hatG$. We have
\begin{equation*}
\hatG^{\Tt}\cong\begin{cases}\SL_{n}\times (\Gm^{\Wt}/\mu_{2}) & \mbox{ if $n$ is odd;}\\
(\SL_{n}\times\Gm^{\Wt})/\mu^{\Delta}_{2} & \mbox{ if $n$ is even.}\end{cases}
\end{equation*}
We define a ``standard representation'' of dimension $n$ and weight $n-1$:
\begin{eqnarray*}
\St:\hatG^{\Tt}&\to& \GL_{n}\\
(g,t)\in\SL_{n}\times \Gm^{\Wt} &\mapsto& t^{n-1}g\in\GL_{n}.
\end{eqnarray*}
The formula we gave above is a homomorphism from  $\SL_{n}\times \Gm^{\Wt}$ to $\GL_{n}$, and it is easy to check that it factors through $\hatG^{\Tt}$.
\end{exam}

\begin{exam}\label{ex:n even} Let $n\geq2$ be an even integer and $G=\SO_{n+1}$. We have $\hatG=\Sp_{n}$. In this case, $(-1)^{2\rho}=-I_{n}\in\hatG$ and $\hatG^{\Tt}\cong(\Sp_{n}\times \Gm^{\Wt})/\mu_{2}^{\Delta}\cong\GSp_{n}$. We also define a ``standard representation'' of dimension $n$ and weight $n-1$
\begin{eqnarray*}
\St:\hatG^{\Tt}&\to& \GSp_{n}\subset\GL_{n}\\
(g,t)\in\Sp_{n}\times \Gm^{\Wt} &\mapsto& t^{n-1}g\in\GSp_{n}.
\end{eqnarray*}
\end{exam}

\subsection{Construction of Kloosterman sheaves}\label{ss:Kl} In this subsection, the base field $k$ is a finite field of characteristic $p>0$. We fix a nontrivial additive character $\psi:\FF_{p}\to\Ql(\mu_{p})^{\times}$, which gives rise to an Artin-Schreier sheaf $\AS_{\psi}$ on $\AA^{1}_{k}$.

We recall the construction of the $\hatG$-Kloosterman local system on $\Gm$ in \cite{HNY}. Here we will actually give an enhancement of it to a $\hatG^{\Tt}$-local system.  

Let $\pi:\GR\to\Gm$ be the Beilinson-Drinfeld affine Grassmannian over the curve $\Gm$ (\cite[Remark 2.8(1)]{HNY}). For $V\in\Rep(\hatG^{\Tt})$, we denote by $\IC_{V,\GR}$ the ``spread-out'' of $\IC_{V}$ on $\GR$. In our case, there is a one-dimensional torus $\Grot$ acting on $\Gm$ by multiplication. This action induces an action of $\Grot$ on $\GR$ such that $\pi$ is $\Grot$-equivariant. Using the $\Grot$-action, the fibration $\pi$ can be trivialized, and we may identify $\GR$ with $\Gr\times\Gm$. Let $\pi_{\Gr}:\GR\to\Gr$ be the projection. Then $\IC_{V,\GR}=\pi^{*}_{\Gr}\IC_{V}$. Note that $\IC_{V,\GR}[1]$ is a perverse sheaf.  

In \cite[\S5.2]{HNY} we defined an open sub ind-scheme $\GRc\subset\GR$ with a morphism $f:\GRc\to\Ga^{r+1}$, where $r$ is the semisimple rank of $G$.
\begin{equation}\label{GR}
\xymatrix{& \GRc\ar[dl]_{\pi^{\circ}}\ar[dr]_{f}\ar[drr]^{F}\\ \Gm & & \Ga^{r+1} \ar[r]^{\sigma}& \AA^1} 
\end{equation}
Let $\IC_{V,\GRc}$ be the restriction of $\IC_{V,\GR}$ to $\GRc$.
 
One of the main results of \cite{HNY} can be stated as follows.
\begin{theorem}[Heinloth-Ng\^o-Yun {\cite[Theorem 1(1)]{HNY}}]\label{th:geom Kl} 
For $V\in\Rep(\hatG^{\Tt})$, we have the following isomorphism in $D^b(\Gm,\Ql(\mu_p))$
\begin{equation}\label{clean}
\pic_{!}(\IC_{V,\GRc}\otimes F^*\AS_\psi)\cong\pic_{*}(\IC_{V,\GRc}\otimes F^*\AS_\psi).
\end{equation}
given by the natural transformation $\pic_{!}\to\pic_{*}$. (The tensor product above is taken over $\Ql$). Let $\Kl^{V}_{\hatG}$ be any one of the two complexes above. Then $\Kl^{V}_{\hatG}$ is a local system on $\Gm$. The assignment $V\mapsto\Kl^{V}_{\hatG}$ gives a tensor functor
\begin{equation*}
\Kl_{\hatG}: \Rep_{\Ql(\mu_{p})}(\hatG^{\Tt})\to\Loc_{\Ql(\mu_{p})}(\Gm).
\end{equation*}
In other words, we get a $\hatG^{\Tt}(\Ql(\mu_{p}))$-local system $\Kl_{\hatG}$ on $\Gm$.
\end{theorem}
Note that the only difference between the above theorem and  \cite[Theorem 1]{HNY} is that we have replaced the usual Langlands dual group $\hatG$ by Deligne's modified Langlands dual group $\hatG^{\Tt}$. The proof in \cite{HNY} works without change. When $\hatG$ is clear from the context, we simply write $\Kl^{V}$ for $\Kl^{V}_{\hatG}$.

\begin{remark}\label{r:wt} The Kloosterman sheaf $\Kl_{\hatG}$ gives rise to a monodromy representation, well-defined up to conjugacy 
\begin{equation*}
\rho^{\Tt}_{\hatG}:\pi_{1}(\Gm/k)\to \hatG^{\Tt}(\Ql(\mu_{p})).
\end{equation*}
The monodromy representation $\rho^{\Tt}_{\hatG}$ is compatible with the Tate torus in the following sense. The composition
\begin{equation*}
\pi_{1}(\Gm/k)\xrightarrow{\rho_{\hatG}^{\Tt}} \hatG^{\Tt}(\Ql(\mu_{p}))\xrightarrow{\Tt}\Gm^{\Tt}(\Ql(\mu_{p}))=\Ql(\mu_{p})^{\times}\end{equation*}
factors through $\pi_{1}(\Gm/k)\to\Gk$ and is given by the $\ell$-adic cyclotomic character $\chi_{\cyc}:\Gk\to \ZZ_{\ell}^{\times}=\Aut(\ZZ_{\ell}(1))$, where $\ZZ_{\ell}(1)=\varprojlim_{n}\mu_{\ell^{n}}(\kbar)$.

On the other hand, if an object $V\in\Rep(\hatG^{\Tt})$ pure of weight $w$, the corresponding local system $\Kl^{V}$  is also pure of weight $w$.
\end{remark}

\subsection{Geometric monodromy of Kloosterman sheaves}\label{ss:monodromy}

\subsubsection{Global monodromy} We also consider the restriction of $\rho^{\Tt}_{\hatG}$ to the geometric fundamental group
\begin{equation*}
\rho^{\geom}_{\hatG}:\pi_{1}(\Gm/\kbar)\to \hatG^{\Tt}(\Ql(\mu_{p})).
\end{equation*}
Let $\hatG^{\geom}$ be the Zariski closure of the image of $\rho^{\geom}_{\hatG}$. Since the composition $\pi_{1}(\Gm/\kbar)\to\hatG^{\Tt}(\Ql(\mu_{p}))\to \Gm^{\Tt}(\Ql(\mu_{p}))$ factors through $\Gk$ (Remark \ref{r:wt}), the image of $\rho^{\geom}_{\hatG}$ lies in $\hatG=\ker(\hatG^{\Tt}\to \Gm^{\Tt})$, and hence $\hatG^{\geom}\subset\hatG$.

\begin{theorem}[Heinloth-Ng\^o-Yun {\cite[Theorem 3]{HNY}}]\label{th:global mono} The subgroup $\hatG^{\geom}$ of $\hatG$ is a connected and simply-connected group of types given by the following table (with restrictions on $p$ stated in the third column)
\begin{center}
\begin{tabular}{|l|l|l|}
\hline
$\hatG$ & $\hatG^{\geom}$ & \textup{condition}\\\hline
$A_{2n}$ & $A_{2n}$ & $p>2$\\
$A_{2n-1}, C_n$ &  $C_n$ & $p>2$\\ 
$B_n, D_{n+1}$ $(n\geq4)$ & $B_n$ & $p>2$ \\
$E_7$ & $E_7$ & $p>2$\\
$E_8$ & $E_8$ & $p>2$\\
$E_6, F_4$ & $F_4$ & $p>2$\\
$B_3,D_4, G_2$ & $G_2$ & $p>3$\\
\hline
\end{tabular}
\end{center}
We have $\hatG^{\geom}=\hatG^{\Out(\hatG),\circ}$  ($\Out(\hatG)$ acting on $\hatG$ as pinned automorphisms) unless $\hatG$ is of type $A_{2n}$ or $B_{3}$. 
\end{theorem}

Let $\calI_0$ and $\calI_\infty$ be the inertia groups at $0$ and $\infty$. These are subgroups of the geometric fundamental group $\pi_{1}(\Gm/\kbar)$, well-defined up to conjugacy. We shall describe the image of $\calI_{0}$ and $\calI_{\infty}$ under the monodromy representation $\rho^{\geom}_{\hatG}:\pi_{1}(\Gm/\kbar)\to\hatG(\Ql(\mu_{p}))$.

\subsubsection{Local monodromy at $0$}\label{loc 0} By \cite[Theorem 1(2)]{HNY}, we know that $\rho^{\geom}_{\hatG}$ is trivial on the wild inertia $\calI^{w}_{0}$ (i.e., $\Kl_{\hatG}$ is tame at $0$) and sends a topological generator of the tame inertia $\calI^{t}_{0}$ to a regular unipotent element in $\hatG$.

\subsubsection{Local monodromy at $\infty$ for large $p$}\label{loc infty} The description of the local monodromy of $\Kl_{\hatG}$ at $\infty$ for large $p$ is obtained in \cite{HNY} based on the work of Gross and Reeder on simple wild parameters \cite{GR}. We review the results here.

Assume $p$ does not divide $\#W$ ($W$ is also identified with the Weyl group of $\hatG$). Let $\hatG^{\ad}$ be the adjoint form of $\hatG$, which is also the adjoint form of $\hatG^{\Tt}$.  We first describe the image of $\calI_{\infty}$ in $\hatG^{\ad}$.

Let $\Dbar_{0}=\rho^{\geom}_{\hatG^{\ad}}(\calI_{\infty})$ and $\Dbar_{1}=\rho^{\geom}_{\hatG^{\ad}}(\calI^{w}_{\infty})$ be the image of the inertia and the wild inertia in $\hatG^{\ad}$. Combining \cite[Theorem 2, Corollary 2.15]{HNY} and \cite[Proposition 5.6]{GR}, we get a complete description of $\Dbar_{0}$ and $\Dbar_{1}$ as follows.

There is a unique maximal torus $\hatT^{\ad}$ of $\hatG^{\ad}$ such that $\Dbar_{1}\subset\hatT^{\ad}[p]$. The group $\Dbar_{0}$ lies in the normalizer $N(\hatT^{\ad})$ of $\hatT^{\ad}$. The quotient $\Dbar_{0}/\Dbar_{1}$ is cyclic of order $h$ (the Coxeter number of $G$), and any generator of $\Dbar_{0}/\Dbar_{1}$ maps to a Coxeter element  $\Cox$ in $W$ under the map $\Dbar_{0}/\Dbar_{1}\to N(\hatT^{\ad})/\hatT^{\ad}=W$.

Moreover,  $\Dbar_1\subset\hatT^{\ad}[p]$ is isomorphic to the additive group of some finite field $\FF_{q}=\FF_{p}[\zeta]$ where $\zeta$ is a primitive $h\nth$ root of unity in $\Fpbar$. The action of a generator of $\Dbar_{0}/\Dbar_{1}$ on $\Dbar_{1}\cong\FF_{p}[\zeta]$ is by multiplication by $\zeta$.

We then describe the image of $\calI_{\infty}$ in $\hatG$. Let $D_{0}=\rho^{\geom}_{\hatG}(\calI_{\infty})$ and $D_{1}=\rho^{\geom}_{\hatG}(\calI^{w}_{\infty})$, which is a $p$-group mapping to $\Dbar_{1}$ in $\hatG^{\ad}$. Since $p$ is prime to $\#Z\hatG$, there is a unique $p$-group $D_{1}\subset\hatG$ mapping onto $\Dbar_{1}$, and $D_{1}$ lies in $\hatT[p]$ ($\hatT$ is the preimage of $\hatT^{\ad}$ in $\hatG$). Again $D_{0}\subset N(\hatT)$ because $N(\hatT)$ is the preimage of $N(\hatT^{\ad})$ in $\hatG$. Since $D_{0}/D_{1}$ is cyclic, $D_{0}$ is of the form $D_{1}\rtimes\jiao{\dCox}$ where $\dCox\in N(\hatT)$ is a lifting of a Coxeter element $\Cox\in W$. Any two liftings of $\Cox$ to $N(\hatT)$ are conjugate under $\hatT$.  Note that when $\hatG$ is not adjoint, the order of $\dCox$ is a multiple $eh$ of $h$, but may not be $h$.

\begin{exam} Let $G=\PGL_{n}$, $\hatG=\SL_{n}$. The Coxeter elements in the Weyl group $W=S_{n}$ are cyclic permutations of length $n$. When $n$ is odd, any lifting of such an element to $N_{\hatG}(\hatT)$ still has order $n$. However, when $n$ is even, any lifting of a Coxeter element in $S_{n}$ to $N_{\hatG}(\hatT)$ has order $2n$.

When $n$ is even and $G=\SO_{n+1}$, $\hatG=\Sp_{n}$. Then we are in the same situation as $\hatG=\SL_{n}$: the Coxeter element has order $n$ but any lifting of it to $N_{\hatG}(\hatT)$ has order $2n$.
\end{exam}

We have the lower numbering filtration of the image of $\calI_{\infty}$ under $\rho_{\hatG}^{\geom}$ (see \cite[Ch IV, Proposition 1]{local}): $D_{0}\rhd D_{1}\rhd D_{2}\rhd\cdots$.

\begin{lemma}\label{l:Swan} Let $V\in\Rep(\hatG^{\Tt})$. Assume $p\nmid\#W$. Then 
\begin{equation*}
\Swan(V)=\sum_{j\geq1}\frac{\dim V-\dim V^{D_j}}{[D_0:D_j]}=\frac{\dim V-\dim V^{D_{1}}}{h}.
\end{equation*}
\end{lemma}
\begin{proof}
The first equality follows from the definition of the Swan conductor (see \cite[Ch VI, Proposition 2]{local} for the closely-related Artin conductor, see also \cite[\S2]{GR}).  Now let $eh$ be the order of $\wt{\Cox}$, then $D_{0}/D_{1}\cong\ZZ/eh\ZZ$. The lower numbering filtration of $D_{0}$ takes the form
\begin{equation}\label{Die}
D_{0}\rhd D_{1}=D_{2}=\cdots=D_{e} \rhd D_{e+1}=\{1\}.
\end{equation}
In fact, let $K=\kbar((t^{-1}))$ and let $L/K$ be Galois extension with Galois group isomorphic to $D_{0}=\rho^{\geom}_{\hatG}(\calI_{\infty})$. The quotient $\Dbar_{0}$ corresponds to a subextension $L'\subset L$. Let $K_{t}$ (resp. $K_{t}'$) be the maximal tamely ramified extension of $K$ in $L$ (resp. $L'$). Since $D_{1}=\Dbar_{1}$, $L$ is the compositum $L'K_{t}$ with $L'\cap K_{t}=K_{t}'$. Pick a generating element $x\in \calO_{L'}$ over $\calO_{K'_{t}}$. Then it is also a generating element of $\calO_{L}$ over $\calO_{K_{t}}$. For any $g\in D_{1}$, we then have $\val_{L}(g(x)-x)=e\val_{L'}(g(x)-x)$. Therefore the jumps of the lower number filtration $\{D_{i}\}$ are exactly $e$ times the jumps of $\{\Dbar_{i}\}$. Since we have $\Dbar_{2}=\{1\}$, the sequence $\{D_{i}\}$ must look like \eqref{Die}.

We then have
\begin{equation*}
\sum_{j\geq1}\frac{\dim V-\dim V^{D_j}}{[D_0:D_j]}=\sum_{j=1}^{e}\frac{\dim V-\dim V^{D_1}}{eh}=\frac{\dim V-\dim V^{D_1}}{h}.
\end{equation*}
\end{proof}

\subsection{The classical Kloosterman sheaf}\label{ss:classical} 
When $G=\PGL_{n}$, let us consider the diagram
\begin{equation}\label{Kloo}
\xymatrix{& \Gm^n\ar[dl]_{\pi}\ar[dr]_{f}\ar[drr]^{F}\\ \Gm & & \Ga^n \ar[r]^{\sigma}& \AA^1}
\end{equation}
Here the morphism $\pi$ is given by taking the product $\pi(x_{1},\cdots,x_{n})=x_{1}\cdots x_{n}$; $f$ is the natural inclusion; $\sigma$ is given by taking the sum $\sigma(x_{1},\cdots,x_{n})=x_{1}+\cdots+ x_{n}$; finally $F=\sigma\circ f$. Deligne defined
\begin{equation*}
\Kl_n:=m_!F^*\AS_\psi[n-1].
\end{equation*}
It is known \cite[Th\'eor\`eme 7.8]{Del} that $\Kl_n$ is concentrated in degree zero, and is a local system on $\GG_{m,\FF_p}$ of rank $n$. Moreover, $\Kl_{n}$ is pure of weight $n-1$. 
% det of \Kl_{n}
\subsubsection{Determinant}\label{sss:det}
For the one-dimensional representation $\det$ of $\GL_{n}$, $\Kl^{\det}_{n}=\det(\Kl_{n})$. Deligne \cite[(7.15.2)]{Del} has determined that
\begin{equation*}
\det\Kl_{n}=\Ql\left(-\frac{n(n-1)}{2}\right).
\end{equation*}
as $\Frob$-modules.

% symplectic pairing for \Kl_{n}, n even
\subsubsection{Symplectic pairing}
When $n$ is even, there is a symplectic pairing (see \cite[\S7.5]{Del})
\begin{equation*}
\Kl_{n}\otimes\Kl_{n}\to\Ql(-n+1).
\end{equation*}

The diagram \eqref{Kloo} is very similar to \eqref{GR}. In fact, when we work with $G=\PGL_{n}$ and $V$ the standard representation of $\hatG^{\Tt}$ (see Example \ref{ex:n odd}), the sheaf $\IC_{V}$ is the shifted constant sheaf $\Ql[n-1]$ supported on a Schubert variety in $\Gr$ which isomorphic to $\PP^{n-1}$.  The ``spread-out'' of $\IC_{V}$ to $\GR$ is the shifted constant sheaf $\Ql[n-1]$ on $\Gm\times\PP^{n-1}$. The intersection of the open subset $\GRc$ with $\Gm\times\PP^{n-1}$ is isomorphic to $\Gm^{n}$. Thus we recover \eqref{Kloo} as part of \eqref{GR}, and $\Kl_{n}$ is the same as $\Kl^{\St}_{\SL_{n}}$.

%\begin{defn} Let $n\geq2$ be an integer. We define
%\begin{equation*}
%G_{n}=\begin{cases}\PGL_{n} & \mbox{ if $n$ is odd;}\\
%\SO_{n+1} & \mbox{ if $n$ is even.}\end{cases}
%\end{equation*}
%\end{defn}
%We define $\hatG^{\Tt}_{n}$ according to the recipe in \S\ref{Sat}. According to Example \ref{ex:n odd} and Example \ref{ex:n even}, we have
%\begin{equation*}
%\hatG^{\Tt}_{n}\cong\begin{cases}\SL_{n}\times(\Gm^{\Wt}/\mu_{2}) & \mbox{ if $n$ is odd;}\\
%\GSp_{n} & \mbox{ if $n$ is even.}\end{cases}
%\end{equation*}
%In both cases, we have defined a standard representation $\St:\hatG^{\Tt}_{n}\to\GL_{n}$ which is pure of weight $n-1$.

The next proposition gives the precise sense in which $\Kl_{\hatG}$ generalizes Deligne's Kloosterman sheaves $\Kl_{n}$.
\begin{prop}\label{p:relation to Kln}
\begin{enumerate}
\item For $n\geq2$, we have an isomorphism of local systems on $\Gm$
\begin{equation*}
\Kl^{\St}_{\SL_{n}}\cong\Kl_{n}.
\end{equation*}
Here $\St$ is the standard representation of $\SL_{n}^{\Tt}$ given in Example \ref{ex:n odd}. 
\item When $n$ is even, there exists $x\in\Gm(\FF_{p})$ and a character $\ep:\Gk\to\{\pm1\}$ such that
\begin{equation}\label{shift Kl}
\Kl^{\St}_{\Sp_{n}}\cong m_{x}^{*}\Kl_{n}\otimes\Ql(\ep).
\end{equation}
Here $m_{x}:\Gm\to \Gm$ is the multiplication by $x$, $\Ql(\ep)$ is the rank one local system on $\Spec k$ obtained from the character $\ep$, and $\St$ is the standard representation of $\Sp_{n}^{\Tt}$ given in Example \ref{ex:n even}. 
\end{enumerate}
\end{prop}
\begin{proof}
(1) is proved in \cite[Proposition 3.4]{HNY}. When $n$ is even, according to the description of  local monodromy of $\Kl_{\Sp_{n}}$ in \S\ref{loc 0} and \S\ref{loc infty}, the local monodromy of $\Kl^{\St}_{\Sp_{n}}$ is unipotent with a single Jordan block; at $\infty$, $\Kl^{\St}_{\hatG_{n}}$ is totally wild with $\Swan(\Kl^{\St}_{\hatG_{n}})=1$. Applying Theorem \cite[8.7.1]{Katz} we conclude that there exists $x\in\Gm(\FF_{p})$ and a continuous character $\ep:\Gk\to\Qlbar^{\times}$ such that \eqref{shift Kl} holds. Since both $\Kl^{\St}_{\hatG_{n}}$ and $\Kl_{n}$ admit a symplectic autoduality with values in $\Ql(-n+1)$, $\ep$ must be quadratic.
\end{proof}

The monodromy representation of $\Kl_{n}$ reads
\begin{equation*}
\rho_{n}:\pi_{1}(\Gm/\FF_{p})\to \SL^{\Tt}_{n}(\Ql(\mu_{p})). 
\end{equation*}
which is well-defined up to conjugacy. Let $\rho_{n}^{\geom}$ be the restriction of $\rho_{n}$ to the geometric fundamental group $\pi_{1}(\Gm/\Fpbar)$, and let $\hatG^{\geom}_{n}$ be the Zariski closure of the image of $\rho^{\geom}_{n}$.

Katz gave a complete description of $\hatG^{\geom}_{n}$.
\begin{theorem}[Katz {\cite[Main Theorem 11.1]{Katz}}]\label{th:global mono Katz} We have
\begin{equation*}
\hatG^{\geom}_{n}=\begin{cases}
\Sp_{n} &  n \textup{ even};\\
\SL_{n} & n \textup{ odd, } p>2;\\
\SO_{n} & n\neq 7 \textup{ odd, } p=2;\\
G_{2} & n=7, p=2. \end{cases}
\end{equation*}
\end{theorem}

\subsection{Moments}\label{ss:moments} In this subsection we define moments of Kloosterman sheaves and study their basic properties.  Let $k$ be a finite field.
\begin{defn}\label{def:moments} For the Kloosterman sheaf $\Kl_{\hatG}$ and a representation $V\in\Rep(\hatG^{\Tt})$, we define three versions of {\em the $V$-moment} of $\Kl_{\hatG}$:
\begin{eqnarray*}
M^V_{!}&:=&\bR\Gamma_{c}(\Gm,\Kl^V)[1];\\
M^V_{*}&:=&\bR\Gamma(\Gm,\Kl^V)[1];\\
M^V_{!*}&:=&\Im(\upH^{0}M^V_!\to \upH^{0}M^V_*)=\Im(\cohoc{1}{\Gm,\Kl^{V}}\to \cohog{1}{\Gm,\Kl^{V}}).
\end{eqnarray*}
Each of these are objects in $D^{b}_{c}(\Spec k)$. When we need to emphasize their dependence on $k$, we denote them by $M^V_{!,k}, M^V_{*,k}$ and $M^V_{!*,k}$, respectively.
\end{defn}

\begin{remark}\label{r:MV degree zero}
The complex $M^{V}_{!}$ lies in cohomological degrees $0$ and $1$, and $\upH^{1}M^{V}_{!}=V_{\pi_{1}(\Gm/\kbar)}(-1)=V_{\hatG^{\geom}}(-1)$. Therefore, if $V^{\hatG^{\geom}}=0$ (which is equivalently to $V_{\hatG^{\geom}}=0$ since $\hatG^{\geom}$ is reductive) then $M^{V}_{!}$ is concentrated in degree zero (i.e., is a plain $\Gk$-module). Dually, $M^{V}_{*}$ lies in cohomological degrees $-1$ and $0$, and is concentrated in degree zero if $V^{\hatG^{\geom}}=0$. 
\end{remark}

\subsubsection{Basic properties of moments}
Let $j:\Gm\incl\PP^1$ be the open inclusion, and let $i_{0}$ and $i_{\infty}$ are the inclusions of the point $0$ and $\infty$ into $\PP^{1}$. The distinguished triangle
\begin{equation*}
j_!\Kl^V\to j_*\Kl^V\to i_{0,*}i_{0}^{*}j_{*}\Kl^{V}\oplus i_{\infty,*}i_{\infty}^{*}j_{*}\Kl^{V}\to j_!\Kl^V[1]
\end{equation*}
gives a long exact sequence
\begin{eqnarray}\label{!* long}
0\to \upH^{-1}M^{V}_{*} \to \upH^{0}i_{0}^{*}j_{*}\Kl^{V}\oplus \upH^{0}i_{\infty}^{*}j_{*}\Kl^{V}\to \upH^{0}M^V_{!}\to \upH^{0}M^V_{*}\to\\
\notag
\to\upH^{1}i_{0}^{*}j_{*}\Kl^{V}\oplus \upH^{1}i_{\infty}^{*}j_{*}\Kl^{V}\to\upH^{1}M^{V}_{!}\to0.
\end{eqnarray}
We have can rewrite the exact sequence \eqref{!* long} as
\begin{eqnarray}\label{!* short}
0\to V^{\hatG^{\geom}}\to V^{\calI_{0}}\oplus V^{\calI_{\infty}}\to\upH^{0} M^V_{!}\to \upH^{0}M^V_{*}\to\\
\notag \to V_{\calI_{0}}(-1)\oplus V_{\calI_{\infty}}(-1)\to V_{\hatG^{\geom}}(-1)\to0.
\end{eqnarray}

\begin{lemma}\label{l:MV pure} Let $V\in\Rep(\hatG^{\Tt})$ be a pure object of weight $w$. As $\Gk$-modules, $\upH^{0}M^{V}_{!}$ is of weights $\leq w+1$ and $\upH^{0}M^{V}_{*}$ is of weights $\geq w+1$. Moreover,
\begin{equation*}
M^{V}_{!*}=\cohog{0}{\PP^{1},j_{!*}\Kl^{V}[1]}=\Gr^{W}_{w+1}\upH^{0}M^{V}_{!}=\Gr^{W}_{w+1}\upH^{0}M^{V}_{*}
\end{equation*}
which is pure of weight $w+1$. Here $j_{!*}\Kl^{V}[1]$ is the middle extension of the perverse sheaf $\Kl^{V}[1]$ from $\Gm$ to $\PP^{1}$.
\end{lemma}
\begin{proof}
The local system $\Kl^{V}$ is pure of weight $w$ by Remark \ref{r:wt}. By \cite[Th\'eor\`eme 3.3.1]{WeilII}, $\upH^{0}M^{V}_{!}=\cohoc{1}{\Gm,\Kl^{V}}$ has weights $\leq w+1$. A dual argument shows that $\upH^{0}M^{V}_{*}$ is of weights $\geq w+1$. Since $M^{V}_{!*}=\Im(\upH^{0}M^{V}_{!}\to \upH^{0}M^{V}_{*})$, it is pure of weight $w+1$.

By \eqref{!* short}, the kernel of the natural map $\upH^{0}M^{V}_{!}\to \upH^{0}M^{V}_{*}$ is a quotient of $\upH^{0}i_{0}^{*}j_{*}\Kl^{V}\oplus \upH^{0}i_{\infty}^{*}j_{*}\Kl^{V}$, which has weight $\leq w$ because $\upH^{0}i_{0}^{*}j_{*}\Kl^{V}$ and $\upH^{0}i_{\infty}^{*}j_{*}\Kl^{V}$ are stalks of the perverse sheaf $j_{!*}\Kl^{V}$ which is pure of weight $w$. Therefore $M^{V}_{!*}=\Gr^{W}_{w+1}\upH^{0}M^{V}_{!}$. A dual argument shows that $M^{V}_{!*}=\Gr^{W}_{w+1}\upH^{0}M^{V}_{*}$.
\end{proof}

Using the description of local monodromy of the Kloosterman sheaves, we can give a formula for the dimension of $M^V_{!}, M^V_{*}$ and $M^V_{!*}$.

\begin{lemma}\label{Sw general} Let $N\in \Lie\hatG$ be a regular nilpotent element. Then we have
\begin{eqnarray}
\label{dim MV!}\dim M^V_{!}&=&\dim M^V_{*}=\Swan(V);\\
\label{dim MV!*}\dim M^V_{!*}&=&\Swan(V)-\dim V^{N}-\dim V^{\calI_\infty}+\dim V^{\hatG^{\geom}}.
\end{eqnarray}
When $p\nmid \#W$, we have
\begin{equation*}
\label{dim MVmore}\dim M^V_{!*}=\frac{\dim V-\dim V^{D_{1}}}{h}-\dim V^{N}-\dim V^{D_{1},\dCox}+\dim V^{\hatG^{\geom}}.
\end{equation*}
\end{lemma}
\begin{proof}
\eqref{dim MV!} follows from the Grothendieck-Ogg-Shafarevich formula. \eqref{dim MV!*} follows from \eqref{dim MV!} and the sequence \eqref{!* short}. Note that $V^{\calI_{0}}=V^{N}$ by the result stated in \S\ref{loc 0}. Finally \eqref{dim MVmore} follows from \eqref{dim MV!*} and Lemma \ref{l:Swan}.
\end{proof}

\subsubsection{Independence of dimension of moments for large $p$} We are going to show that the dimensions of $M^{V}_{\FF_{p},?}$ ($?=!, *$ and $!*$) are independent of $p$ when $p$ is sufficiently large. 

Let $\xcoch(\hatT)_{\QQ,\prim}\subset\xcoch(\hatT)_{\QQ}$ be the subspace where $\Cox$ acts via a primitive $h\nth$ root of unity. Then $\xcoch(\hatT)_{\QQ,\prim}$ is a simple $\QQ[\Cox]$-module isomorphic to $\QQ[x]/(\phi_{h}(x))$, the cyclotomic field generated by $h\nth$ roots of unity. Let $\xcoch(\hatT)_{\prim}=\xcoch(\hatT)_{\QQ,\prim}\cap\xcoch(\hatT)$. Let $\xch(\hatT)^{\bot}_{\prim}\subset\xch(\hatT)$ be the annihilator of $\xcoch(\hatT)_{\prim}$.

\begin{lemma}\label{l:const dim} Assume $p\nmid\#W$ so that we have $D_{1}\subset\hatT[p]$ for some maximal torus $\hatT\subset \hatG$ (see \S\ref{loc infty}). Let $V\in\Rep(\hatG)$ and let $V(\l)$ be the $\l$-weight space of $V$ for $\l\in\xch(\hatT)$. Then for sufficiently large $p$ (depending on $V$), we have 
\begin{equation*}
V^{D_{1}}=\bigoplus_{\l\in\xch(\hatT)^{\bot}_{\prim}}V(\l)
\end{equation*}
In particular, according to Lemma \ref{Sw general}, $\dim M^{V}_{!,\FF_{p}}, \dim M^{V}_{*,\FF_{p}}$ and $\dim M^{V}_{!*,\FF_{p}}$ are constant for $p$ sufficiently large.
\end{lemma}
\begin{proof} For each weight $\l\in\xcoch(\hatT)$, $D_{1}$ acts on $V(\l)$ via the character $\l|_{D_{1}}:D_{1}\to\Ql(\mu_{p})^{\times}$.  Therefore we need to show that $\l|_{D_{1}}$ is trivial for sufficiently large $p$ if and only if $\l\in\xcoch(\hatT)^{\bot}_{\prim}$.

We define $\xcoch(\hatT)_{\FF_{p},\prim}\subset\xcoch(\hatT)_{\FF_{p}}=\xcoch(\hatT)\otimes_{\ZZ}\FF_{p}$ to be the $\FF_{p}[\Cox]$-submodules generated by those on which $\Cox$ acts via a primitive $h\nth$ root of unity (so $\xcoch(\hatT)_{\FF_{p},\prim}$ may not be a simple $\FF_{p}[\Cox]$-module). For sufficiently large $p$, $\xcoch(\hatT)_{\FF_{p},\prim}$ is the reduction mod $p$ of $\xcoch(\hatT)_{\prim}$, i.e., $\xcoch(\hatT)_{\FF_{p},\prim}=\xcoch(\hatT)_{\prim,\FF_{p}}$.

We identify $\hatT[p]$ with $\xcoch(\hatT)_{\FF_{p}}$ hence view $D_{1}$ as a $\FF_{p}[\Cox]$-submodule of $\xcoch(\hatT)_{\FF_{p}}$. By the description of $D_{1}$ in \S\ref{loc infty},  we have $D_{1}\subset \xcoch(\hatT)_{\FF_{p},\prim}$. By the discussion above, for $p$ large, we may view $D_{1}$ as a $\FF_{p}[\Cox]$-submodule of $\xcoch(\hatT)_{\prim,\FF_{p}}$.

If $\l\bot\xcoch(\hatT)_{\prim}$, then $\l$ restricts trivially to $\xcoch(\hatT)_{\prim,\FF_{p}}$. Hence for sufficiently large $p$, $\l\mod p$ is trivial on $\xcoch(\hatT)_{\FF_{p},\prim}$, hence on $D_{1}$.

On the other hand, if $\l$ is nonzero on $\xcoch(\hatT)_{\prim}$, we claim that $\l|_{D_{1}}$ is also nontrivial for sufficiently large $p$. Let $\barl$ be the image of $\barl$ in $\xch(\hatT)_{\prim}=\xcoch(\hatT)^{\vee}_{\prim}$. The vectors $\barl, \Cox(\barl),\cdots,\Cox^{\phi(h)-1}(\barl)$ span the $\QQ$-vector space $\xch(\hatT)_{\QQ,\prim}$. Therefore, for sufficiently large $p$, $\barl\mod p$ generate $\xch(\hatT)_{\FF_{p},\prim}$ as a $\FF_{p}[\Cox]$-module. If $\barl|_{D_{1}}=1$, then $\barl\mod p$ could only generate a proper $\FF_{p}[\Cox]$-submodule of $\xch(\hatT)_{\FF_{p},\prim}$ (contained in the annihilator of $D_{1}$). Therefore $\barl|_{D_{1}}$ is nontrivial.
\end{proof}

\subsection{Moments via homogeneous Fourier transform}\label{ss:moments HF} 
We would like to express the moments $M^{V}_{!}$ using homogeneous Fourier transform, which has the advantage that the new formula makes sense over any base field (i.e., the Artin-Schreier sheaf disappears in the new formula). Basic definitions and properties of the homogeneous Fourier transform are recalled in Appendix \ref{app}, following the original article of Laumon \cite{Laumon}. 

\subsubsection{An alternative description of $\GRc$}\label{alt Gr} Here we give an alternative description of the open part $\GRc$ of the Beilinson-Drinfeld Grassmannian defined in \cite[\S5.2]{HNY}. We would like to show that $\GRc$ is a group ind-scheme over $\Gm$ and the morphism $f:\GRc\to\Ga^{r+1}$ is a group homomorphism.

Fix a pair of opposite Borel subgroups $B,B^{\opp}\subset G$ with intersection $T=B\cap B^{\opp}$ a maximal torus of $G$. Let $LG_{0}$ and $LG_{\infty}$ be the loop group of $G$, with the formal parameter chosen to be the local coordinate $t$ around $0\in\PP^{1}$ and $s=t^{-1}$ around $\infty\in\PP^{1}$. We use $B^{\opp}$ to define an Iwahori subgroup $\bI^{\opp}_{0}\subset LG_{0}$, i.e., $\bI_{0}$ is the preimage of $B^{\opp}$ under the evaluation map $L^{+}G_{0}\to G$. Similarly, we use $B$ to define an Iwahori subgroup $\bI_{\infty}\subset LG_{\infty}$. We also consider the following Moy-Prasad subgroups of $\bI_{\infty}$ (see \cite[\S1.2]{HNY}):
\begin{equation*}
\bI_{\infty}=\bI_{\infty}(0)\rhd\bI_{\infty}(1)\rhd\bI_{\infty}(2).
\end{equation*}
Here $\bI_{\infty}(1)$ is the preimage of $N$ (radical of $B$) under the evaluation map $L^{+}G_{\infty}\to G$, so that $\bI_{\infty}(0)/\bI_{\infty}(1)=T$. On the other hand $\bI_{\infty}(1)/\bI_{\infty}(2)$ can be identified with $(r+1)\nth$ product of the additive group $\Ga$, where $r$ is the rank of $G$;  each copy of $\Ga$ is the root subgroup of $LG_{\infty}$ corresponding to a simple affine root.

Recall in \cite[\S1.4]{HNY} we introduced moduli stacks $\Bun_{G(0,i)}$ for $i=0,1,2$. It classifies $G$-torsors over $\PP^{1}$ with an $\bI^{\opp}_{0}$-level structure at $0$ and an $\bI_{\infty}(i)$-level structure at $\infty$. There is an open point $\star\in\Bun_{G(0,1)}$ (with no automorphisms), which corresponds to the trivial
$G$-torsor over $\PP^{1}$ with the tautological $\bI^{\opp}_{0}$ and $\bI_{\infty}(1)$-level structures. We denote this $G$-torsor with level structures by $\calE^{\star}_{0,1}$. Its image in $\Bun_{G(0,0)}$ will be denoted $\calE^{\star}_{0,0}$, with automorphism group $T$.

The connected components of $\Bun_{G(0,i)}$ are indexed by $\Omega=\xcoch(T)/\ZZ\Phi^{\vee}$ where $\ZZ\Phi^{\vee}$ is the coroot lattice. For $\omega\in\Omega$, we let $\Bun^{\omega}_{G(0,i)}$ be the corresponding component. We may identify $\Omega$ with the quotient $N_{LG_{0}}(\bI^{\opp}_{0})/\bI^{\opp}_{0}$, hence it acts on $\Bun_{G(0,i)}$ by changing the $G$-torsors in the formal neighborhood of  $0$. This action  permutes the components of $\Bun_{G(0,i)}$ simply transitively. Using $\Omega$-translates of $\calE^{\star}_{0,i}$ we get an open point $\omega(\star)\in\Bun^{\omega}_{G(0,i)}$, which corresponds to a $G$-torsor $\calE^{\omega(\star)}_{0,i}$ with level structures at $0$ and $\infty$. The $\Omega$-action only changes a $G$-torsor in the formal neighborhood of  $0$, hence there is a canonical isomorphisms $\calE^{\omega(\star)}_{0,i}|_{\PP^{1}-\{0\}}\isom\calE^{\star}_{0,i}|_{\PP^{1}-\{0\}}$ for all $\omega\in\Omega$.
 
For $x\in \Gm(k)$, and $\omega\in\Omega$, define an ind-scheme
\begin{equation*}
\frG^{\omega}_{x}=\Isom_{\PP^{1}-\{x\}}(\calE^{\omega(\star)}_{0,1},\calE^{\star}_{0,1}).
\end{equation*}
When $\omega=1$, $\frG^{1}_{x}$ is the automorphism group (ind-scheme) of $\calE^{\star}_{0,1}|_{\PP^{1}-\{x\}}$, which acts on each $\frG^{\omega}_{x}$. When $x$ varies over $\Gm$, we get an ind-scheme $\frG^{\omega}$ over $\Gm$, and we define
\begin{equation*}
\frG_{x}:=\bigsqcup_{\omega\in\Omega}\frG^{\omega}_{x}; \hspace{1cm} \frG:=\bigsqcup_{\omega\in\Omega}\frG^{\omega}.
\end{equation*}
We define a group structure on $\frG$ as follows. For $\varphi\in\frG^{\omega}$ and $\varphi'\in\frG^{\omega'}$, we define $\varphi\circ\varphi'$ to be the composition $\calE^{\omega\omega'(\star)}_{0,1}\xrightarrow{\omega(\varphi')}\calE^{\omega(\star)}_{0,1}\xrightarrow{\varphi}\calE^{\star}_{0,1}$. It is easy to check that this defines a group ind-scheme structure on $\frG$ over $\Gm$.

For each $\varphi\in\frG^{\omega}_{x}$, which is an isomorphism $\calE^{\omega(\star)}_{0,1}|_{\PP^{1}-\{x\}}\to\calE^{\star}_{0,1}|_{\PP^{1}-\{x\}}$, we may consider its restriction $\varphi_{\infty}$ to the formal neighborhood of $\infty\in\PP^{1}$. Since $\calE^{\omega(\star)}_{0,1}$ is the same as $\calE^{\star}_{0,1}$  away from $0$, $\varphi_{\infty}$ gives an element in $L^{+}G_{\infty}$. Since $\varphi_{\infty}$ preserves the $\bI_{\infty}(1)$-level structures, $\varphi_{\infty}\in\bI_{\infty}(1)$. This defines a group homomorphism
\begin{equation*}
\ev_{\infty}:\frG\to \bI_{\infty}(1),
\end{equation*}
and hence a group homomorphism by composition
\begin{equation*}
f_{\frG}:\frG\to \bI_{\infty}(1)\to \bI_{\infty}(1)/\bI_{\infty}(2)\cong\Ga^{r+1}.
\end{equation*}

\begin{lemma} There is a canonical isomorphism of ind-schemes over $\Gm$
\begin{equation*}
\iota:\frG\isom\GRc
\end{equation*}
such that $f_{\frG}=f\circ\iota$.
\end{lemma}
\begin{proof}
We first recall the definition of $\GRc$ in \cite{HNY}. We defined the Hecke correspondence (see \cite[\S2.3]{HNY})
\begin{equation*}
\xymatrix{ & \Hk_{G(0,1)}\ar[dl]^{\pr_{1}}\ar[dr]^{\pr_{2}}\ar[rr]^{\pr_{\Gm}} & & \Gm \\
\Bun_{G(0,1)} & & \Bun_{G(0,1)}}
\end{equation*}
In \cite[\S5.2]{HNY}, the Beilinson-Drinfeld Grassmannian is defined to be the fiber $\pr_{2}^{-1}(\star)$ of the open point $\star\in\Bun_{G(0,1)}$ in $\Hk_{G(0,1)}$. The open subscheme $\GRc\subset\GR$ is defined as the preimage of the union of the open points $\sqcup_{\omega\in\Omega}\omega(\star)\subset\Bun_{G(0,1)}$ under the first projection $\pr_{1}:\GR\to\Bun_{G(0,1)}$. Therefore $\GRc$ is the union of the preimages of $(\omega(\star),\star)$ under the projections $\pr_{1}\times\pr_{2}:\Hk_{G(0,1)}\to\Bun_{G(0,1)}\times\Bun_{G(0,1)}$. By the moduli interpretation of $\Hk_{G(0,1)}$, this preimage is exactly $\frG$. The rest of the lemma is clear.
\end{proof}

\begin{remark}\label{r:frG to Gr} Let us give the morphism  $\frG_{x}\incl\GR_{x}$ explicitly. The $G$-torsors $\calE^{\star}_{0,1}$ and $\calE^{\omega(\star)}_{0,1}$ are canonically trivialized at $x\in\Gm$ by definition. Therefore, restricting an isomorphism $\varphi\in\frG_{x}$ to the formal punctured neighborhood of $x$ gives a homomorphism of group ind-schemes
\begin{equation}\label{frG to LG}
\ev_{x}:\frG_{x}\to LG_{x}
\end{equation}
where $LG_{x}$ is the loop group of $G$ at $x$. Now the morphism $\frG_{x}\to\GR_{x}$ is given by composing \eqref{frG to LG} with the projection $LG_{x}\to\GR_{x}=LG_{x}/L^{+}G_{x}$.

In \S\ref{bifilter} we will give a more explicit description of $\frG$ in the case $G=\PGL_{n}$.
\end{remark}

\subsubsection{Torus action on $\frG$} We may  define a variant of $\frG$ as follows. Instead of considering isomorphisms between $\calE^{\omega(\star)}_{0,1}$ we consider $\calE^{\omega(\star)}_{0,0}$. This way we get an ind-scheme $\frG'^{\omega}\to\Gm$ whose fiber over $x\in\Gm$ is  $\Isom_{\PP^{1}-\{x\}}(\calE^{\omega(\star)}_{0,0}, \calE^{\star}_{0,0})$. Let $\frG'=\sqcup_{\omega\in\Omega}\frG'^{\omega}$. Similarly, we have the evaluation map $\ev'_{\infty}:\frG'\to\bI_{\infty}$ and hence a morphism
\begin{equation*}
\tau:\frG'\to\bI_{\infty}\to\bI_{\infty}/\bI_{\infty}(1)=T.
\end{equation*}
By construction, $\frG=\tau^{-1}(1)$. Note also that $T=\Aut_{\PP^{1}}(\calE^{\omega(\star)}_{0,0})=\Aut_{\PP^{1}}(\calE^{\star}_{0,0})$ acts on each $\frG'^{\omega}_{x}$ both from the right and from the left by pre and post composition. Combing these two actions gives an adjoint action of $T$ on $\frG'$. The morphism $\ev'_{\infty}$ is equivariant under the adjoint $T$-actions: $T$ also acts on $\bI_{\infty}$ by conjugation. Since the homomorphism $\tau$ is invariant under the adjoint $T$-action, the $T$-action preserves $\frG$. Therefore we obtain an action of $T$ on $\frG$, such that the morphism $f_{\frG}:\frG\to\bI_{\infty}(1)/\bI_{\infty}(2)$ is $T$-equivariant.

Recall $\Grot$ is the one-dimensional torus acting on $\PP^{1}$ by scaling the natural coordinate: $s\in\Grot$ as $t\mapsto s^{-1}t$. This action induces a $\Grot$-action on $\frG\cong\GRc$ which commutes with the adjoint $T$-action. This way we conclude

\begin{lemma}
There is an action of $T\times\Grot$ on $\frG$ making the morphism $f_{\frG}:\frG\to\bI_{\infty}(1)/\bI_{\infty}(2)\cong\Ga^{r+1}$ equivariant under $T\times\Grot$. Here,  for $i=0,1\cdots,r$, $T\times\Grot$ acts on the $i\nth$ coordinate of $\Ga^{r+1}$ through the simple affine root $\alpha_{i}$  of  the loop group $LG_{\infty}$. Identifying $\frG_{x}$ with $\GRc_{x}$, the $T$-action on $\GRc_{x}\subset LG_{x}/L^{+}G_{x}$ is via left multiplication.
\end{lemma}

Define a one-dimension subtorus $\Gm^{(\rho^\vee,h)}\subset T\times\Grot$ by
\begin{equation*}
\Gm^{(\rho^\vee,h)}\ni s\mapsto (s^{\rho^{\vee}}, s^{h})\in T\times\Grot.
\end{equation*}
where $h$ is the Coxeter number of $G$. For example, when $G=\PGL_{n}$, $\Gm^{(\rho^\vee,h)}\subset T\times\Grot$ is given by $(\diag(s^{n-1},s^{n-2},\cdots, s^{1},1), s^{n})$. This subtorus $\Gm^{(\rho^\vee,h)}$ has the property that it acts on each simple affine root space via by dilation. Therefore the map $F=\sigma\circ f_{\frG}:\frG\cong\GRc\to\AA^{1}$ is $\Gm^{(\rho^\vee,h)}$-equivariant, where $\Gm^{(\rho^\vee,h)}$  acts on $\AA^{1}$ still by dilation.

Using the $\Grot$-action we may trivialize the group ind-scheme $\frG\to\Gm$ and write $\frG=\frG_{1}\times\Gm$. We have a canonical identification $\frG_{1}\cong\GRc_{1}$. We may also identify $\GR_{1}$ with the affine Grassmannian $\Gr=LG/L^{+}G$ (by using the standard local coordinate $t-1$ at $1\in\Gm$). We write $\Grc$ for $\GRc_{1}$ under this identification. Now $T$ acts only on the $\frG_{1}$-factor (by conjugation) and $\Grot$ only acts on the $\Gm$-factor (by translation). Therefore
\begin{equation*}
[\GRc/\Gm^{(\rho^\vee,h)}]\cong[\Grc/\mu_{h}^{\rho^{\vee}}]\cong[\frG_{1}/\Ad(\mu_{h}^{\rho^{\vee}})].
\end{equation*}
The notation emphasizes that $\mu_{h}$ acts on  $\Grc$ via the cocharacter $\rho^{\vee}:\Gm\to T$ and the adjoint $T$-action on $\frG_{1}\cong\Grc$. 

Because $F$ is equivariant under $\Gm^{(\rho^\vee,h)}$, it descends to a morphism of stacks $[\GRc/\Gm^{(\rho^\vee,h)}]\to[\AA^{1}/\Gm^{(\rho^\vee,h)}]$. Therefore we may rewrite $F$ as a morphism
\begin{equation}\label{barF}
\barF: [\Grc/\mu_{h}^{\rho^{\vee}}]\cong[\frG_{1}/\Ad(\mu_{h}^{\rho^{\vee}})]\to[\AA^{1}/\Gm].
\end{equation}

Every object in $\Sat^{\Tt}$ is also $T$-equivariant, and in particular $\mu_{h}$-equivariant. We use the same notation $\IC_{V}$ to mean a $\mu_{h}$-equivariant perverse sheaf on $\Gr$, whose restriction to $\Grc$ is still denoted by $\ICc_{V}$. The complexes $F_{!}\IC_{V,\GRc}$ and $F_{*}\IC_{V,\GRc}$ belong to $D^{b}_{\Gm}(\AA^{1},\Ql)$, and can be identified with $\barF_{!}\ICc_{V}$ and $\barF_{*}\ICc_{V}$.

Let $S$ be any base scheme. For an object $\calF\in D^{b}_{\GmS}(\AA^{1}_{S},\Ql)$, we use $\calF_{\eta}$ to denote its restriction to the ``open point'' $S=[\GmS/\GmS]$. The following proposition gives a formula for $M^{V}_{!}$, which has the potential of making sense over a general base scheme $S$ other than $\Spec k$.

\begin{prop}\label{p:MV} There are canonical isomorphisms in $D^{b}(\Spec k, \Ql(\mu_{p}))$
\begin{equation}\label{MV! HF} 
M^{V}_{!}\cong\HF(\barF_{!}\ICc_{V})_{\eta}\otimes\Ql(\mu_{p}).
\end{equation}
\end{prop}
\begin{proof}
By definition of $\Kl^{V}$ and $M^{V}_{!}$, and proper base change, we have
\begin{equation*}
M^{V}_{!}=\bR\Gamma_{c}(\Gm,\Kl^{V})[1]=\bR\Gamma_{c}(\GRc, \IC_{V,\GRc}\otimes F^{*}\AS_{\psi})[1]=\bR\Gamma_{c}(\AA^{1}, F_{!}\IC_{V,\GRc}\otimes\AS_{\psi})[1].
\end{equation*}
The last object is the stalk at $1\in\AA^{1}(k)$ of the Fourier-Deligne transform $\Four_{\psi}(F_{!}\IC_{V,\GRc})$, see \S\ref{Fourier-Del}. By Theorem \ref{th:Fourier}\eqref{FD=HF}, we conclude that $M^{V}_{!}$ is the same as the stalk at $1$ of $\HF(\barF_{!}\ICc_{V})$, which is the same as $\HF(\barF_{!}\ICc_{V})_{\eta}\otimes\Ql(\mu_{p})$.
\end{proof}

\subsection{Autoduality of moments}\label{ss:autoduality}
For $V,W\in\Rep(\hatG^{\Tt})$, the isomorphism $\Kl^{V}\otimes\Kl^{W}\isom\Kl^{V\otimes W}$ induces a cup product
\begin{equation*}
\cohoc{*}{\Gm,\Kl^{V}}\otimes\cohoc{*}{\Gm,\Kl^{W}}\to\cohoc{*}{\Gm,\Kl^{V\otimes W}},
\end{equation*}
which is equivalent to a map
\begin{equation}\label{Klcup}
M^{V}_{!}\otimes M^{W}_{!}\to M^{V\otimes W}_{!}[1].
\end{equation}
When $W=V^{\vee}$, $\Kl^{V^{\vee}}$ is the dual local system of $\Kl^{V}$, therefore \eqref{Klcup} composed with the evaluation map $V\otimes V^{\vee}\to \Ql$ (the trivial representation) gives
\begin{equation}\label{pairing VVdual}
M^{V}_{!}\otimes M^{V^{\vee}}_{!}\to M^{V\otimes V^{\vee}}_{!}[1]\cong\cohoc{*}{\Gm,\Ql}[2]\to\Ql(-1)
\end{equation}
where the last map is given by projecting to $\cohoc{2}{\Gm,\Ql}$.

\subsubsection{The opposition $\sigma$}\label{sss:opp} Recall that in \cite[\S6.1]{HNY} we studied the behavior of $\Kl_{\hatG}$ under pinned automorphisms of $G$. Let $\sigma$ be the opposition on $G$: this is the pinned automorphism of $G$ which acts as $-w_{0}$ on $\xcoch(T)$. Then $\sigma$ induces an involution (which we still denote by $\sigma$) on $\hatG^{\Tt}$ as well which is the identity on $\Gm^{\Wt}$. Then \cite[Corollary 6.5 and Table 2]{HNY} says that when $G$ is not of type $A_{2n}$, the monodromy of $\Kl_{\hatG}$ lies in $\hatG^{\Tt,\sigma}$, i.e., we have a tensor functor
\begin{equation*}
\calK: \Rep(\hatG^{\Tt,\sigma})\to\Loc(\Gm)
\end{equation*}
(we suppress the coefficient field $\Ql(\mu_{p})$ in the notation), denoted $V\mapsto \calK^{V}$, such that the following diagram is made commutative by a canonical natural isomorphism
\begin{equation*}
\xymatrix{\Rep(\hatG^{\Tt})\ar[r]^{\Res}\ar@/_2pc/[rr]^{\Kl_{\hatG}} & \Rep(\hatG^{\Tt,\sigma})\ar[r]^{\calK} & \Loc(\Gm)}
\end{equation*}
When $G$ is of type $A_{2n}$, then {\em loc.cit.} says that there is a $\hatG\rtimes\jiao{\sigma}$-local system $\calK$ on $\Gm$ such that the following diagram is made commutative by a canonical natural isomorphism
\begin{equation*}
\xymatrix{\Rep(\hatG^{\Tt}\rtimes\jiao{\sigma})\ar[r]^{\calK}\ar[d]^{\Res} & \Loc(\Gm)\ar[d]^{[2]^{*}}\\
\Rep(\hatG^{\Tt})\ar[r]^{\Kl_{\hatG}} & \Loc(\Gm)}
\end{equation*}
Here $[2]^{*}$ is the pullback functor along the square morphism $[2]:\Gm\to\Gm$.

For $V\in\Rep(\hatG^{\Tt})$, let $V^{\sigma}$ be the same vector space with the action of $\hatG^{\Tt}$ given by the composition $\hatG^{\Tt}\xrightarrow{\sigma}\hatG^{\Tt}\to\GL(V)$.

\begin{lemma}\label{l:sigma} Let $V\in\Rep(\hatG)$. When $G$ is not of type $A_{2n}$ there is a canonical isomorphism
\begin{equation}\label{Kl not A2n}
\Kl^{V}_{\hatG}\isom\Kl^{V^{\sigma}}_{\hatG}.
\end{equation}
When $G$ is of type $A_{2n}$, there is a a canonical isomorphism
\begin{equation}\label{Kl A2n}
[2]_{*}\Kl^{V}_{\hatG}\isom[2]_{*}\Kl^{V^{\sigma}}_{\hatG}.
\end{equation}
In particular, for all $G$, we always have a canonical isomorphism
\begin{equation}\label{sigma M}
M^{V}_{!}\isom M^{V^{\sigma}}_{!}.
\end{equation}
\end{lemma}
\begin{proof} We shall use the local system $\calK$ introduced in \S\ref{sss:opp}.
When $G$ is not of type $A_{2n}$, both sides of \eqref{Kl not A2n} can be identified with $\calK^{V}$ (since $V=V^{\sigma}$ as $\hatG^{\Tt,\sigma}$-representations). When $G$ is of type $A_{2n}$, both sides of \eqref{Kl A2n} can be identified with $\calK^{W}$ where $W=\Ind^{\hatG^{\Tt}\rtimes\jiao{\sigma}}_{\hatG^{\Tt}}V=\Ind^{\hatG^{\Tt}\rtimes\jiao{\sigma}}_{\hatG^{\Tt}}V^{\sigma}\in\Rep(\hatG^{\Tt}\rtimes\jiao{\sigma})$.
\end{proof}

\begin{prop}\label{p:autodual p} Let $V$ be an irreducible object in $\Rep(\hatG^{\Tt})$ of weight $w$. Then there is a canonical isomorphism
\begin{equation}\label{dual M}
M^{V}_{!}\cong M^{V^{\vee}}_{!}(-w)
\end{equation}
such that the pairing
\begin{equation}\label{autodual p}
M^{V}_{!}\otimes M^{V}_{!}\cong M^{V}_{!}\otimes M^{V^{\vee}}_{!}(-w)\xrightarrow{\eqref{pairing VVdual}}\Ql(-w-1)
\end{equation}
is $(-1)^{w+1}$-symmetric.
\end{prop}
\begin{proof}
Since $V$ is irreducible of weight $w$, we have $V^{\vee}(-w)\cong V^{\sigma}$. The isomorphism \eqref{dual M} comes from this isomorphism and \eqref{sigma M}. 

We may rewrite the isomorphism  $V^{\vee}(-w)\cong V^{\sigma}$ as a perfect pairing
\begin{equation}\label{semilinear}
(\cdot,\cdot):V\otimes V\to\Ql(-w)
\end{equation}
such that $(gu,\sigma(g)v)=(u,v)$ for all $g\in\hatG$.  When $G$ is not of type $A_{2n}$, we have a local system $\calK^{V}$ constructed in let $\hatH=\hatG^{\Tt,\sigma}$ and let $W=V$ viewed as an $\hatH$-representation. When $G$ is of type $A_{2n}$, let $\hatH=\hatG^{\Tt}\rtimes\jiao{\sigma}$ and let $W=\Ind^{\hatH}_{\hatG^{\Tt}}V$ as an $\hatH$-representation. From the proof of Lemma \ref{l:sigma}, the pairing \eqref{autodual p} is induced from the pairing of local systems on $\Gm$:
\begin{equation}\label{KW}
\calK^{W}\otimes\calK^{W}\to\Ql(-w).
\end{equation}
The above pairing comes from a pairing $W\otimes W\to\Ql(-w)$ of $\hatH$-modules, which in turn comes from the pairing \eqref{semilinear} (via an induction process when $G$ is of type $A_{2n}$).  Hence the pairing \eqref{KW} has the same symmetry type as \eqref{semilinear}.

We now claim that \eqref{semilinear} is $(-1)^{w}$-symmetric. In fact, since $V$ is an irreducible representation of $\hatG$, pairings on $V$ satisfying $(gu,\sigma(g)v)=(u,v)$ (for all $g\in\hatG$) are unique up to a scalar. In particular, the transposed pairing $(u,v)':=(v,u)$ also satisfies the same condition, hence $(v,u)=(u,v)'=c(u,v)$ for some constant $c$, and clearly $c=\pm1$. Let $\varphi:\SL_{2}\to \hatG^{\sigma}\subset\hatG$ be the principal $\SL_{2}$ whose restriction to the diagonal torus of $\SL_{2}$ is given by the cocharacter $2\rho$ of $\hatG$. Then \eqref{semilinear} is $\varphi(\SL_{2})$-invariant. Decomposing $V$ as a direct sum of irreducible $\SL_{2}$-modules, the irreducible representation $\Sym^{n}$ with $n$ largest appears with multiplicity one. In fact, the highest weight space under the $\SL_{2}$-action is the same as the highest weight space for $\hatG$ (which is one-dimensional), and we have $n=\jiao{2\rho,\l}$ if $\l$ is the highest weight of $V$. The symmetry type $c$ of $V$ is then the same as the symmetric type when restricted to $\Sym^{n}$. From the discussion in \S\ref{irr}, the weight $w$ of $V$ is also $\jiao{2\rho,\l}$. Since any $\SL_{2}$-invariant autoduality on $\Sym^{n}$ is $(-1)^{n}$-symmetric, the constant $c$ must be $(-1)^{n}=(-1)^{w}$.

Now we have shown that the pairing \eqref{KW} is $(-1)^{w}$-symmetric. Since taking $\upH^{1}_{c}$ changes the sign of the pairing, the pairing \eqref{autodual p} is $(-1)^{w+1}$-symmetric.
\end{proof}

\begin{remark}\label{r:factor} Since $M^{V}_{!*}=\cohog{1}{\PP^{1},j_{!*}\Kl^{V}}$ by Lemma \ref{l:MV pure}, the pairing \eqref{pairing VVdual} factors through a perfect pairing
\begin{equation*}
M^{V}_{!*}\otimes M^{V^{\vee}}_{!*}\to\Ql(-1).
\end{equation*}
Therefore the pairing \eqref{autodual p} also factors through a $(-1)^{w+1}$-symmetric perfect pairing 
\begin{equation*}
M^{V}_{!*}\otimes M^{V}_{!*}\to\Ql(-w-1).
\end{equation*}
\end{remark}

%%%%%%%%%%%%%

\section{The Galois representations attached to moments}

The goal of this section is to prove the Main Theorem \ref{th:main}, which attaches a Galois representation of $\GQ$ to each pair $(\hatG, V\in\Rep(\hatG^{\Tt}))$, such that the local behavior at $p$ of this Galois representation is closely related to the moment $M^{V}_{!*,\FF_{p}}$ defined earlier in Definition \ref{def:moments}. The proof starts by ``gluing'' the moments defined over $\FF_{p}$ into a moment defined over $\Spec\ZZ$.

\subsection{Moments over $\ZZ$} In this subsection, we consider schemes  over $S=\SZl$. We use underlined letters to denote schemes over $S$. We use the following notation for inclusions of closed points and the generic point of $S$:
\begin{equation*}
\Spec\FF_{p}\xrightarrow{i_{p}}S\xleftarrow{j}\Spec\QQ.
\end{equation*}

Let $\unG$ be the Chevalley group over $S$ with the same root datum as the simple group $G$ we fixed in the beginning. The affine Grassmannian $\GR, \Gr$,  as well as $\GRc\cong\frG$, $\Grc\cong\frG_{1}$ have natural models $\unGR, \unGr$, etc. over $S$. For example, the affine Grassmannian $\unGr$ over $S$ can be defined as the fppf quotient $L\unG/L^{+}\unG$. See \cite[\S5.b.1, Proposition 5.3]{PZ}. Its generic fiber $\Gr_{\QQ}$ and special fibers $\Gr_{\FF_{p}}$ are the old affine Grassmannians defined over a field, see \cite[Corollary 5.6]{PZ}.

\begin{lemma}\label{l:IC model} Let $V\in\Rep(\hatG^{\Tt})$. Let $j:\Gr_{\QQ}\incl\unGr$ be the inclusion of the generic fiber and $i_{p}:\Gr_{\FF_{p}}\incl\unGr$ be the inclusion of the mod $p$ special fiber. Let $\unIC_{V}=j_{!*}\IC_{V}$. Then we have
\begin{equation*}
i^{*}\unIC_{V}\cong\IC_{V,\FF_{p}}.
%\label{IC model !} i^{!}\unIC_{V}[2](1)\cong\IC_{V,\FF_{p}}.
\end{equation*}
Here $\IC_{V,\FF_{p}}\in\Sat^{\Tt}_{\FF_{p}}$ is the perverse sheaf on $\Gr_{\FF_{p}}$ corresponding to $V$.
\end{lemma}
\begin{proof}Since we are dealing with a single prime $p$ at a time, we may restrict the situation to $\SZp$, and retain the same notation.
By \cite[Proposition 9.15]{PZ}, the nearby cycles sheaf $\Psi(\IC_{V})$ (a perverse sheaf on $\Gr_{\FF_{p}}$) is isomorphic to $\IC_{V,\FF_{p}}$. Also, by \cite[Proposition 9.12]{PZ}, the monodromy action on $\Psi(\IC_{V})$ is trivial (because we are in the situation of a split $G$ and hyperspecial parahoric level structure, in the context of \cite{PZ}). Then, by \cite[Lemma 9.19]{PZ}, $i^{*}j_{!*}\IC_{V}$ is isomorphic to $\Psi(\IC_{V})$, which is in turn isomorphic to $\IC_{V,\FF_{p}}$.
\end{proof}

\begin{remark}\label{r:tensor int} We may similarly define the convolution product of $\unIC_{V}*\unIC_{W}$ on $\unGr$, using the integral model of the convolution diagram \cite[\S4]{MV}. We claim that there is a natural isomorphism 
\begin{equation*}
\unIC_{V\otimes W}\cong\unIC_{V}*\unIC_{W}
\end{equation*}
In fact, both sides above are isomorphic when restricted to $\Gr_{\QQ}$, and they both have trivial monodromy on the nearby cycles towards the special fiber $\Gr_{\FF_{p}}$ (in the case of $\unIC_{V}*\unIC_{W}$, use fact that the nearby cycles commute with proper push-forward). Therefore by \cite[Lemma 9.19]{PZ} again, both sides above are middle extensions of their restriction to $\Gr_{\QQ}$, hence they are canonically isomorphic to each other.
\end{remark}

Although we do not have Kloosterman sheaves defined over $\GG_{m,S}$, the moments of Kloosterman sheaves do have natural analogs over $S$. In fact,  from the description of $\frG\cong\GRc$ in \S\ref{alt Gr}, we see that the morphism $\barF$ in \eqref{barF} can be extended to a morphism over $S$:
\begin{equation}\label{bunF}
\bunF:[\unGr^{\c}/\un{\mu}_{h}^{\rho^{\vee}}]\cong[\un{\frG}_{1}/\Ad(\un{\mu}_{h}^{\rho^{\vee}})]\to[\un{\AA}^{1}/\unGm].
\end{equation}
From Proposition \ref{p:MV}, we arrive at the following natural definition.

\begin{defn}\label{def:moments Q} For $V\in\Rep(\hatG^{\Tt})$, we define
\begin{eqnarray*}
\unM^{V}_{!}=\HF(\bunF_{!}\unIC^{\c}_{V})_{\eta}
%M^{V}_{*,\QQ}=\HF(\barF_{*}\IC_{V})_{\eta};\\
%M^{V}_{!*,\QQ}:=\Im(\upH^{0}M^{V}_{!,\QQ}\to\upH^{0}M^{V}_{*,\QQ}).
\end{eqnarray*}
as an object in $D^{b}_{c}(S,\Ql)$. Here, as usual, $\unIC^{\c}_{V}$ denotes the restriction of $\unIC_{V}$ to $\unGr^{\c}$.
\end{defn}
%By definition, $M^{V}_{!,\QQ}$ and $M^{V}_{*,\QQ}$ are object in $D^{b}(\Spec\QQ,\Ql)$, while $M^{V}_{!*,\QQ}$ is a $\Ql$-sheaf on $\Spec\QQ$, which is the same as a continuous $\GQ$ representation.

\begin{lemma}\label{l:unM stalks}
We have an isomorphism in $D^{b}(\Spec\FF_{p},\Ql(\mu_{p}))$ 
\begin{equation}\label{p stalk}
i^{*}_{p}\unM^{V}_{!}\otimes\Ql(\mu_{p})\cong M^{V}_{!,\FF_{p}}.
\end{equation}
Here $ M^{V}_{!,\FF_{p}}$ is defined in Definition \ref{def:moments} for the base field $k=\FF_{p}$.
\end{lemma}
\begin{proof} 
For arbitrary base change $i:S'\to S$, the formation of $\HF$ commutes with $i^{*}$: this follows from proper base change. Therefore we have
\begin{equation*}
i^{*}_{p}\unM^{V}_{!}\cong i^{*}_{p}\HF(\bunF_{!}\unIC^{\c}_{V})_{\eta}\cong \HF(i^{*}_{p}\bunF_{!}\unIC^{\c}_{V})_{\eta}\cong\HF(\barF_{!}i^{*}_{p}\unIC^{\c}_{V})_{\eta}
\end{equation*}
By Lemma \ref{l:IC model}, we have $i^{*}_{p}\unIC^{\c}_{V}\cong \IC^{\c}_{V,\FF_{p}}$. Therefore
\begin{equation*}
i^{*}_{p}\unM^{V}_{!}\cong\HF(\barF_{!}\IC^{\c}_{V,\FF_{p}})_{\eta}.
\end{equation*}
Using Proposition \ref{p:MV}, we get \eqref{p stalk}. 
\end{proof}

The next proposition will be used in the next subsection to establish a pairing on $\unM^{V}_{!}$ analogous to the pairing on $M^{V}_{!,\FF_{p}}$ defined in Proposition \ref{p:autodual p}. What we need there is a definition of the cup product \eqref{Klcup} without appealing to Kloosterman sheaves.

\begin{prop}\label{p:tensor moments} For $V,W\in\Rep(\hatG^{\Tt})$, there is a bifunctorial morphism
\begin{equation}\label{cup unM}
\unM^{V}_{!}\otimes\unM^{W}_{!}\to\unM^{V\otimes W}_{!}[1]
\end{equation}
such that after taking $i^{*}_{p}$ for any prime $p\nmid\ell$ and using the isomorphism \eqref{p stalk}, the above morphism becomes the cup-product \eqref{Klcup} for the base field $\FF_{p}$.
\end{prop}
\begin{proof}
The moduli stack $\Bun_{G(0,1)}$ defined in \S\ref{alt Gr} has a natural integral model $\unBun_{\unG(0,1)}$ over $S$, which contains an open substack $\unBun_{\unG(0,1)}^{\c}=\sqcup_{\omega\in\Omega}\omega(\un{\star})$ consisting of the open $S$-points of $\unBun_{\unG(0,1)}$. Consider the moduli functor $\unfrG^{(2)}$ classifying the data $(x,\calE',\calE,\alpha,\beta)$ where $x\in\unGm, \calE'\in\unBun_{\unG(0,1)}, \calE\in\unBun_{\unG(0,1)}^{\c}$, $\alpha:\calE|_{\un{\PP}^{1}-\{x\}}\isom\calE'|_{\un{\PP}^{1}-\{x\}}$ and $\beta:\calE'|_{\un{\PP}^{1}-\{x\}}\isom\calE^{\un{\star}}_{0,1}|_{\un{\PP}^{1}-\{x\}}$.

We have morphisms
\begin{equation*}
\xymatrix{ & \unfrG^{(2)}\ar[dl]^{p_{1}}\ar[dr]^{p_{3}}\ar[d]^{p_{2}} \\
\unGR & \unGRc & \un{\Hk}^{\c}_{\unG(0,1)} }
\end{equation*}
where $p_{1}(x,\calE',\calE,\alpha,\beta)=(x,\calE',\beta); p_{2}(x,\calE',\calE,\alpha,\beta)=(x,\calE,\beta\circ\alpha)$ and $p_{3}(x,\calE',\calE,\alpha,\beta)=(x,\calE,\calE',\alpha)$. The fibers of $p_{1}$ are isomorphic to the affine Grassmannian $\unGr$, hence $\unfrG^{(2)}$ is also represented by an ind-scheme. 

We have
\begin{equation}\label{convIC}
\unICc_{V\otimes W}\cong(\unIC_{V}*\unIC_{W})|_{\unGRc}\cong p_{2,!}(p_{1}^{*}\unIC_{V}\otimes p_{3}^{*}\unIC_{W}).
\end{equation}
Here the first equality follows from Remark \ref{r:tensor int}. In the last term, we use $\unIC_{W}$ to denote also the complex on $\un{\Hk}_{\unG(0,1)}$ whose restriction to each fiber of the projection $\un{\Hk}_{\unG(0,1)}\to\unGm\times\unBun_{\unG(0,1)}$ (the fibers are identified with the affine Grassmannian) is isomorphic to $\unIC_{W}$. The second equality follows from the definition of the convolution product for sheaves on the affine Grassmannian in \cite[\S4]{MV}.

We consider the open ind-scheme $j:\unfrG^{(2),\c}\incl\unfrG^{(2)}$ classifying the data $(x,\calE',\calE,\alpha,\beta)$ with the open condition that $\calE'\in\unBun_{\unG(0,1)}^{\c}$. Then there is an isomorphism
\begin{equation*}
\unfrG\times_{\unGm}\unfrG\isom\unfrG^{(2),\c}
\end{equation*}
mapping $(x,\varphi':\calE^{\omega'(\star)}_{0,1}|_{\PP^{1}-\{x\}}\to\calE^{\star}_{0,1}|_{\PP^{1}-\{x\}}, \varphi:\calE^{\omega(\star)}_{0,1}|_{\PP^{1}-\{x\}}\to\calE^{\star}_{0,1}|_{\PP^{1}-\{x\}})$ to $(x,\calE^{\omega'\omega(\star)}_{0,1},\calE^{\omega'(\star)}_{0,1}, \omega'(\varphi), \varphi')$. 
Under this isomorphism, the morphism $p_{2}^{\c}=p_{2}|_{\unfrG^{(2),\c}}:\unfrG^{(2),\c}\to\unGRc\cong\unfrG$ becomes the multiplication map under the group structure of $\unfrG$
\begin{equation*}
m:\unfrG\times_{\unGm}\unfrG\to\unfrG
\end{equation*}
Moreover, 
\begin{equation*}
j^{*}(p_{1}^{*}\unIC_{V}\otimes p_{3}^{*}\unIC_{W})\cong \unICc_{V}\boxtimes_{\unGm}\unICc_{W}.
\end{equation*}
Therefore, by adjunction, we get a map
\begin{equation}\label{conv VW}
m_{!}(\unICc_{V}\boxtimes_{\unGm}\unICc_{W})\cong p_{2,!}j_{!}j^{*}(p_{1}^{*}\unIC_{V}\otimes p_{3}^{*}\unIC_{W})\to p_{2,!}(p_{1}^{*}\unIC_{V}\otimes p_{3}^{*}\unIC_{W})\cong\unICc_{V\otimes W}.
\end{equation}
Since $\unF:\unfrG\to \un{\AA}^{1}$ is a group homomorphism, we have an isomorphism in $D^{b}_{\unGm}(\un{\AA}^{1})$:
\begin{equation*}
\unF_{!}m_{!}(\unICc_{V}\boxtimes_{\unGm}\unICc_{W})\cong \unF_{!}\unICc_{V}*_{+}\unF_{!}\unICc_{W}.
\end{equation*}
Here $\calF_{1}*_{+}\calF_{2}=a_{!}(\calF_{1}\boxtimes\calF_{2})$ means the convolution on $\un{\AA}^{1}$  using the addition map $a:\un{\AA}^{1}\times\un{\AA}^{1}\to\un{\AA}^{1}$. Descending to the stack $[\un{\AA}^{1}/\unGm]$ and combining with \eqref{conv VW}, we get a morphism in $D^{b}_{c}([\un{\AA}^{1}/\unGm])$, bifunctorial in $V$ and $W$
\begin{equation*}
\bunF_{!}\unICc_{V}*_{+}\bunF_{!}\unICc_{W}\to \bunF_{!}\unICc_{V\otimes W}
\end{equation*}
Applying $\HF$ and using the compatibility between convolution $*_{+}$ and the tensor product (see Theorem \ref{th:Fourier}\eqref{conv}), we get
\begin{equation*}
\HF(\bunF_{!}\unICc_{V})_{\eta}\otimes\HF(\bunF_{!}\unICc_{W})_{\eta}\cong\HF(\bunF_{!}\unICc_{V}*_{+}\bunF_{!}\unICc_{W})_{\eta}[1]\to\HF(\bunF_{!}\unICc_{V\otimes W})_{\eta}[1],
\end{equation*}
which gives \eqref{cup unM}.

Next we try to identify the mod $p$ stalk of \eqref{cup unM} with \eqref{Klcup}.
Now we work over $\FF_{p}$. We have the following commutative diagram 
\begin{equation*}
\xymatrix{& \frG\times_{\Gm}\frG\ar[r]^{(F,F)}\ar[d]^{m}\ar[dl]_{(\pi,\pi)} & \AA^{1}\times\AA^{1}\ar[d]^{a}\\
\Gm & \frG\ar[r]^{F}\ar[l]_{\pi} & \AA^{1}}
\end{equation*}
We get a commutative diagram of cohomology groups
\begin{equation*}
\xymatrix{\cohoc{*}{\frG, \ICc_{V}\otimes F^{*}\AS_{\psi}}\otimes\cohoc{*}{\frG, \ICc_{W}\otimes F^{*}\AS_{\psi}}\ar@[d]\ar[r]^(.6){\sim} & \cohoc{*}{\Gm,\Kl^{V}}\otimes\cohoc{*}{\Gm,\Kl^{W}}\ar@[d]\\
\cohoc{*}{\frG\times_{\Gm}\frG, \ICc_{V}\boxtimes_{\Gm}\ICc_{W}\otimes (F,F)^{*}(\AS_{\psi}\boxtimes\AS_{\psi})}\ar[d]^{\alpha}\ar[r]^(.6){\sim} &  \cohoc{*}{\Gm,\Kl^{V}\otimes\Kl^{W}}\ar[dd]^{\wr}\\
\cohoc{*}{\frG, m_{!}(\ICc_{V}\boxtimes_{\Gm}\ICc_{W})\otimes F^{*}\AS_{\psi}}\ar[d]^{\eqref{conv VW}}  \\
\cohoc{*}{\frG, \ICc_{V\otimes W}\otimes F^{*}\AS_{\psi}}\ar[r] & \cohoc{*}{\Gm,\Kl^{V\otimes W}}}
\end{equation*}
The isomorphism $\alpha$ in the middle left  column uses the composition of
\begin{equation*}
(F,F)^{*}(\AS_{\psi}\boxtimes\AS_{\psi})\cong (F,F)^{*}a^{*}\AS_{\psi}\cong m^{*}F^{*}\AS_{\psi}
\end{equation*}
and the projection formula
\begin{equation*}
m_{!}((\ICc_{V}\boxtimes_{\Gm}\ICc_{W})\otimes m^{*}F^{*}\AS_{\psi})\cong m_{!}(\ICc_{V}\boxtimes_{\Gm}\ICc_{W})\otimes F^{*}\AS_{\psi}.
\end{equation*}
It is easy to see that the composition of the left column is the same as $i^{*}_{p}$  applied to the map \eqref{cup unM}; the right column is the cup product giving \eqref{Klcup}. Therefore the result of $i^{*}_{p}$ on \eqref{cup unM} is the same as \eqref{Klcup}.
\end{proof}

\subsection{Construction of the Galois representation}\label{ss:Gal}
In this subsection we fix an irreducible representation $V$ of $\hatG^{\Tt}$, pure of weight $w$ such that $V^{\hatG^{\geom}}=0$. Here $\hatG^{\geom}$ is the Zariski closure of the geometric monodromy of $\Kl_{\hatG}$ for large $p$ as tabulated in Theorem \ref{th:global mono}. We note that this condition is rather weak: when $\hatG^{\geom}=\hatG$ this simply means that $V$ is not the trivial representation.

\subsubsection{The open subset $S'$}  Let $S_{\bad}$ be the set of primes $p$ such that $V^{\hatG^{\geom}_{p}}\neq0$, where $\hatG^{\geom}_{p}$ is the Zariski closure of the geometric monodromy of $\Kl_{\hatG}$ over $\GG_{m,\FF_{p}}$. According to the table in Theorem \ref{th:global mono}, $S_{\bad}\subset\{2,3\}$. Let $S'=S-S_{\bad}$. This is a Zariski open subset of $S$. 

When $\hatG=\SL_{2n}$ or $\Sp_{2n}$, $\Kl^{\St}_{\hatG}$ is essentially the same as $\Kl_{2n}$ according to Proposition \ref{p:relation to Kln}. Then by Katz's Theorem \ref{th:global mono Katz}, $\hatG^{\geom}_{p}=\Sp_{2n}$ for all primes $p$, hence $S'=S$ in these cases. 

\begin{lemma}\label{l:perv} The complex $\unM^{V}_{!}[1]$ is a perverse sheaf on $S$, and $\unM^{V}_{!}|_{S'}$ is a sheaf.
\end{lemma}
\begin{proof}
We first show that $\unM^{V}_{!}[1]\in\pD^{\geq0}(S,\Ql)$. In fact, $\unIC_{V}[1]$ is perverse (because the base scheme $S$ now have dimension one) and $\bunF$ when restricted to $\unGr^{\c}_{\leq\l}$ is affine (because $\unGr^{\c}_{\leq\l}$ is a closed subscheme of finite type of the ind algebraic group $\un{\frG}_{1}$). Therefore $\bunF_{!}\unIC^{\c}_{V}[1]\in\pD^{\geq0}([\un{\AA}^{1}/\unGm],\Ql)$. Since the functor $\HF$ is $t$-exact with respect to the perverse $t$-structure (Theorem \ref{th:Fourier}\eqref{perv}), $\HF(\bunF_{!}\unIC^{\c}_{V})[1]\in\pD^{\geq0}([\un{\AA}^{1}/\unGm],\Ql)$ and hence $\unM^{V}_{!}[1]\in\pD^{\geq0}(S,\Ql)$.

We then show that the stalks of $\unM^{V}_{!}$ lie in degree $\leq1$, which, combined with the previous paragraph will prove that $\unM^{V}_{!}[1]$ is perverse. Since $\unM^{V}_{!}$ is constructible, it suffices to show that $i_{p}^{*}\unM^{V}_{!}$ lies in degree $\leq1$ for every prime $p\nmid\ell$, and in degree $\leq0$ for almost  all $p$. In fact, 
by \eqref{p stalk}, the mod $p$ stalk (up to tensoring with $\Ql(\mu_{p})$) is $M^{V}_{!,\FF_{p}}\cong\cohoc{*}{\Gm, \Kl^{V}}[1]$,  which lies in degree 0 and 1. For $p$ appearing in $S'$, the Zariski closure of the geometric monodromy of $\Kl_{\hatG}$ is the same as $\hatG^{\geom}$, hence  $\cohoc{2}{\Gm, \Kl^{V}}=V_{\hatG^{\geom}}(-1)=0$ by our assumption on $V$, which implies that $i_{p}^{*}\unM^{V}_{!}$ lies in degree $0$. This last statement also shows that $\unM^{V}_{!}|_{S'}$ is a sheaf. This finishes the proof.
\end{proof}

\begin{prop}\label{p:autodual} Recall $w$ is the weight of $V$. There is a canonical pairing (of complexes on $S$):
\begin{equation}\label{! pairing}
\unM^{V}_{!}\otimes\unM^{V}_{!}\to \Ql(-w-1)
\end{equation}
which is $(-1)^{w+1}$-symmetric when restricted to $S'$, i.e., it factors through $\Sym^{2}(\unM^{V}_{!}|_{S'})$ if $w$ is odd and through $\wedge^{2}(\unM^{V}_{!}|_{S'})$ if $w$ is even.
\end{prop}
\begin{proof} For the trivial representation $\one$, $\IC_{\one}$ is the constant sheaf supported at the unit section $S\incl\unGr$. The zero section $z:S\to [\un{\AA}^{1}/\unGm]$ can be factored as $S\xrightarrow{g}\BB\unGm\xrightarrow{\alpha}[\un{\AA}^{1}/\unGm]$. Direct calculation shows that 
\begin{equation*}
\HF(z_{!}\Ql)_{\eta}\cong g^{*}g_{!}\Ql[1].
\end{equation*}
Hence
\begin{equation*}
\unM^{\one}_{!}=\HF(\bunF_{!}\ICc_{\one})_{\eta}=\HF(z_{!}\Ql)_{\eta}\cong g^{*}g_{!}\Ql[1].
\end{equation*}

Applying Proposition \ref{p:tensor moments} to $W=V^{\vee}$ (dual representation) we get a canonical map
\begin{equation}\label{pre pairing}
\unM^{V}_{!}\otimes\unM^{V^{\vee}}_{!}\to\unM^{V\otimes V^{\vee}}_{!}[1]\xrightarrow{\ep}\unM^{\one}_{!}[1]\cong g^{*}g_{!}\Ql[2]\to\Ql(-1).
\end{equation}
Here the map $\ep$ is induced by the evaluation map $V\otimes V^{\vee}\to\one$; the last map is induced by the truncation $g_{!}\Ql\to \bR^{2}g_{!}\Ql[-2]\cong\Ql[-2](-1)$.

Next, we would like to produce an isomorphism
\begin{equation}\label{isom VVdual}
\unM^{V}_{!}\cong\unM^{V^{\sigma}}_{!}
\end{equation}
whose mod $p$ fiber is the same as the isomorphism \eqref{sigma M} in Lemma \ref{l:sigma}. Now the argument of Lemma \ref{l:sigma} no longer works because we do not have Kloosterman sheaves over $S$. What we do instead is to unravel the argument of \cite[\S6.1]{HNY} for the $\sigma$-invariance of $\Kl_{\hatG}$ and give a construction of the isomorphism \eqref{isom VVdual} which only uses the morphism $\bunF$ in \eqref{bunF}. Let $\ep=-1$ if $G$ is of type $A_{2n}$ and $\ep=1$ otherwise. We view $\ep$ as a point of $\Grot$. Consider the involution $(\sigma,\ep)$ acting on $\un{\frG}\cong\unGR^{\c}$, where $\sigma$ is the pinned involution of $G$ which induces $-w_{0}$ on $\xcoch(T)$ and $\ep$ acts via $\Grot$. This action commutes with the $\unGm^{(\rho^\vee,h)}$-action on $\un{\frG}$ and descends to an involution on the stack $[\un{\frG}/\unGm^{(\rho^\vee,h)}]$. It is easy to check that $\bunF$ in \eqref{bunF} is invariant under this involution. Hence we get an isomorphism
\begin{equation*}
\unM^{V}_{!}=\HF(\bunF_{!}\unIC^{\c}_{V})_{\eta}\cong\HF(\bunF_{!}(\sigma,\ep)^{*}\unIC^{\c}_{V})_{\eta}=\HF(\bunF_{!}\unIC^{\c}_{V^{\sigma}})_{\eta}=\unM^{V^{\sigma}}_{!}.
\end{equation*}
Since $V^{\sigma}\cong V^{\vee}(-w)$, we can apply \eqref{isom VVdual} to the pairing \eqref{pre pairing} as we did in Proposition \ref{p:autodual p},  and get the desired pairing \eqref{! pairing}.  The construction guarantees that after taking the mod $p$ stalk, \eqref{! pairing} becomes the pairing \eqref{autodual p}. 

% the sign of pairing: needs more argument

Finally we determine the parity of the pairing \eqref{! pairing}. We write $\unM^{V}_{!}|_{S'}\otimes\unM^{V}_{!}|_{S'}$ as a direct sum $\Sym^{2}(\unM^{V}_{!}|_{S'})\oplus\wedge^{2}(\unM^{V}_{!}|_{S'})$. Suppose first that $w$ is even. Let
\begin{equation*}
\alpha:\unS:=\Sym^{2}(\unM^{V}_{!}|_{S'})\to \Ql(-w-1)
\end{equation*}
be the restriction of the pairing \eqref{! pairing} to the symmetric part of $\unM^{V}_{!}|_{S'}\otimes\unM^{V}_{!}|_{S'}$.  By Proposition \ref{p:autodual p}, each mod $p$ stalk of the pairing \eqref{! pairing} is skew symmetric, therefore $i_{p}^{*}\alpha=0$. Since $\unS$ is a local system over an open subset $U\subset S'$ and $i_{p}^{*}\alpha=0$ for those $p$ appearing in $U$,  $\alpha|_{U}$ must be zero.  Hence $\alpha$ factors through $\unS\to i_{*}i^{*}\unS\to \Ql(-w-1)$ where $i:S'-U\incl S'$ is the inclusion. The latter map becomes $\beta: i^{*}\unS\to i^{!}\Ql(-w-1)$ by adjunction. But now $i^{*}\unS$ lies in degrees $0$ by Lemma \ref{l:perv} while $i^{!}\Ql$ lies in degree two, $\beta$ must also be zero. This shows that $\alpha=0$ and the pairing \eqref{! pairing} factors through $\wedge^{2}(\unM^{V}_{!}|_{S'})$. The case where $w$ is odd is proved in the same way.
\end{proof}

The pairing \eqref{! pairing} induces a morphism
\begin{equation}\label{MDM}
\unM^{V}_{!}[1]\to (\DD(\unM^{V}_{!}[1]))(-w-2).
\end{equation}
from the left copy of $\unM^{V}_{!}$ to the dual of the right copy of $\unM^{V}_{!}$. Here we use the convention that the dualizing complex on $S$ is  $\Ql[2](1)$, so that the Verdier duality is $\DD(-)=\bR\un{\Hom}_{S}(-,\Ql[2](1))$. By Lemma \ref{l:perv}, both the source and the target of the above morphism are perverse sheaves.

\begin{defn}\label{define !*Q} We define an object $\unM^{V}_{!*}\in D^{b}_{c}(S,\Ql)$ by requiring that $\unM^{V}_{!*}[1]$ be the image of morphism \eqref{MDM} in the abelian category of perverse sheaves on $S$. 
\end{defn}
By definition, we have a morphism $\unM^{V}_{!}\to\unM^{V}_{!*}$ which is a surjection of perverse sheaves after applying a shift. We denote the generic stalk of $\unM^{V}_{!*}$ by $M^{V}_{!*,\QQ}$, which is viewed as a continuous $\GQ$-module (concentrated in degree zero).

By construction, $\ker(\unM^{V}_{!}\to\unM^{V}_{!*})$ is lies in the radical of the pairing \eqref{! pairing}. Therefore, \eqref{! pairing} factors  through a $(-1)^{w+1}$-symmetric pairing on $\unM^{V}_{!*}|_{S'}$ such that following diagram is commutative
\begin{equation}\label{two pairings}
\xymatrix{\unM^{V}_{!}|_{S'}\otimes\unM^{V}_{!}|_{S'}\ar@<3ex>[d]\ar@<-3ex>[d]\ar[r]^{\eqref{! pairing}} & \Ql(-w-1)\ar@{=}[d]\\
\unM^{V}_{!*}|_{S'}\otimes\unM^{V}_{!*}|_{S'}\ar[r] & \Ql(-w-1)}
\end{equation}
Taking the generic stalk $M^{V}_{!*,\QQ}$ of $\unM^{V}_{!*}$, we obtain a $(-1)^{w+1}$-symmetric Galois equivariant pairing on $M^{V}_{!*,\QQ}$:
\begin{equation}\label{perfect pairing Q}
\jiao{,}^{V}: M^{V}_{!*,\QQ}\otimes M^{V}_{!*,\QQ}\to\Ql(-w-1).
\end{equation}
This pairing is perfect because by construction, $\ker(M^{V}_{!,\QQ}\to M^{V}_{!*,\QQ})$ is exactly the radical of the generic stalk of the pairing \eqref{! pairing}.

The Galois representation $M^{V}_{!*,\QQ}$ is the Galois representation $M_{\ell}$ mentioned in Theorem \ref{th:main}. The rest of the theorem will be proved in the subsequent subsections.

\subsection{Proof of Theorem \ref{th:main}\eqref{ram}} We keep the same notation from the previous subsection.
We shall need the following notion. Let $U\subset S$ be a nonempty Zariski open subset.

\begin{defn}\label{def:mid} A $\Ql$-sheaf $\unP$ on $U$ is called {\em punctual} if it is a direct sum of sheaves supported at closed points of $U$. A $\Ql$-sheaf $\unM$ on $U$ is said to be a {\em middle extension sheaf}, if $\unM[1]$ is a middle extension perverse sheaf, i.e., $\unM[1]=u_{!*}u^{*}\unM[1]$ for some non-empty open subset $u:U'\incl U$.
\end{defn}

We collect some elementary properties about middle extension $\Ql$-sheaves, which are standard exercises in perverse sheaf theory and we omit the proof.
\begin{lemma}\label{l:ME} Let $\ME(U,\Ql)$ be the category of middle extension $\Ql$-sheaves on $U$. Then we have
\begin{enumerate}
\item The category $\ME(U,\Ql)$ is a Serre subcategory of $\Ql$-sheaves on $U$. 
%\item Every sheaf $\unM$ on $\SZp$ admits a unique presentation as an extension
%\begin{equation}
%0\to \unM^{0}\to \unM\to \unM^{1}\to0,
%\end{equation}
%where $\unM^{0}\cong\pH^{0}\unM$ is  supported at the closed point and $\unM^{1}=(\pH^{1}\unM)[-1]\in\ME(\SZp,\Ql)$.
\item Let  $\unM\in\ME(U,\Ql)$, then its geometric generic stalk $M$ affords a continuous Galois representation
\begin{equation*}
\rho_{\unM}:\GQ\to\GL(M)
\end{equation*}
The functor $\unM\mapsto(M,\rho_{\unM})$ gives an equivalence of categories
\begin{equation*}
\ME(U,\Ql)\isom\Rep^{\textup{aeur}}_{\textup{cont}}(\GQ,\Ql)
\end{equation*}
where $\Rep^{\textup{aeur}}_{\textup{cont}}(\GQ,\Ql)$ is the abelian category of continuous representations of $\GQ$ into finite dimensional $\Ql$-vector spaces which are unramified at almost all primes $p$.
\item\label{ME stalk} Let $\Spec\FF_{p}$ be a closed point of $U$ and let $M$ be the geometric generic stalk of $\unM$ which affords the $\GQ$-action. Fix an embedding of $\Qbar\incl\Qpbar$ and use it to define the decomposition and the inertia groups at $p$, i.e., $\GQ>\GQp>\calI_{p}$. Then we have canonical isomorphisms of $\Frob_{p}$-modules
\begin{equation*}
i^{*}_{p}\unM\cong M^{\calI_{p}}.
\end{equation*}
\end{enumerate}
\end{lemma}

\begin{prop}\label{p:midext}  The complex $\unM^{V}_{!*}|_{S'}$ is a middle extension sheaf on $S'$.
\end{prop}
\begin{proof} Let $\unM=\unM^{V}_{!}|_{S'}$. Then by Lemma \ref{l:perv}, $\unM[1]$ is a perverse sheaf satisfying the following stronger condition: the mod $p$ stalks of $\unM$ all lie in degree $-1$. We claim that
\begin{enumerate}
\item There is no nonzero map from $\unM[1]$ to punctual sheaves.
\item $\unM[1]$ fits into an exact sequence of perverse sheaves
\begin{equation*}
0\to \unP\to\unM[1]\to \unN[1]\to0
\end{equation*}
where $\unP$ is a punctual sheaf and $\unN$ is a middle extension sheaf. 
\end{enumerate}
For (1), suppose $i_{p,*}P$ is a punctual sheaf supported at some closed point $\Spec\FF_{p}\incl S'$, with $P$ in degree zero, then a map $\unM[1]\to  i_{p,*}P$ is adjoint to $i_{p}^{*}\unM[1]\to P$. Since the stalk $i_{p}^{*}\unM[1]$ lies in degree $-1$, this map must be zero. 

For (2), let $u:S''\incl S'$ be the open locus where $\unM$ is lisse, and consider the natural map of perverse sheaves $\alpha: u_{!}u^{*}\unM[1]\to\unM[1]$. Since the cokernel of $\alpha$ is punctual, and by (1) $\unM[1]$ does not map to punctual sheaves, $\alpha$ must be surjective. Since the kernel of $\alpha$ is also punctual, it is contained in the maximal punctual subsheaf $\ker(u_{!}u^{*}\unM[1]\to u_{!*}u^{*}\unM[1])$. Therefore we can write $\unM[1]$ as an extension of the middle extension perverse sheaf $\unN[1]=u_{!*}u^{*}\unM[1]$ and a punctual sheaf as in (2).

Now let $\unM'=\DD(\unM^{V}_{!})|_{S'}[-2](-w-2)$. By Verdier duality, the perverse sheaf $\unM'[1]$ satisfies the dual of the two properties (1) and (2): $\unM'[1]$ does not admit any nonzero map from punctual sheaves and it fits into an exact sequence
\begin{equation*}
0\to\unN'[1]\to\unM'[1]\to\unP'\to0
\end{equation*}
where $\unN'\in\ME(S',\Ql)$ and $\unP'$ is punctual. 

Consider the morphism $\unM[1]\to\unM'[1]$ given by \eqref{MDM}. The properties (1)(2) of $\unM[1]$ and the dual properties of $\unM'$ imply that $\unM[1]\to\unM'[1]$ necessarily factors as
\begin{equation*}
\unM[1]\surj\unN[1]\xrightarrow{\gamma}\unN'[1]\incl\unM'[1].
\end{equation*}
Therefore $\unM^{V}_{!*}\cong\Im(\unN\xrightarrow{\gamma[-1]}\unN')$ is a middle extension sheaf, since both $\unN$ and $\unN'$ are.
\end{proof}

\subsubsection{Proof of Theorem \ref{th:main}\eqref{ram}}\label{Pf ram} Let $p\neq\ell$. The condition that $\Kl^{V}_{\hatG}$ over $\GG_{m,\Fpbar}$ does not contain a trivial sub local system is equivalent to $V^{\hatG^{\geom}_{p}}=0$. For such a prime $p$,  $\Spec\FF_{p}$ is contained in $S'$. Consider the perfect pairing $\jiao{,}^{V}$ \eqref{perfect pairing Q} of the $\GQp$-module $M:=M^{V}_{!*,\QQ}$. Let $\calI_{p}<\GQp$ be the inertia group. A general fact says that the restriction of the pairing $\jiao{,}^{V}$ to $\calI_{p}$-invariants may not be perfect, but it factors through a {\em perfect} pairing between $\barM:=\Im(M^{\calI_{p}}\to M_{\calI_{p}})$:
\begin{equation}\label{MI to MII}
\xymatrix{M^{\calI_{p}}\otimes M^{\calI_{p}}\ar@<-3ex>@{->>}[d]\ar@<3ex>@{->>}[d] \ar[r] &\Ql(-w-1)\ar@{=}[d]\\
 \barM\otimes\barM\ar[r]^{\textup{perfect}} & \Ql(-w-1)}
\end{equation}
Here the vertical maps are the quotient map $M^{\calI_{p}}\surj \barM$.

By Proposition \ref{p:midext}, $\unM^{V}_{!*}|_{S'}$ is a middle extension sheaf, hence $i_{p}^{*}\unM^{V}_{!*}\cong M^{\calI_{p}}$ by Lemma \ref{l:ME}\eqref{ME stalk}. On the other hand, taking the mod $p$ stalks of the diagram \eqref{two pairings} and using Lemma \ref{l:unM stalks}, we get another diagram
\begin{equation}\label{M! to MI}
\xymatrix{M^{V}_{!,\FF_{p}}\otimes M^{V}_{!,\FF_{p}}\ar@<-3ex>[d]\ar@<3ex>[d] \ar[r] & \Ql(-w-1)\ar@{=}[d]\\
(M^{\calI_{p}}\otimes\Ql(\mu_{p}))\otimes_{\Ql(\mu_{p})} (M^{\calI_{p}}\otimes\Ql(\mu_{p}))\ar[r] & \Ql(-w-1)}
\end{equation}
Combining \eqref{MI to MII} and \eqref{M! to MI}, we get a commutative diagram of pairings
\begin{equation*}
\xymatrix{M^{V}_{!,\FF_{p}}\otimes M^{V}_{!,\FF_{p}}\ar@<-3ex>[d]\ar@<3ex>[d] \ar[r] &\Ql(-w-1)\ar@{=}[d]\\
(\barM\otimes\Ql(\mu_{p}))\otimes_{\Ql(\mu_{p})}(\barM\otimes\Ql(\mu_{p}))\ar[r]^(.7){\textup{perfect}} & \Ql(-w-1)}
\end{equation*}
In other words, the natural map $M^{V}_{!,\FF_{p}}\to\barM\otimes\Ql(\mu_{p})$ kills the radical $R$ of the pairing on $M^{V}_{!,\FF_{p}}$.
On the other hand, we also know from Remark \ref{r:factor} that the same pairing on $M^{V}_{!,\FF_{p}}$ factors through a {\em perfect} pairing on the quotient $M^{V}_{!*,\FF_{p}}$, identifying $M^{V}_{!*,\FF_{p}}$ as $M^{V}_{!,\FF_{p}}/R$. Therefore we get the desired inclusion \eqref{inv coinv}.

We have the following immediate corollary of Theorem \ref{th:main}\eqref{ram}. Recall from Lemma \ref{l:const dim} that $\dim M^{V}_{!*, \FF_{p}}$ is a constant for $p$ large; we denote this dimension by $d^{V}_{!*}$. 

\begin{cor}\label{c:unram} Assumptions as in Theorem \ref{th:main}. Let $B$ be the set of primes $p$ such that either $V^{\hatG^{\geom}_{p}}\neq0$ or $\dim_{\Ql(\mu_{p})} M^{V}_{!*,\FF_{p}}\neq d^{V}_{!*}$. Then the system of Galois representations $\{\rho^{V}_{\ell}\}$ is Frobenius-compatible (same as the notion of ``strictly compatible'' in \cite[\S2.3]{SerreGal}) with exceptional set $B$. More precisely, when $p\notin B\cup\{\ell\}$, the Galois representation $\rho^{V}_{\ell}$ is unramified at $p$, and $M^{V}_{!*,\FF_{p}}$ is isomorphic to $M^{V}_{!*,\QQ}\otimes\Ql(\mu_{p})$ as $\Frob_{p}$-modules. The characteristic polynomial of $\rho^{V}_{\ell}(\Frob_{p})$ has $\ZZ$-coefficients which are independent of $\ell$, and all of its roots are $p$-Weil numbers of weight $w+1$.
\end{cor}

\subsection{Proof of Theorem \ref{th:main}\eqref{motivic}}

\subsubsection{Weight filtration on Galois representations} By \cite[\S6.1.1]{WeilII}, we may talk about mixed sheaves over varieties $X_{\QQ}$ defined over $\QQ$. A mixed $\Ql$-sheaf $\calF$ over $X_{\QQ}$ is a sheaf that extends to a sheaf $\un{\calF}$ on an integral model $\un{X}$ of $X_{\QQ}$ over $\ZZ[1/N]$ (for some $N$) such that $\un{\calF}$ admits a filtration with pure subquotients on each special fiber of $\un{X}$. In particular we may talk about mixed $\Ql$-sheaves over $\Spec\QQ$, i.e., continuous $\GQ$-modules on $\Ql$-vector spaces with weight filtrations. In \cite[\S6.1]{WeilII}, Deligne proved that the category of mixed sheaves on varieties over $\QQ$ are stable under the usual sheaf-theoretic operations. In particular, for a mixed sheaf $\calF$ over $X_{\QQ}$, its cohomology $\cohoc{i}{X_{\Qbar},\calF}$ and $\cohog{i}{X_{\Qbar},\calF}$ are $\GQ$-modules with canonical weight filtrations.

Notation: in the sequel, when we write $\cohog{i}{X,\Ql}$ or $\cohoc{i}{X,\Ql}$ for an algebraic variety $X$ over $\QQ$, we always mean $\cohog{i}{X_{\Qbar},\Ql}$ or $\cohoc{i}{X_{\Qbar},\Ql}$, as $\GQ$-modules.

%Before proving the theorem, we first make a remark on the weight filtration of a representation of $\GQ$. A $\GQ$-module $M$ (which is a finite dimensional $\Ql$-vector space), unramified at almost all primes $p$, is called pure of weight $w$ if for almost all primes $p$, the action of the geometric Frobenius $\Frob_{p}$ on   $M$ is pure of weight $w$ in the sense of Deligne (i.e., eigenvalues have absolute value $p^{w/2}$ for all embedding $\Ql\incl\CC$). 
%
%A weight filtration on a $\GQ$-module $M$ is an increasing filtration
%\begin{equation*}
%0=W_{-m}M\subset\cdots\subset W_{-1}M\subset W_{0}M\subset W_{1}M\subset\cdots\subset W_{n}M=M
%\end{equation*}
%such that $\Gr^{W}_{i}M$ is pure of weight $i$. 
%
%
%A weight filtration is unique if it exists. All morphisms between $\GQ$-modules are strictly compatible with weight filtrations (if they are defined). A subquotient of a $\GQ$-module which carries a weight filtration also carries a weight filtration.
%
%We say $M$ has weight $\leq n$ (resp. $\geq n$) if $M$ carries a weight filtration with $W_{n}M=M$ (resp. $W_{n-1}M=0$).
%
%If $M$ comes from geometry, i.e., it is a Tate twist of a subquotient of the $\ell$-adic cohomology $\cohog{n}{X_{\Qbar},\Ql}$ or $\cohoc{n}{X_{\Qbar},\Ql}$ for some algebraic variety $X$ over $\QQ$, then $M$ carries a weight filtration. In fact, under the comparison isomorphism $\cohog{n}{X(\CC),\QQ}\otimes\Ql\isom\cohog{n}{X_{\Qbar},\Ql}$ we may define the weight filtration on $\cohog{n}{X_{\Qbar},\Ql}$ by the transport of the weight filtration on the Betti cohomology $\cohog{n}{X(\CC),\QQ}$.

\subsubsection{Proof of Theorem \ref{th:main}\eqref{motivic}}\label{Pf mot} 
Up to Tate twist we may assume $\IC_{V}=\IC_{\l}$, which is pure of weight $w=\jiao{2\rho,\l}$. To simplify notation, we denote $\IC_{\l,\GR}$ by $\IC_{\l}$, and denote $\IC_{\l,\GRc}$ by $\ICc_{\l}$. 

Applying Theorem \ref{th:Fourier}\eqref{F:sp} to the object $\barF_{!}\IC^{\c}_{\l}$, we get a long exact sequence 
\begin{equation}\label{F10}
\cdots\to\cohoc{0}{F^{-1}(1),\ICc_{\l}}\to\upH^{0}M^{V}_{!,\QQ}\to \cohoc{1}{F^{-1}(0),\ICc_{\l}}\to\cdots
\end{equation}
Every object in this sequence carries a weight filtration. Since $\IC^{\c}_{\l}$ is pure of weight $w$, $F_{!}\ICc_{\l}$ is of weight $\leq w$. This implies that $\cohog{1}{F^{-1}(0),\ICc_{\l}}$ has weight $\leq w+1$ and $\cohog{0}{F^{-1}(1),\ICc_{\l}}$ has weight $\leq w$. Since $M^{V}_{!*,\QQ}$ is pure of weight $w+1$ (Corollary \ref{c:unram}), the exact sequence \eqref{F10} implies that $M^{V}_{!*,\QQ}$ (as a quotient of $\upH^{0}M^{V}_{!,\QQ}$) can be identified with a subspace of $\Gr^{W}_{w+1}\cohoc{1}{F^{-1}(0),\ICc_{\l}}$.

The complex $\ICc_{\l}$ is the intersection cohomology complex of $\GRc_{\leq\l}$, which is an open subset of $\GR_{\leq\l}$. We have  a $T\times\Grot$-equivariant resolution of singularities $\nu:Y=Y_{1}\times\Gm\to \Gr_{\leq\l}\times\Gm=\GR_{\leq\l}$ given by the Bott-Samelson resolution of $\Gr_{\leq\l}$. Let $\nu^{\c}:Y^{\c}\to\GRc_{\leq\l}$ be the restriction of $\nu$ to $\GRc_{\leq\l}$, which is again $T\times\Grot$-equivariant. By the decomposition theorem \cite[Th\'eor\`eme 6.2.5]{BBD}, $\ICc_{\l}$ is geometrically (i.e., over $\Qbar$) a direct summand of $\pH^{0}\nu^{\c}_{*}\Ql[w]$. It is necessarily invariant under $\GQ$ because it is the only direct summand of $\pH^{0}\nu^{\c}_{*}\Ql[w]$ with full support. Therefore $\cohoc{1}{F^{-1}(0), \ICc_{\l}}$ is a direct summand of $\cohoc{1}{F^{-1}(0), \pH^{0}\nu^{\c}_{*}\Ql[w]}$. On the other hand, the Leray spectral sequence attached to the perverse filtration on $\nu^{\c}_{*}\Ql[w]$ degenerates at the $E_{1}$ page by the decomposition theorem again, therefore $\cohoc{1}{F^{-1}(0), \pH^{0}\nu^{\c}_{*}\Ql[w]}$ is in turn a subquotient of $\cohoc{1}{F^{-1}(0),\nu^{\c}_{*}\Ql[w]}$. Let $Z=(F\circ\nu^{\c})^{-1}(0)$, then  $\cohoc{1}{F^{-1}(0),\nu^{\c}_{*}\Ql[w]}=\cohoc{w+1}{Z,\Ql}$. At this point we have shown that $M^{V}_{!*,\QQ}$ is a subquotient of $\Gr^{W}_{w+1}\cohoc{w+1}{Z,\Ql}$ as $\GQ$-modules.

Since $F\circ\nu^{\c}:Y^{\c}\to \AA^{1}$ is dominant and $Y^{\c}$ is smooth irreducible of dimension $w+1$, we have $\dim Z=w$. Since  $F\circ\nu^{\c}$ is $\Gm^{(\rho^\vee,h)}$-equivariant, the torus $\Gm^{(\rho^\vee,h)}$ still acts on $Z=(F\circ\nu^{\c})^{-1}(0)$. Moreover, we have a projection $\pi_{Z}:Z\subset Y^{\c}\subset Y_{1}\times\Gm\to\Gm$ which is equivariant under the homomorphism $\Gm^{(\rho^\vee,h)}\to\Grot$. Let $Z_{1}=\pi_{Z}^{-1}(1)$, then $\dim Z_{1}=w-1$ and $Z=Z_{1}\twtimes{\mu_{h}}\Gm^{(\rho^\vee,h)}$. By K\"unneth formula, we then have
\begin{eqnarray}\notag
&&\Gr^{W}_{w+1}\cohoc{w+1}{Z,\Ql}\cong\Gr^{W}_{w+1}\cohoc{w+1}{Z_{1}\times\Gm^{(\rho^\vee,h)},\Ql}^{\mu_{h}}\\
\label{Kun} &\cong&\Gr^{W}_{w+1}\cohoc{w}{Z_{1},\Ql}^{\mu_{h}}\otimes\Gr^{W}_{0}\cohoc{1}{\Gm}\oplus\Gr^{W}_{w-1}\cohoc{w-1}{Z_{1},\Ql}^{\mu_{h}}\otimes\Gr^{W}_{2}\cohoc{2}{\Gm}.
\end{eqnarray}
Since $\cohoc{*}{Z_{1},\Ql}$ is the stalk at $0$ of the complex $(F\circ\nu^{\c}_{1})_{!}\Ql$ (where $\nu^{\c}_{1}:Y_{1}^{\c}\to \Grc_{\leq\l}$ is the resolution), which has weight $\leq0$, therefore $\cohoc{w}{Z_{1}, \Ql}$ has weight $\leq w$. Therefore the first summand in \eqref{Kun} is zero, and we have
\begin{equation*}
\Gr^{W}_{w+1}\cohoc{w+1}{Z,\Ql}\cong\Gr^{W}_{w-1}\cohoc{w-1}{Z_{1},\Ql}^{\mu_{h}}(-1).
\end{equation*}
At this point, we have shown that $M^{V}_{!*,\QQ}$ is a subquotient of $\Gr^{W}_{w-1}\cohoc{w-1}{Z_{1},\Ql}(-1)$ as $\GQ$-modules, where $Z_{1}$ is an algebraic variety over $\QQ$ of dimension $w-1$. 

Let $\barZ_{1}$ be the closure of $Z_{1}$ in $Y_{1}$, which is projective and defined over $\QQ$. We then have $\Gr^{W}_{w-1}\cohoc{w-1}{Z_{1},\Ql}\incl\Gr^{W}_{w-1}\cohog{w-1}{\barZ_{1},\Ql}$. Finally let $X\to\barZ_{1}$ be a resolution of singularities over $\QQ$, so that $X$ is smooth and projective of dimension $w-1$, then $\Gr^{W}_{w-1}\cohog{w-1}{\barZ_{1},\Ql}\incl\cohog{w-1}{X,\Ql}$ (see \cite[Proposition 8.2.5]{HodgeIII}). Combining the previous steps we have shown that $M^{V}_{!*,\QQ}$ is a subquotient of $\cohoc{w-1}{X,\Ql}(-1)$ as $\GQ$-modules, where $X$ is a smooth and projective of dimension $w-1$ over $\QQ$. This finishes the proof of Theorem \ref{th:main}\eqref{motivic}.

%%%%%%%

\section{Symmetric power moments of $\Kl_{2}$}
In this section we shall prove  Evans's conjectures. We start by giving more precise information on the motives underlying the symmetric power moments of the classical Kloosterman sheaves in \S\ref{bifilter}. In \S\ref{ss:dim p>2} and \S\ref{ss:dim p=2}, we make explicit calculations on moments of $\Kl_{2}$, which are necessary for the proof of Theorem \ref{th:intro} in \S\ref{pf Kl2}. Evans's conjectures are proved in \S\ref{ss:Sym5}-\S\ref{ss:Sym7}.

\subsection{The motives underlying the symmetric power moments}\label{bifilter}
In this subsection, we explicitly describe the algebraic varieties over $\ZZ$ whose $\ell$-adic cohomology accommodates the Galois representations attached to the symmetric power moments of $\Kl_{n}$. 

\subsubsection{Bifiltered vector bundles} We shall work in the slightly different context where $G=\GL_{n}$ instead of $\PGL_{n}$. We first describe the moduli stacks $\Bun_{n}(0,i):=\Bun_{\GL_{n}(0,i)}$ (as recalled in \S\ref{alt Gr}) more explicitly in this case.

The stack $\Bun_{n}(0,0)$ classifies triples $(\calE,F^*\calE,F_*\calE)$ where
\begin{enumerate}
\item $\calE$ is a vector bundle of rank $n$ on $\PP^1$;
\item a decreasing filtration $F^*\calE$ giving a complete flag of the fiber of $\calE$ at $\infty$:
\begin{equation*}
\calE=F^0\calE\supset F^1\calE\supset\cdots\supset F^n\calE=\calE(-\{\infty\});
\end{equation*}
\item an increasing filtration $F_*\calE$ giving a complete flag of the fiber of $\calE$ at $0$: 
\begin{equation*}
\calE(-\{0\})=F_0\calE\subset F_1\calE\subset\cdots\subset F_n\calE=\calE;
\end{equation*}
\end{enumerate}

\begin{remark}\label{ext fil}Given a triple $(\calE,F^*\calE,F_*\calE)$ as above, we can extend the filtration $F_{*}\calE$ to all integer indices such that $F_{i+n}\calE=F_{i}\calE\otimes\calO(\{0\})$. Similarly we can extend the filtration $F^{*}\calE$ to all integer indices such that $F^{i+n}\calE=F^{i}\calE\otimes\calO(-\{\infty\})$.
\end{remark}

The moduli stack $\Bun_{n}(0,1)$ classifies the data $(\calE,F^*\calE,F_*\calE, \{\barv_{i}\}_{i=1}^{n})$ where $(\calE,F^*\calE,F_*\calE)$ is as above and $\barv_{i}$ is a basis of $F^{i-1}\calE/F^{i}\calE$ for $i=1,2,\cdots, n$.

The moduli stack  $\Bun_{n}(0,2)$ classifies the data $(\calE,F^*\calE,F_*\calE, \{\tilv_{i}\}_{i=1}^{n})$ where $(\calE,F^*\calE,F_*\calE)$ is as above and $\tilv_{i}\in F^{i-1}\calE/F^{i+1}\calE$  is a vector whose projection to $F^{i-1}\calE/F^{i}\calE$ is nonzero ($i=1,2,\cdots, n$).

\subsubsection{Open points}The stack $\Bun_n(0,i)$ is decomposed into connected components $\Bun^d_n(0,i)$ according to the integer $d=\deg(\calE)$. For each component of $\Bun^{d}_{n}(0,1)$, there is a unique open point $(\calE_d,F^*,F_*,\{\ebar_{i}\})$ which we recall now (see also \cite[\S3.1]{HNY}).  

When $d=0$, we take $\calE_0=\calO e_1\oplus\cdots\oplus\calO e_n$. The symbols $e_1,\cdots, e_n$ are incorporated only to emphasize the ordering of the various factors. Define $F^i\calE_0=\calO(-\{\infty\})e_1\oplus\cdots\oplus\calO(-\{\infty\})e_i\oplus\calO e_{i+1}\oplus\cdots\oplus \calO e_n$ for $i=0,\cdots,n-1$, and extend this definition to $F^i\calE_0$ for all $i\in\ZZ$ as in Remark \ref{ext fil}. Define $F_i\calE_0=\calO e_1\oplus\cdots\oplus\calO e_i\oplus\calO(-\{0\})e_{i+1}\oplus\cdots\oplus\calO(-\{0\}) e_n$ for $i=0,\cdots, n-1$ and extend this definition to $F_{i}\calE_0$ for all $i\in\ZZ$ as in Remark \ref{ext fil}. This finishes the definition of $(\calE_0,F^*,F_*)$.

For arbitrary integer $d$, we define $\calE_d=F_{n+d}\calE_0$ and $F_i\calE_d=F_{i+d}\calE_0$. Therefore we have canonical isomorphism between $\calE_d$ and $\calE_0$ over $\PP^1-\{0\}$. We define $F^i\calE_d|_{\PP^1-\{0\}}$ to be the image of $F^i\calE_0$ under this isomorphism, and let $F^i\calE_d$ be the vector bundle obtained by gluing $F^i\calE_d|_{\PP^{1}-\{0\}}$ and $\calE_d|_{\PP^1-\{\infty\}}$ over $\Gm$. 

The basis vectors $\{e_i\}$ give a canonical basis $\ebar_i\in F^{i-1}\calE_0/F^{i}\calE_0=F^{i-1}\calE_d/F^{i}\calE_d$. It also gives a vector $\tile_i\in F^{i-1}\calE_0/F^{i+1}\calE_0=F^{i-1}\calE_d/F^{i+1}\calE_d$.  

For each $d\geq0$, let $\frG_{d}$ be the space of homomorphisms $\phi:\calE_{0}\to\calE_{d}$ of coherent sheaves such that
\begin{itemize}
\item $\phi(F^i\calE_0)\subset F^i\calE_d$ and $\phi(F_i\calE_0)\subset F_i\calE_d$ for all $i\in\ZZ$.
\item $\phi(\ebar_i)=\ebar_{i}$, under the natural identification $F^{i-1}\calE_0/F^{i}\calE_0=F^{i-1}\calE_d/F^{i}\calE_d$. 
\item The zeros of $\phi$ are concentrated at a single point in $\Gm$.
\end{itemize}

\subsubsection{Torus action} Let $T_n=\Gm^n$ be the diagonal torus of $\GL_n$. The object $(\calE_d,F^*,F_*)$ admits an action of $T_n$ by $(t_1,\cdots, t_n)\cdot e_{i}=t_{i}e_i$. Let $s$ be the local coordinate at $\infty\in\PP^{1}$. The torus $\Grot$ acts on $\PP^{1}$ by scaling $s$. We normalize the equivariant structure of $\calE_0$ by making $\Grot$ acting trivially on $e_i$, which then induces a canonical $\Grot$-equivariant structure on all $\calE_d$ and their filtrations.

The torus $T_n$ acts on $\frG^{d}$ by $t\cdot \phi=t\circ \phi \circ t^{-1}$, where the  $t$ (resp. $t^{-1}$) refers the action of $T_n$ on $\calE_0$ (resp. $\calE_d$). This action factors through the adjoint torus $T^{\ad}_{n}\subset\PGL_{n}$. The $\Grot$-equivariant structures on both $(\calE_0,F^*,F_*)$ and $(\calE_d,F^*,F_*)$ induce an $\Grot$-action on $\frG^{d}$, which commutes with the action of $T_n$.

Let $\phi:\calE_{0}\to\calE_{d}$ be a point in $\frG^{d}$. For $i=1,\cdots, n-1$ we can write $\phi(\tile_i)=\tile_{i}+a_{i}\ebar_{i+1}\mod F^{i+1}\calE_d$ for some $a_{i}\in\AA^{1}$. For $i=n$, $\phi(\tile_{n})=\tile_{n}+a_{n}s\ebar_{1}\mod F^{i+1}\calE_{d}$ for some $a_{n}\in\AA^{1}$. The assignment $\phi\mapsto(a_1,\cdots,a_n)$ defines a morphism $f=f_{n,d}:\frG^{d}\to\AA^n$. A direct calculation shows that the map $f_{n,d}$ is $T_n\times\Grot$-equivariant.

We thus get a diagram which is analogous to \eqref{GR} and \eqref{Kloo}:
\begin{equation*}
\xymatrix{& \frG^{d}\ar[dl]_{\pi}\ar[dr]_{f_{n}^{d}}\ar[drr]^{F_{n}^{d}}\\ \Gm & & \AA^n \ar[r]^{\sigma}& \AA^1}
\end{equation*}
Here $\pi$ takes $\phi\in\frG^{d}$ to the support of the zeros of $\phi$, which is a unique point in $\Gm$ by assumption on $\phi$. The schemes in the above diagram are defined over $\QQ$.

To emphasize the dependence on $n$ and on the prime $\ell$, we introduce the notation
\begin{equation}\label{define M2}
M^{d}_{n,\ell}:=M^{\Sym^{d}}_{!*,\QQ}
\end{equation} 
where the right side is the  generic stalk of the object defined in Definition \ref{define !*Q} for the $\Sym^{d}$-moment of the usual Kloosterman sheaf $\Kl_{n}$.

%Definition \ref{def:moments Q} and  , i.e., $\hatG=\SL_{n}$. We write $M^{d}_{n,?,\QQ}$ for the generic stalk of $\unM^{d}_{n,?}$.

%The $\Ql$-vector space $M^{d}_{n,!*,\QQ}$ is equipped with  a continuous action of $\GQ$:
%\begin{equation*}
%\rho^{d}_{n}:\GQ\to\GL(M^{d}_{n,!*,\QQ}).
%\end{equation*}

\begin{prop}\label{p:mot dn} Let $Z_{0}=F^{d,-1}_{n}(0)$ and $Z_{1}=F^{d,-1}_{n}(1)$. These are affine varieties of dimension $d(n-1)$ over $\QQ$.
\begin{enumerate}
\item The $\GQ$-module $M^{d}_{n,\ell}$ fits into a short exact sequence
\begin{equation}\label{Md Gr}
0\to M^{d}_{n,\ell}\to\Gr^{W}_{d(n-1)+1}\cohoc{d(n-1)+1}{Z_{0,\Qbar},\Ql}\xrightarrow{\Sp} \Gr^{W}_{d(n-1)+1}\cohoc{d(n-1)+1}{Z_{1,\Qbar},\Ql}\to0.
\end{equation}
where the map $\Sp$ is induced from the specialization map $\Sp:\cohoc{*}{Z_{0},\Ql}\to\cohoc{*}{Z_{1},\Ql}$.
\item Define a $\QQ$-Hodge structure $M^{d}_{n,\Hod}$ of weight $d(n-1)+1$ to be the kernel of the specialization map
\begin{equation*}
\Sp_{\Hod}:\Gr^{W}_{d(n-1)+1}\cohoc{d(n-1)+1}{Z_{0}(\CC),\QQ}\to\Gr^{W}_{d(n-1)+1}\cohoc{d(n-1)+1}{Z_{1}(\CC),\QQ}.
\end{equation*}
Then the multiset of Hodge-Tate weights of  $M^{d}_{n,\ell}$ (viewed as a $\GQl$-module) is the same as the multiset where $i\in\ZZ$ appears $\dim\Gr_{F}^{-i}M^{d}_{n,\Hod}$ (where $F$ is the Hodge filtration on $M^{d}_{n,\Hod}\otimes\CC$). In particular, the multiset of Hodge-Tate weights of $M^{d}_{n,\ell}$ is independent of $\ell$.

%\item Define $Z^{d}_{n}=F^{d,-1}_{n}(0)\cap\pi^{-1}(1)\subset\frG^{d}$, which is an affine variety of dimension $d(n-1)-1$ on which $\mu_{n}$ acts via $\rho^{\vee}$. Then we have an embedding of $\GQ$-modules
%\begin{equation*}
%M^{d}_{n,!*,\QQ}\incl\Gr^{W}_{d(n-1)-1}\cohoc{d(n-1)-1}{Z^{d}_{n},\Ql}^{\mu_{n}}(-1).
%\end{equation*}
\end{enumerate}
 \end{prop}
\begin{proof}

(1)  The proof is a more refined version of the first part of the argument in \S\ref{Pf mot}. By Lemma \ref{l:unM stalks}, Corollary \ref{c:unram} and Lemma \ref{l:MV pure}, for large $p$ we may identify the mod $p$ stalk of $\unM^{\Sym^{d}}_{!*}$ with the weight $d(n-1)+1$-quotient of the mod $p$ stalk of $\unM^{\Sym^{d}}_{!}$. Therefore we have
\begin{equation*} 
M^{d}_{n,\ell}=M^{\Sym^{d}}_{!*,\QQ}=\Gr^{W}_{d(n-1)+1}M^{\Sym^{d}}_{!,\QQ}.
\end{equation*}

The morphism $F^{d}_{n}$ is equivariant under the action of $\Gm^{(\rho^{\vee},n)}\subset T_{n}\times\Grot$ on $\frG^{d}$ and the dilation action on $\AA^{1}$, hence it descends to a morphism of stacks
\begin{equation*}
\barF^{d}_{n}:[\frG^{d}/\Gm^{(\rho^{\vee},n)}]\to[\AA^{1}/\Gm].
\end{equation*}
The scheme $\frG^{d}$ is naturally a subscheme of $\frG$ in the context of \S\ref{alt Gr} for the group $G=\PGL_{n}$. When $V=\Sym^{d}$, the intersection complex $\IC_{V,\GRc}$ is the shifted constant sheaf $\Ql[d(n-1)]$ on $\frG^{d}$. Therefore by Definition \ref{def:moments Q}, we have
\begin{equation*}
M^{\Sym^{d}}_{!,\QQ}\cong\HF(\barF^{d}_{n,!}\Ql)_{\eta}[d(n-1)].
\end{equation*}
By Theorem \ref{th:Fourier}\eqref{F:sp}, we get an exact sequence
\begin{equation}\label{seq Z}
\cohoc{d(n-1)}{Z_{1,\Qbar},\Ql}\to M^{\Sym^{d}}_{!,\QQ}\to\cohoc{d(n-1)+1}{Z_{0,\Qbar},\Ql}\to\cohoc{d(n-1)+1}{Z_{1,\Qbar},\Ql}\to0.
\end{equation}
Here the maps $\cohoc{*}{Z_{0,\Qbar},\Ql}\to\cohoc{*}{Z_{0,\Qbar},\Ql}$ are the maps called ``$\Sp$'' in Theorem \ref{th:Fourier}\eqref{F:sp}. Since $\Ql[d(n-1)]$ is the same as the intersection complex on $\frG^{d}$, $\cohoc{i}{Z_{\alpha,\Qbar},\Ql}$ is of weight $\leq i$ for all $i$ and $\alpha=0$ or $1$. Taking the weight $d(n-1)+1$ part of the sequence \eqref{seq Z}, we get the short exact sequence \eqref{Md Gr}.

(2) Fix an embedding $\Ql\incl\CC$. We introduce an abelian category $\calH$ consisting of triples $(H_{\ell},H_{\Hod},c)$ where $H_{\ell}$ is a de Rham representation of $\GQl$ over $\Ql$, $H_{\Hod}$ is a $\QQ$-Hodge structure and $c$ is an isomorphism compatible with Hodge filtrations on both sides
\begin{equation*}
c: H_{\Hod}\otimes_{\QQ}\CC\cong (H_{\ell}\otimes_{\Ql}B_{\dR})^{\GQl}\otimes_{\Ql}\CC.
\end{equation*}
The morphisms in $\calH$ are required to respect the isomorphisms $c$.

Since both specialization maps $\Sp$ and $\Sp_{\Hod}$ are surjective, to show (2), we only need to show that, for $\alpha=0$ and $1$, we can complete the pair $(\Gr^{W}_{d(n-1)+1}\cohoc{d(n-1)+1}{Z_{\alpha,\Qbar},\Ql}, \Gr^{W}_{d(n-1)+1}\cohoc{d(n-1)+1}{Z_{\alpha}(\CC),\QQ})$ into an object in $\calH$, i.e., providing a comparison between the pure weight pieces of the \'etale and singular cohomologies (after tensoring with period rings) which preserve Hodge filtrations. In fact, we will show a stronger statement. For any variety $Z$ over $\QQ$, and any $m,w\in\ZZ$, we know that the $\GQl$-module $\Gr^{W}_{w}\cohoc{m}{Z_{\Qbar},\Ql}$ (in fact the whole $\cohoc{m}{Z_{\Qbar},\Ql}$) is de Rham by a theorem of Kisin \cite[Theorem 3.2]{Kisinp}. We will show that there is an isomorphism of $\CC$-vector spaces equipped with Hodge filtrations
\begin{equation}\label{Hodge ell}
(\Gr^{W}_{w}\cohoc{m}{Z(\CC),\QQ})\otimes_{\QQ}\CC\cong (\Gr^{W}_{w}\cohoc{m}{Z_{\Qbar},\Ql}\otimes_{\Ql}B_{\dR})^{\GQl}\otimes_{\Ql}\CC.
\end{equation}
In other words, we will complete the pair $(\Gr^{W}_{w}\cohoc{m}{Z_{\Qbar},\Ql}), \Gr^{W}_{w}\cohoc{m}{Z(\CC),\QQ}))$ into an object $\Gr^{W}_{w}h^{m}_{c}(Z)\in\calH$.

For $Z$ proper smooth over $\QQ$, the isomorphism \eqref{Hodge ell} follows from Faltings's theorem \cite{Faltings} and classical Hodge theory: both sides of \eqref{Hodge ell} can be identified with the algebraic de Rham cohomology of $Z$ (with scalars extended to $\CC$) equipped with the Hodge filtration. We now have the object $h^{m}(Z)\in\calH$.

In general, when $Z$ is not necessarily smooth, one can find a proper hypercovering $\{Z_{n}\}$ of $Z$ with each $Z_{n}$ a smooth variety over $\QQ$, and a simplicial compactification $\{Z_{n}\}\incl\{X_{n}\}$ with normal crossing divisors $\{D_{n}=X_{n}-Z_{n}\}$, as in \cite[\S6.2.8]{HodgeIII}.  For each $i\geq0$, let $\tilD_{n,i}$ be the normalization of closed subscheme of $X_{n}$ where at least $i$ local components of $D_{n}$ meet. Then $\tilD_{n,i}$ is proper smooth over $\QQ$ and $\tilD_{n,0}=X_{n}$. According to the $\upH^{*}_{c}$ variant of \cite[Proposition 8.1.20(2)]{HodgeIII}, $\Gr^{W}_{w}\cohoc{m}{Z(\CC),\QQ}$ is the $(m-w)\nth$ cohomology of the simple complex associated with the double complex $C^{a,b}_{\Hod}=\cohog{w}{D_{a,b}(\CC),\QQ}$ (the differentials are given by pullback maps). The same argument shows that $\Gr^{W}_{w}\cohoc{m}{Z_{\Qbar},\Ql}$ is the $(m-w)\nth$ cohomology of the simple complex associated with the double complex $C^{a,b}_{\ell}=\cohog{w}{D_{a,b,\Qbar},\Ql}$. Since $D_{a,b}$ is proper smooth over $\QQ$, the pair $(C^{a,b}_{\ell},C^{a,b}_{\Hod})$ can be completed into the object $h^{w}(D_{a,b})\in\calH$. This way we get a double complex $C^{*,*}_{\calH}$ with $C^{a,b}_{\calH}=h^{w}(D_{a,b})$ in the category $\calH$. The $(m-w)\nth$ cohomology of the simple complex associated with $C^{*,*}_{\calH}$ is also an object in $\calH$, giving the desired object $\Gr^{W}_{w}h^{m}_{c}(Z)\in\calH$.

\end{proof}

\subsubsection{The case $n=2$}
When $d=2k+1$, $k\geq0$, $\frG^{d}$ classifies maps $\phi:\calO e_{1}\oplus\calO e_{2}\to \calO(k)e_{1}\oplus\calO(k+1)e_{2}$ such that $\phi_{0}(e_{1})\in\Span\{e_{2}\}$, and $\phi_{\infty}(e_{1})=e_{1}+\Span\{e_{2}\}, \phi_{\infty}(e_{2})=e_{2}$. Using the local coordinate $s$ at $\infty$, every point of $\frG^{d}$ is uniquely represented by a matrix
\begin{equation*}
\phi=\left(\begin{array}{cc} 1+x_{1}s+\cdots+x_{k}s^{k} & y_{1}s+\cdots+y_{k+1}s^{k+1}\\
z_{0}+z_{1}s+\cdots+z_{k}s^{k} & 1+w_{1}s+\cdots+w_{k}s^{k}\end{array}\right)
\end{equation*}
subject to the equation
\begin{equation*}
(1+\cdots+x_{k}s^{k})(1+\cdots+w_{k}s^{k})-(y_{1}s+\cdots+y_{k+1}s^{k+1})(z_{0}+\cdots+z_{k}s^{k})=(1+\frac{x_{1}+w_{1}-y_{1}z_{0}}{2k+1}s)^{2k+1}.
\end{equation*}
and an open condition
\begin{equation}\label{not zero}
x_{1}+w_{1}-y_{1}z_{0}\neq0.
\end{equation}

When $d=2k$, $k\geq0$, we can similarly write every point $\phi\in\frG^{d}$ as a matrix
\begin{equation*}
\phi=\left(\begin{array}{cc} 1+x_{1}s+\cdots+x_{k}s^{k} & y_{1}s+\cdots+y_{k}s^{k}\\
z_{0}+z_{1}s+\cdots+z_{k-1}s^{k-1} & 1+w_{1}s+\cdots+w_{k}s^{k}\end{array}\right)
\end{equation*}
subject to the equation
\begin{equation*}
(1+\cdots+x_{k}s^{k})(1+\cdots+w_{k}s^{k})-(y_{1}s+\cdots+y_{k}s^{k})(z_{0}+\cdots+z_{k-1}s^{k-1})=(1+\frac{x_{1}+w_{1}-y_{1}z_{0}}{2k}s)^{2k}.
\end{equation*}
and the same open condition as \eqref{not zero}.

In both cases, $\frG^{d}$ is a $d+1$-dimensional affine subvariety of $\AA^{2d}$ defined by imposing $d-1$ equations and removing one divisor \eqref{not zero}. The action of $t\in T^{\ad}_{2}$ on $\frG^{d}$ sends $y_{i}\mapsto t^{-1}y_{i}$ and $z_{i}\mapsto tz_{i}$. The action of $t\in\Grot$ on $\frG^{d}$ sends $(?)_{i}\mapsto t^{i}(?)_{i}$, where $(?)=x,y,z$ or $w$. The morphism $f_{n,d}$ is given by
\begin{equation*}
\phi\mapsto (z_{0},y_{1}).
\end{equation*}

\subsection{Dimension of moments: $p>2$}\label{ss:dim p>2} From now on, till the end of the paper, we shall restrict our consideration to $\Kl_{2}$. We introduce the notation
\begin{equation*}
M^{d}_{2,?,\FF_{p}}:=M^{\Sym^{d}}_{?, \FF_{p}} \textup{ for }?=!,*\textup{ or }!*.
\end{equation*}
Here the right side is understood in the context $\hatG=\SL_{2}$ and Definition \ref{def:moments} for the base field $k=\FF_{p}$.

We would like to give explicit formula for the dimension of $M^{d}_{2,?}$ for $?=!,*$ and $!*$. Let $V=\Sym^{d}$. According to Theorem \ref{th:global mono Katz}, we have $V^{\hatG^{\geom}_{p}}=0$ for {\em all} primes $p$. Therefore $M^{V}_{2,!}$ and $M^{V}_{2,*}$ are both concentrated in degree zero, and the exact sequence \eqref{!* short} simplifies to
\begin{eqnarray}\label{!* short Kl2}
0\to V^{\calI_{0}}\oplus V^{\calI_{\infty}}\to M^d_{2,!,\FF_{p}}\to M^d_{2,*,\FF_{p}}\to V_{\calI_{0}}(-1)\oplus V_{\calI_{\infty}}(-1)\to0.
\end{eqnarray}

We first calculate $V^{\calI_{0}}$ and $V^{\calI_{\infty}}$ as $\Frob_{p}$-modules.
\begin{lemma}[Special case of {\cite[Theorem 7.3.2(3)]{Katz}}]\label{Frob action at 0} For all primes $p$, we have $V^{\calI_{0}}\cong\Ql$ as $\Frob_{p}$-modules.
\end{lemma}

\begin{lemma}\label{local inv} Let $p>2$ be a prime. As $\Frob_{p}$-modules
\begin{equation*}
V^{\calI_{\infty}}\cong\begin{cases}\Ql(-d/2)^{[\frac{d}{2p}]+1} & d\equiv0\mod 4\\
\Ql(-d/2)^{[\frac{d}{2p}]}& d\equiv2\mod 4\\
0 & d \textup{ odd.}\end{cases}
\end{equation*}
\end{lemma}
\begin{proof} By the discussion in \S\ref{loc infty}, we may assume that the image of the wild inertia $\calI^{w}_{\infty}$ under the monodromy representation of $\Kl_{2}$ is contained in the diagonal torus $\hatT_{1}\subset\SL_{2}$. Then the wild inertia $\calI^{w}_{\infty}$ surjects onto $\hatT_{1}[p]$. The Coxeter element acts as $\left(\begin{array}{cc} 0 & 1 \\ -1 & 0\end{array}\right)\in\SL_{2}$. Let $\phi\in\GL_{2}=\hatG^{\Tt}$ be the image of a lifting of $\Frob_{\infty}$. Since $\phi$ normalizes the image of $\calI^{w}_{\infty}$, we must have $\phi\in N_{\GL_{2}}(\hatT_{1})=N_{\GL_{2}}(\hatT)$ where $\hatT$ is the diagonal torus in $\GL_{2}$. Up to multiplying $\phi$ by the Coxeter element (changing the lifting of the Frobenius) we may assume that $\phi\in\hatT$. From \S\ref{sss:det} we know that $\det\phi=p$, therefore $\phi=\left(\begin{array}{cc} \sqrt{p}\alpha & 0 \\ 0 & \sqrt{p}\alpha^{-1}\end{array}\right)$ for some $\alpha\in\Qlbar^{\times}$. The commutation relation between $\Frob_{\infty}$ and $\calI^{t}_{\infty}$ forces that 
\begin{equation*}
\left(\begin{array}{cc} \sqrt{p}^{-1}\alpha^{-1} & 0 \\ 0 & \sqrt{p}^{-1}\alpha\end{array}\right)\left(\begin{array}{cc} 0 & 1 \\ -1 & 0\end{array}\right)\left(\begin{array}{cc} \sqrt{p}\alpha & 0 \\ 0 & \sqrt{p}\alpha^{-1}\end{array}\right)\equiv\left(\begin{array}{cc} 0 & 1 \\ -1 & 0\end{array}\right)^{p}\mod \hatT_{1}[p].
\end{equation*}
This implies that $\alpha\in\pm\mu_{p}$.

Under the action of $\hatT_{1}\cong\Gm$, we have a weight decomposition
\begin{equation*}
V=\bigoplus_{-d\leq i\leq d, i\equiv d\mod2}V(i).
\end{equation*}
We have 
\begin{equation*}
V^{D_{1}}=\bigoplus_{-d\leq jp\leq d, j\equiv d\mod 2}V(jp).
\end{equation*}

When $d$ is odd, $\Cox^{2}$ acts as $-1$ on $V^{D_{1}}$ and hence $V^{\calI_{\infty}}=0$. 

When $d$ is even, the action of $\phi=\sqrt{p}\diag(\alpha,\alpha^{-1})$ on $V(jp)$ is via $p^{d/2}\alpha^{jp}$. Since $\alpha^{2p}=1$ and $2p| jp$, we conclude that $\phi$ acts on $V^{D_{1}}$ by  the scalar $p^{d/2}$.

It remains to calculate the dimension of $V^{\calI_{\infty}}=V^{D_{1},\Cox}$ when $d$ is even. The Coxeter element acts on $V$ as an involution, and it permutes the factors $V(j)$ and $V(-j)$ for $j\neq0$, and acts on $V(0)$ by $(-1)^{d/2}$. Therefore
\begin{equation*}
\dim V^{D_{1},\Cox}=\begin{cases}[\frac{d}{2p}]+1 & d\equiv0\mod4;\\
[\frac{d}{2p}]& d\equiv2\mod4;\\ 0& d\textup{ odd}\end{cases}
\end{equation*}

\end{proof}

Applying Lemma \ref{Sw general} and Lemma \ref{local inv} we get
\begin{cor}[Fu-Wan {\cite{FW}}]\label{c:M dim p} For $p>2$, we have
\begin{eqnarray*}
\dim M^{d}_{2,!,\FF_{p}}=\dim M^{d}_{2,*,\FF_{p}}=\begin{cases}\frac{d}{2}-[\frac{d}{2p}] & d\textup{ even}\\ \frac{d+1}{2}-[\frac{d}{2p}+\frac{1}{2}] & d\textup{ odd}\end{cases};\\
\dim M^{d}_{2,!*,\FF_{p}}=\begin{cases}\frac{d}{2}-2[\frac{d}{2p}]-2 & d\equiv0\mod4\\
\frac{d}{2}-2[\frac{d}{2p}]-1 & d\equiv2\mod4\\
\frac{d-1}{2}-[\frac{d}{2p}+\frac{1}{2}] & d\textup{ odd}\end{cases};.
\end{eqnarray*}
\end{cor}

We shall also use the following formula for the determinant of the $\Frob$-action on $M^{d}_{2,!*,\FF_{p}}$, proved by Fu-Wan \cite[Theorem 0.1]{FW2}.

\begin{theorem}[Fu-Wan{\cite{FW2}}]\label{th:det} We have
\begin{equation*}
p^{-\frac{d+1}{2}\dim M^{d}_{2,!*,\FF_{p}}}\det(\Frob,M^{d}_{2,!*,\FF_{p}})=\begin{cases}1 & d\textup{ even, } \\
\left(\frac{-2}{p}\right)^{[\frac{d}{2p}+\frac{1}{2}]}\prod_{0\leq j\leq \frac{d-1}{2}, p\nmid 2j+1}\left(\frac{(-1)^{j}(2j+1)}{p}\right) & d\textup{ odd.}\end{cases}
\end{equation*}
\end{theorem}

\subsection{Dimension of moments: $p=2$}\label{ss:dim p=2} 

\subsubsection{Inertia image at $\infty$ when $p=2$} Let $\ell$ be an odd prime, and $\rho:\calI_\infty\to\SL_2(\Ql)$ be the restriction of the monodromy representation of $\Kl_2$ over $\GG_{m,\FF_2}$ to the inertia at $\infty$. Let $D_{0}=\rho(\calI_{\infty})<\SL_{2}(\Ql)$. We first determine the image $\Dbar_{0}$ of $\rho(\calI_\infty)$ in $\PGL_2(\Ql)$. Note that $\Dbar_{0}$ is a finite subgroup of $\PGL_{2}$ which has a normal Sylow 2-subgroup. According to the classification of Platonic groups, we have the following possibilities
\begin{enumerate}
\item $\Dbar_{0}$ is cyclic ;
\item $\Dbar_{0}$ is dihedral of order $2^{k}$, $k>1$;
\item $\Dbar_{0}\cong A_{4}$. 
\end{enumerate}

Suppose $\Dbar_{0}$ is cyclic, so is $D_{0}$. Therefore $\rho$ decomposes into a sum of two characters  which are inverse to each other. In this case $\Swan(\Kl_{2})$  must be an even number and cannot be one. Thus we can eliminated the first possibility.

Suppose $\Dbar_{0}$ is dihedral of order $2^{k}$ for some $k>1$, then up to conjugacy $D_{0}=\mu_{2^{k}}\cdot\jiao{w_{0}}$ where $\mu_{2^{k}}$ is embedded into the diagonal torus of $\SL_{2}$ and $w_{0}=\left(\begin{array}{cc} 0 & 1 \\ -1 & 0\end{array}\right)$. Then $\Sym^{2}(\rho)$ can be decomposed as $V_{2}\oplus V_{1}$ with $\dim V_{2}=2$, $\Swan(V_{2})>0$, $\dim V_{1}=1$ on which $D_{0}$ acts nontrivially through its quotient $D_{0}/\mu_{2^{k}}\cong\jiao{w_{0}}$. This implies that $\Swan(\Sym^{2}(\Kl_{2}))\geq2$. However, Katz proved \cite[Proposition 10.4.1]{Katz} that $\Swan(\Kl_{2}\otimes\Kl_{2})=1$. This is a contradiction. Therefore we can eliminated the second possibility as well.

In conclusion, we must have $\Dbar_{0}\cong A_{4}$. This also implies that $D_{0}\cong\wt{A}_4$, the preimage of $A_4<\PGL_2(\Ql)$ in $\SL_2(\Ql)$. Let $K_4<A_4$ be the Klein group of order 4. The preimage of $K_4$ in $\wt{A}_4$ is the quarternion group $Q_8=\{\pm 1,\pm I, \pm J,\pm K\}$. The lower numbering filtration of $D_0=\rho(\calI_\infty)\cong\wt{A}_4$ is a decreasing filtration
\begin{equation*}
D_0\rhd D_1\rhd D_2 \rhd\cdots\rhd D_m=\{1\}.
\end{equation*}
Here $D_1=\rho(\calI^w_\infty)=Q_8$, each quotient $D_i/D_{i+1}$ is an elementary abelian 2-group, and $m$ is the smallest integer such that $D_m$ is trivial. Then the Swan conductor of $\rho$ (or of  $\Kl_{2}$ at $\infty$) is
\begin{equation*}
\Swan(\Kl_{2})=\sum_{j\geq1}\frac{2-\dim \rho^{D_j}}{[D_0:D_j]}=\frac{2}{3}(1+\sum_{j=2}^{m-1}\frac{1}{[D_1:D_j]}).
\end{equation*}
Here we use the fact that $\rho^{D_{m-1}}=0$ because $\rho$ is irreducible. Since $\Swan(\Kl_{2})=1$, $D_1=Q_8$, we have $\sum_{j=2}^{m-1}\#D_j=4$. Note each $\#D_j=2$ or $4$. There are only two cases:
\begin{enumerate}
\item $\#D_2=\#D_3=2,D_4=\{1\}$.
\item $\#D_2=4,D_3=\{1\}$.
\end{enumerate}
The second case is impossible because the order four subgroups of $Q_8$ are not elementary abelian 2-groups. We are left with the first possibility, and the lower numbering filtration looks like
\begin{equation*}
D_0=\wt{A}_4\rhd D_1=Q_8 \rhd D_2=D_3=\{\pm1\} \rhd D_4=\{1\}.
\end{equation*}  

\subsubsection{Image of Frobenius}\label{Frobenius p=2} Finally we determine the image of Frobenius at $\infty$. Let $\phi\in\GL_{2}$ be the image of a lifting of the geometric Frobenius $\Frob_{\infty}$. By \S\ref{sss:det} we may write $\phi=\sqrt{2}\phi_{1}$ where $\phi_{1}\in\SL_{2}$. Since $\phi_{1}$ normalizes $\wt{A}_{4}$, it lies in $N_{\SL_{2}}(\wt{A}_{4})=\wt{S}_{4}$, the preimage of $S_{4}=N_{\PGL_{2}}(A_{4})$ in $\SL_{2}$. The tame quotient of the local Galois group then maps to $\wt{S}_{4}/Q_{8}\cong S_{3}$. The image of a generator of the tame inertia is a cyclic permutation $C\in S_{3}$. The relation $\phi^{-1}C\phi=C^{2}$ forces the image of $\phi$ in $S_{3}$ to be a transposition. In particular, $\phi_{1}\in\wt{S}_{4}-\wt{A}_{4}$.

With the preparations above, we may now do calculations.
\begin{lemma} When $p=2$, the Swan conductor of the action of $\calI_\infty$ on $V=\Sym^d$ is given by
\begin{equation}\label{Sw p=2}
\Swan(\Sym^d)=\begin{cases}
\frac{d+1}{2} & d\textup{ odd};\\
\left[\frac{d+2}{4}\right] & d\textup{ even}.\end{cases}.
\end{equation}
\end{lemma}
\begin{proof}
Use the standard basis $\{x^ry^{d-r}\}_{r=0,\cdots,d}$ for $\Sym^d$. We extend scalars from $\Ql$ to $\Ql(i)$ where $ i=\sqrt{-1}$. Recall the quarternion group $Q_{8}=\{\pm1,\pm I,\pm J,\pm K\}$ acts on the standard 2-dimensional representation $\Span\{x,y\}$ as
\begin{equation*}
I: (x,y)\mapsto(ix,-iy);  \hspace{5pt} J: (x,y)\mapsto(-y,x).
\end{equation*}
When $d$ is odd, the action of $D_3=\{\pm1\}$ on $\Sym^d$ is by $-1$. Therefore $D_1,D_2,D_3$ do not have invariants on $\Sym^d$. We get
\begin{equation}\label{p=2 d=odd}
\Swan(\Sym^d)=\sum_{j=1}^{3}\frac{d+1}{[D_0:D_j]}=\frac{d+1}{2}.
\end{equation}
When $d$ is even, the action of $\calI_\infty$ on $\Sym^d$ factors through $A_4<\PGL_2(\Ql)$. Therefore $D_{2}$ acts trivially on $\Sym^{d}$ and $D_{1}$ acts via its quotient $K_{4}$. In terms of a standard basis $x^iy^{d-i}$ for $\Sym^d$, $I: x^iy^{d-i}\mapsto (-1)^{d/2+i}x^iy^{d-i}$ and $J: x^iy^{d-i}\mapsto (-1)^ix^{d-i}y^i$. Therefore $\dim(\Sym^d)^{K_4}=d/4+1$ if $4|d$ and $\dim(\Sym^d)^{K_4}=[d/4]$ if $d\equiv2\mod4$ . The Swan conductor is
\begin{equation}\label{p=2 d=even}
\Swan(\Sym^d)=\frac{1}{3}(d+1-\dim(\Sym^d)^{K_4})=\left[\frac{d+2}{4}\right].
\end{equation}
\end{proof}

\begin{lemma}\label{l:local inv p=2} When $p=2$ and $V=\Sym^{d}$, then we have an isomorphism of $\Frob$-modules
\begin{equation}\label{Frob mod infty p=2}
V^{\calI_{\infty}}\cong\begin{cases}    0 & d\textup{ odd};  \\
\Ql(d/2)^{[d/24]+1}\oplus\sgn(d/2)^{[d/24]}&  d\equiv0\mod8;\\
\Ql(d/2)^{[d/24]}\oplus\sgn(d/2)^{[d/24]+1}&  d\equiv6\mod8;\\
\Ql(d/2)^{[d/24]}\oplus\sgn(d/2)^{[d/24]} & d\equiv2,4,10\mod24;\\
\Ql(d/2)^{[d/24]+1}\oplus\sgn(d/2)^{[d/24]+1} & d\equiv12,18,20\mod24.\\
\end{cases}
\end{equation}
Here $\sgn$ stands for the one-dimensional $\Frob$-module on which $\Frob$ acts as $-1$.
\end{lemma}
\begin{proof}
We have $V^{\calI_{\infty}}=(\Sym^{d})^{\wt{A}_{4}}$, and we would like to calculate $\dim(\Sym^{d})^{\wt{A}_{4}}$ and $\Tr(\phi,(\Sym^{d})^{\wt{A}_{4}})$ (where $\phi$ is the image of a lift of $\Frob_{\infty}$ as we discussed in \S\ref{Frobenius p=2}).

We first make a general remark about the calculation. If $G$ is a finite group with a normal subgroup $N\lhd G$. Let $V$ be a representation of $G$, and let $\bara\in G/N$. Then $\bara$ acts on $V^{N}$ and its trace is given by
\begin{equation*}
\Tr(\bara,V^{N})=\frac{1}{\#N}\sum_{x\in\bara N}\Tr(x,V).
\end{equation*}
Applying this remark to $N=\wt{A}_{4}\lhd G=\wt{S}_{4}$, taking $\bara=1$ we get
\begin{equation*}
\sum_{d\geq0}\dim(\Sym^d)^{\wt{A}_4}t^d=\frac{1}{24}\sum_{g\in\wt{A}_4}\frac{1}{1-\Tr(g)t+t^2},
\end{equation*}
Here we write $\Tr(g)$ for the trace of $g\in\wt{A}_4$ under the standard two-dimensional representation. Similarly, taking $\bara=\phi_{1}$ (note $\phi=\sqrt{2}\phi_{1}$ where $\phi_{1}\in\wt{S}_{4}$), we get
\begin{equation*}
\sum_{d\geq0}\Tr(\phi_{1},(\Sym^d)^{\wt{A}_4})t^d=\frac{1}{24}\sum_{g\in \phi\wt{A}_4}\frac{1}{1-\Tr(g)t+t^2},
\end{equation*}

We may identify $\wt{S}_4$ with the quarternions
\begin{eqnarray}\label{list 1}
&&\pm1,\pm I,\pm J,\pm K, \frac{1}{2}(\pm 1\pm I\pm J\pm K),\\
\label{list 2}
&&\frac{1}{\sqrt{2}}(\pm 1\pm I), \frac{1}{\sqrt{2}}(\pm 1\pm J), \frac{1}{\sqrt{2}}(\pm 1\pm K), \frac{1}{\sqrt{2}}(\pm I\pm J), \frac{1}{\sqrt{2}}(\pm J\pm K), \frac{1}{\sqrt{2}}(\pm K\pm I).
\end{eqnarray}
The line \eqref{list 1} consists of elements in $\wt{A}_{4}$. Since $\Tr(\pm1)=\pm2$, $\Tr(\pm I)=\Tr(\pm J)=\Tr(\pm K)=0$, we see that one element in $\wt{A}_{4}$ have trace 2, one has trace $-2$, 8 has trace $1$, 8 has trace $-1$ and the rest 6 has trace zero. Therefore
\begin{eqnarray*}
\sum_{d\geq0}\dim(\Sym^d)^{\wt{A}_4}t^d&=&\frac{1}{24}\left(\frac{1}{(1-t)^2}+\frac{1}{(1+t)^2}+\frac{6}{1+t^2}+\frac{8}{1-t+t^2}+\frac{8}{1+t+t^2}\right)\\
&=&\frac{1}{12}\sum_{j\geq0}(2j+1)t^{2j}+\frac{1}{4}\sum_{j\geq0}(-1)^jt^{2j}+\frac{2}{3}\sum_{j\geq0}(t^{6j}-t^{6j+4})\end{eqnarray*}
Computing coefficients, we get
\begin{equation}\label{local inv p=2}
\dim V^{\calI_{\infty}}=\begin{cases}0 & d\textup{ odd}; \\
[\frac{d}{12}]+1 & d\equiv 0,6,8\mod12\\
[\frac{d}{12}] & d\equiv 2,4,10\mod12.\end{cases}
\end{equation}

Since $\phi_{1}\notin\wt{A}_{4}$, the coset $\phi_{1}\wt{A}_{4}$ consists of the second line \eqref{list 2} of the above list. We see that 6 of these elements have trace $\sqrt{2}$, 6 of them have trace $-\sqrt{2}$ and the rest have trace  zero. Therefore,
\begin{eqnarray*}
\sum_{d\geq0}\Tr(\phi_{1},\dim(\Sym^d)^{\wt{A}_4})t^d&=&\frac{1}{24}\left(\frac{6}{1-\sqrt{2}t+t^{2}}+\frac{6}{1+\sqrt{2}t+t^{2}}+\frac{12}{1+t^2}\right)\\
&=&\frac{1}{2}\left(\sum_{j\geq0}(-1)^{j}t^{4j}+\sum_{j\geq0}(-1)^{j}t^{4j+2}+\sum_{j\geq0}(-1)^jt^{2j}\right).
\end{eqnarray*}
Therefore
\begin{equation*}
\Tr(\phi_{1},\dim(\Sym^d)^{\wt{A}_4})=\begin{cases}1 & d\equiv0\mod8;\\
-1 & d\equiv6\mod8;\\
0 & \textup{otherwise}.\end{cases}
\end{equation*}
Since $\phi=\sqrt{2}\phi_{1}$, we get
\begin{equation}\label{Frob infty p=2}
\Tr(\Frob_{\infty},V^{\calI_{\infty}})=\begin{cases}2^{d/2} & d\equiv0\mod8; \\
-2^{d/2} & d\equiv6\mod8;\\
0 & \textup{otherwise}.\end{cases}
\end{equation}
Finally, since $\phi^{2}_{1}$ lies in the image of $\calI_{\infty}$, it acts trivially on $V^{\calI_{\infty}}$. Therefore the eigenvalues of $\Frob=\sqrt{2}\phi_{1}$ on $V^{\calI_{\infty}}$ are $\pm2^{d/2}$. Combining information from \eqref{local inv p=2} and from \eqref{Frob infty p=2} we get the multiplicities of the eigenvalues $\pm2^{d/2}$, and hence the formula \eqref{Frob mod infty p=2}.
\end{proof}

Combining the previous lemmas we get
\begin{cor} When $p=2$, $d>1$, we have
\begin{eqnarray*}
\dim M^{d}_{2,!,\FF_{2}}=\dim M^{d}_{2,*,\FF_{2}}=\begin{cases}\frac{d+1}{2}& d\textup{ odd}\\
[\frac{d+2}{4}] & d\textup{ even}\end{cases};\\
\dim M^{d}_{2,!*,\FF_{2}}=\begin{cases} \frac{d-1}{2} & d\textup{ odd}\\
2[\frac{d+2}{12}] & d\equiv2,4,6,8,10\mod12;\\
\frac{d}{6}-2& d\equiv0\mod12.\end{cases}.
\end{eqnarray*}
\end{cor}

In the following table, we summarize the dimension of $M^{d}_{2,!*, \FF_{p}}$ for first few primes $p$.

\begin{tabular}{|c|c|c|c|c|c|c|c|c|c|c|c|c|c|}
\hline
$d$        & 1 & 2 & 3 & 4 & 5 & 6 & 7 & 8 & 9 & 10 & 11 & 12 & 13\\
\hline
good $p$   & 0 & 0 & 1 & 0 & 2 & 2 & 3 & 2 & 4 & 4 & 5 & 4 & 6\\
\hline
duality    &   &   & + &   & + & - & + & - & + & - & + & - & +\\
\hline
$p=2$      &&& good & & good & 0 & good & 0 & good & 2 & good & 0 & good\\
\hline
$p=3$	   &&&	0 &   & 1 & 0 & 2 & 0 & 2 & 2 & 3 & 0 & 4\\
\hline
$p=5$      &&& good & & 1 & good & 2 & good & 3 & 2 & 4 & 2 & 5\\
\hline
$p=7$      &&& good & & good & good & 2 & good & 3 & good & 4 & good & 5\\
\hline
$p=11$      &&& good & & good & good & good & good & good & good & 4 & good & 5\\
\hline
$p=13$      &&& good & & good & good & good & good & good & good & good & good & 5\\
\hline
\end{tabular}

\subsection{Proof of Theorem \ref{th:intro}}\label{pf Kl2} 
We recall the notation $M^{d}_{2,\ell}$ from \eqref{define M2}, which is equipped with the continuous $\GQ$-action
\begin{equation*}
\rho^{d}_{2,\ell}:\GQ\to\GL(M^{d}_{2,\ell}).
\end{equation*}

From Corollary \ref{c:M dim p}, we see that the dimension of $M^{d}_{2,!*,\FF_{p}}$ for large $p$ is
\begin{equation}\label{stable dim}
\dim M^{d}_{2,!*,\FF_{p}}=\begin{cases} 2[\frac{d+2}{4}]-2 &  d \textup{ even}, p>d/2;\\ \frac{d-1}{2} & d \textup{ odd}, p>d \textup{ or } p=2. \end{cases}
\end{equation}
This gives the dimension of the $\GQ$-module $M^{d}_{2,\ell}$.

\subsubsection{When $d$ is odd} Let $M_{\ell}:=M^{d}_{2,\ell}(\frac{d+1}{2})$. It is equipped with the orthogonal pairing
\begin{equation*}
M_{\ell}\otimes M_{\ell}\to\Ql
\end{equation*}
induced from the orthogonal (since $(-1)^{d+1}=1$) pairing $M^{d}_{2,\ell}\otimes M^{d}_{2,\ell}\to\Ql(-d-1)$ from Theorem \ref{th:main}. We then get an orthogonal Galois representation
\begin{equation*}
\rho_{\ell}=\rho^{d}_{2,\ell}\chi_{\cyc,\ell}^{(d+1)/2}:\GQ\to\Og(M_{\ell})\cong \Og_{(d-1)/2}(\Ql)
\end{equation*}
We can compute $\det(\rho_{\ell})$ using Theorem \ref{th:det}. For $p>d$, we have
\begin{equation*}
\det(\rho_{\ell}(\Frob_{p}))=p^{-\frac{d+1}{2}\frac{d-1}{2}}\det(\Frob_{p},M^{d}_{2,!*,\FF_{p}})=\left(\frac{(-3)5(-7)\cdots(\pm d)}{p}\right)=\left(\frac{p}{d!!}\right).
\end{equation*}
Therefore $\det(\rho_{\ell})$ is the quadratic Dirichlet character $\left(\frac{\cdot}{d!!}\right)$. In particular, $\rho_{\ell}$ is ramified at the odd primes $p\leq d$.

Finally, by Lefschetz trace formula and the sequence \eqref{!* short Kl2} we have
\begin{equation*}
-m^{d}_{2}(p)=\Tr(\Frob, M^{d}_{2,!,\FF_{p}})=\Tr(\Frob_{0},V^{\calI_{0}})+\Tr(\Frob_{\infty}, V^{\calI_{\infty}})+\Tr(\Frob, M^{d}_{2,!*,\FF_{p}}).
\end{equation*}
For all primes $p$ we have $\Tr(\Frob_{0},V^{\calI_{0}})=1$ by Lemma \ref{Frob action at 0}, and $V^{\calI_{\infty}}=0$ by Lemma \ref{local inv} and Lemma \ref{l:local inv p=2}. Therefore
\begin{equation*}
-m^{d}_{2}(p)=1+\Tr(\Frob, M^{d}_{2,!*,\FF_{p}}).
\end{equation*}
For $p>d, p\neq\ell$, Corollary\ref{c:unram} together with \eqref{stable dim} implies that $\rho_{\ell}$ is unramified, and that $M^{d}_{2,!*,\FF_{p}}$ is isomorphic to $M^{d}_{2,\ell}$ as $\Frob_{p}$-modules. Therefore, $\Tr(\Frob, M^{d}_{2,!*,\FF_{p}})$ is equal to $p^{(d+1)/2}\Tr(\Frob_{p},M_{\ell})$. Combining these facts we get the formulae \eqref{md2 odd}.

\subsubsection{When $d$ is even} Let $M_{\ell}:=M^{d}_{2,\ell}$. It is equipped with the symplectic (since $(-1)^{d+1}=-1$) pairing
\begin{equation*}
M_{\ell}\otimes M_{\ell}\to \Ql(-d-1)
\end{equation*}
from Theorem \ref{th:main}. We then get a symplectic Galois representation
\begin{equation*}
\rho_{\ell}=\rho^{d}_{2,\ell}:\GQ\to\GSp(M_{\ell})\cong\GSp_{2[(d+2)/4]-2}(\Ql)
\end{equation*}
with the  similitude character equal to $\chi_{\cyc}^{-d-1}$. Corollary \ref{c:unram} together with \eqref{stable dim} implies that $\rho_{\ell}$ is unramified at primes $p>d/2$, $p\nmid\ell$.

\subsubsection{Proof of \eqref{md2 even ram} and \eqref{md2 even}}\label{sss:ram Kl2} The argument for \eqref{md2 odd} works as well for $d$ even. In particular, we know that $\rho_{\ell}$ is unramified for $p>d/2, p\neq\ell$, in which case $M^{d}_{2,!*,\FF_{p}}$ is isomorphic to $M^{d}_{2,\ell}=M_{\ell}$ as $\Frob_{p}$-modules.

Let us deal with the general case $p>2$ ($p\neq\ell$ is always assumed). Let $\unM_{!}=\unM^{\Sym^{d}}_{!}$ be the sheaf on $S=\SZl$ defined in Definition \ref{def:moments Q} for the $\Sym^{d}$-moment of $\Kl_{2}$. In the proof of Proposition \ref{p:midext} we have shown that $\unM_{!}$ fits into an exact sequence
\begin{equation*}
0\to\unP[-1]\to\unM_{!}\to\unN\to0
\end{equation*}
where $\unP=\oplus_{p}i_{p,*}\delta_{p}$ is a punctual sheaf (the sum is finite) and $\unN$ is a middle extension sheaf on $S$. By definition, the middle extension sheaf $\unM_{!*}$ is a quotient of $\unN$. Let $\unK=\ker(\unN\to\unM_{!*})$, which is also a middle extension sheaf. Taking the stalk at a prime $p\neq\ell$, we get an exact sequence
\begin{equation*}
0\to i_{p}^{*}\unM_{!}\to i_{p}^{*}\unN\to\delta_{p}\to0.
\end{equation*}
By Lemma \ref{l:unM stalks}, we have $i_{p}^{*}\unM_{!}\cong M^{d}_{!,\FF_{p}}$. Recall the notation: for a sheaf $\unN$ over $S$ we use $N_{\QQ}$ to denote its generic stalk, which is a continuous $\GQ$-module. Since $\unN$ is a middle extension sheaf, we have $i_{p}^{*}\unN=N_{\QQ}^{\calI_{p}}$, which is an extension of $M_{\ell}^{\calI_{p}}$ by $K_{\QQ}^{\calI_{p}}$. For large enough $p$, $\delta_{p}=0$ and $i_{p}^{*}K$ is equal to $\ker(i^{*}_{p}\unM_{!}\to i_{p}^{*}\unM_{!*})=\ker(M^{d}_{2,!,\FF_{p}}\to M^{d}_{2,!*,\FF_{p}})=V_{p}^{\calI_{0}}\oplus V_{p}^{\calI_{\infty}}$. (Here we write $V_{p}^{\calI_{\infty}}$ etc. to emphasize its dependence on $p$). By Lemma \ref{Frob action at 0} and Lemma \ref{local inv}, we see that for $p$ large,
\begin{equation}\label{Frob unK}
\Tr(\Frob_{p},K_{\QQ})=1+\begin{cases}0& d\equiv2\mod4;\\ p^{d/2} & d\equiv0\mod4.\end{cases}
\end{equation}
Therefore, up to semisimplification, we have $K_{\QQ}\cong \Ql\oplus\Ql(-d/2)$ as Galois representations, which is unramified over $S$. In particular, \eqref{Frob unK} holds for {\em all} primes  $p\neq\ell$.

For every $p\neq\ell$, let $L_{p}$ be the kernel of the map $N_{\QQ}^{\calI_{p}}\surj M^{\calI_{p}}_{\ell}\to M_{\ell,\calI_{p}}$. Then $L_{p}$ is an extension of $\ker(M^{\calI_{p}}_{\ell}\to M_{\ell,\calI_{p}})$ by $K_{\QQ}$. The inclusion $i_{p}^{*}\unM_{!}\cong M^{d}_{!,\FF_{p}}\incl i_{p}^{*}\unN=N_{\QQ}^{\calI_{p}}$ fits into an exact sequence of maps
\begin{equation}\label{three inj}
\xymatrix{0\ar[r] & V^{\calI_{0}}_{p}\oplus V^{\calI_{\infty}}_{p}\ar@{^{(}->}[d]\ar[r] & M^{d}_{!,\FF_{p}}\ar@{^{(}->}[d]\ar[r] & M^{d}_{!*,\FF_{p}}\ar@{^{(}->}[d]\ar[r] & 0\\
0\ar[r] & L_{p}\ar[r] & N_{\QQ}^{\calI_{p}}\ar[r] & \Im(M^{\calI_{p}}_{\ell}\to M_{\ell,\calI_{p}})\ar[r] & 0}
\end{equation}
We claim that the vertical maps are all injective: the argument in already contained in the proof of Theorem \ref{th:main}\eqref{ram} in \S\ref{Pf ram}, namely for both rows the first term is the radical of the natural pairing on the middle term. When $p>2$, comparing the formula for $V_{p}^{\calI_{\infty}}$ in Lemma \ref{local inv} and the formula for $K_{\QQ}$ in \eqref{Frob unK}, we see that $\ker(M^{\calI_{p}}_{\ell}\to M_{\ell,\calI_{p}})=L_{p}/K_{\QQ}$ should contain at least $\Ql(-d/2)^{[\frac{d}{2p}]}$ as a $\Frob_{p}$-submodule. In other words, $M_{\ell}$ contains at least $[\frac{d}{2p}]$ unipotent Jordan blocks of size $\geq2$ under the action of $\calI_{p}$. We then have
\begin{equation*}
\dim M_{\ell}\geq 2[\frac{d}{2p}]+\dim \Im(M^{\calI_{p}}_{\ell}\to M_{\ell,\calI_{p}})\geq 2[\frac{d}{2p}]+\dim M^{d}_{2,!*,\FF_{p}}.
\end{equation*}
However, by Lemma \ref{c:M dim p}, the above inequalities have to be equalities. Equivalently, all vertical maps in \eqref{three inj} have to be isomorphisms. This means that for $p>2$ and $p\neq\ell$,  $\calI_{p}$ acts unipotently on $M_{\ell}$ with $[\frac{d}{2p}]$ Jordan blocks of size $2$ and trivial Jordan blocks elsewhere. Moreover, the action of $\Frob_{p}$ on $M^{\calI_{p}}_{\ell}$ is $\Ql(-d/2)^{[\frac{d}{2p}]}\oplus M^{d}_{2,!*,\FF_{p}}$. This proves the decomposition \eqref{md2 even ram} with $U_{p}\cong M^{d}_{2,!*,\FF_{p}}$ as $\Frob_{p}$-modules.

Finally, by Lefschetz trace formula and using the fact that $M^{d}_{2,!,\FF_{p}}=N^{\calI_{p}}_{\QQ}$, we have
\begin{equation*}
-m^{d}_{2}(p)=\Tr(\Frob, M^{d}_{2,!,\FF_{p}})=\Tr(\Frob_{p},N_{\QQ}^{\calI_{p}})=\Tr(\Frob_{p},K_{\QQ})+\Tr(\Frob_{p}, M_{\ell}^{\calI_{p}}).
\end{equation*}
Combined with \eqref{Frob unK}, we get \eqref{md2 even}.

In the following subsections, we shall study the cases $d=5,6,7$ and $8$ in more details.

\subsection{$\Sym^{5}$ of $\Kl_{2}$}\label{ss:Sym5}
For $d=5$, Theorem \ref{th:intro} gives  a Frobenius-compatible system of continuous Galois representations
\begin{equation*}
\rho_{\ell}=\rho^{5}_{2,\ell}\chi_{\cyc}^{3}:\GQ\to\Og_{2}(\Ql)
\end{equation*}
with determinant $\left(\frac{\cdot}{15}\right)$ and ramified at $3,5$ and $\ell$. 

At $p=3$ or $5$, according to Theorem \ref{th:main}\eqref{ram}, we have $\dim\Im(M^{\calI_{p}}_{\ell}\to M_{\ell,\calI_{p}})\geq \dim M^{5}_{2,!*,\FF_{p}}=1$. Since $\rho_{\ell}$ does ramify at $p$, the only possibility is that $\calI_{p}$ acts through its tame quotient and a generator of $\calI^{t}_{p}$ maps to a matrix conjugate to $\diag(1,-1)$.

The Hodge-Tate weights of $\rho^{5}_{2,\ell}$ lie in $\{-1,-2,\cdots,-5\}$ by Theorem \ref{th:main}\eqref{motivic}, hence the Hodge-Tate weights of $\rho_{\ell}$ are $(k,-k)$ where $k\in\{0,1,2\}$ (we see from Proposition \ref{p:mot dn} that $k$ is independent of $\ell$).

Let $K=\QQ(\sqrt{-15})$. The restriction $\rho_{\ell}|_{\Gal(\Qbar/K)}$ is an abelian representation into $\SO_{2}(\Ql)$ which is ramified only at $\ell$. If $k=0$, then $\rho_{\ell}|_{\Gal(\Qbar/K)}$ form a Frobenius-compatible family of Galois representations with finite image. We may choose $\ell$ such that $\rho_{\ell}|_{\Gal(\Qbar/K)}$ is everywhere unramified, which then factors through the class group $\textup{Cl}(\calO_{K})=\ZZ/2\ZZ$. But this possibility can be excluded by numerical calculation of $m^{5}_{2}(p)$ for the first few $p$.

Therefore $k=1$ or $2$, and the Frobenius-compatible family $\rho_{\ell}\chi_{\cyc}^{-k}:\GQ\to\GL_{2}(\Ql)$ then comes from a CM form of weight $2k+1$, level $\Gamma_{0}(15)$ and nebentypus $\left(\frac{\cdot}{15}\right)$ (using the converse theorem since the $L$-function of $\rho_{\ell}$ is the same as the $L$-function of the Galois character $\rho|_{\Gal(\Qbar/K)}$). The value of $k$ can again be checked by numerics.

\subsection{$\Sym^{6}$ and $\Sym^{8}$ of $\Kl_{2}$}\label{ss:Sym68} When $d=6$ or $d=8$, Theorem \ref{th:intro} gives a Frobenius-compatible system of continuous Galois representations
\begin{equation*}
\rho_{\ell}=\rho^{d}_{2,\ell}:\GQ\to \GL_{2}(\Ql).
\end{equation*}
with determinant $\det(\rho_{\ell})=\chi_{\cyc}^{-d-1}$, which are at most ramified at $2,3$ and $\ell$. Moreover at $p=3$, the action of $\GQp$ is via $J_{2}(-d/2)$. By Theorem \ref{th:main}\eqref{motivic}, the Hodge-Tate weights of $\rho_{\ell}|_{\GQl}$ are $(-\frac{d}{2}-1+k,-\frac{d}{2}-k)$ for $1\leq k\leq \frac{d}{2}$. 

%\begin{lemma}\label{l:irr8} For any $\ell>3$, $\rho_{\ell}$ is absolutely irreducible.
%\end{lemma}
%\begin{proof}
%If not, then there is a Dirichlet character $\chi$ of $\GQ$ such that $\rho_{\ell}$ is of the form $\left(\begin{array}{cc} \chi\chi_{\cyc}^{5-k} &  * \\ 0 & \chi^{-1}\chi_{\cyc}^{4+k}\end{array}\right)$ or $\left(\begin{array}{cc} \chi^{-1}\chi_{\cyc}^{4+k} &  * \\ 0 & \chi\chi_{\cyc}^{5-k}\end{array}\right)$. In any case $\Tr(\rho_{\ell}(\Frob_{p}))=p^{5-k}\chi(p)+\chi(p)^{-1}p^{4+k}$ must be a rational number. Taking its imaginary part we see that $(p^{5-k}-p^{4+k})\Im\chi(p)=0$, hence $\Im\chi(p)=0$ and $\chi(p)$ is always real for almost all $p$. This is possible only if $\chi$ is quadratic. Also $\chi$ is ramified at $2,3$, possibly at $\ell$ and nowhere else. Hence $\chi=\left(\frac{\cdot}{6}\right)$ or $\left(\frac{\cdot}{6\ell}\right)$. By Proposition \ref{p:Kl2} we have for all $p\geq 5$, $p\neq\ell$:
%\begin{equation*}
%-m^{8}_{2}(p)-1-p^{4}=\Tr(\rho_{\ell}(\Frob_{p}))=\pm(p^{5-k}+p^{4+k}).
%\end{equation*}
%We check by numerics that this is impossible.
%\end{proof}

\subsubsection{Proof of theorem \eqref{th:Sym8}--Reduction to finite calculation}\label{pf Sym8} We will prove part (1) of Theorem \ref{th:Sym8} which reduces the proof to a finite calculation. Set $d=8$, then $1\leq k\leq 4$ for all $\ell$. Let $\rho'_{\ell}=\rho_{\ell}\chi_{\cyc,\ell}^{5-k}$. Then the Hodge-Tate weights of $\rho'_{\ell}$ are $(0,1-2k)$.

We shall use the Serre conjecture, proved by Khare and Wintenberger \cite{KW}, to deduce the modularity of $\rho'_{\ell}$, following an argument originally due to Serre in \cite[\S4.8]{Serre}. One can find a similar argument from Kisin's survey article \cite[Theorem 1.4.3]{Kisin}. Let $\rhobar'_{\ell}$ be the reduction mod $\ell$ of $\rho'_{\ell}$, well-defined up to semisimplification. Let $L$ be the set of primes $\ell$ such that
\begin{itemize}
\item $\ell>7$.
\item The variety $X$ in Theorem \ref{th:main}\eqref{motivic} has good reduction at $\ell$, so that $\rho_{\ell}$ is crystalline at $\ell$. By Fontaine-Laffaille theory (note that $\ell>\dim X=7$, see \cite{FL}, see also \cite[Corollary 4.3]{BLZ}), the semisimplification of  $\rhobar'_{\ell}|_{\calI_{\ell}}$ is either of the form $1\oplus\omega^{1-2k}$ or $\omega'^{1-2k}\oplus\omega''^{1-2k}$. Here $\omega$ is the mod $\ell$ cyclotomic character and $\omega',\omega'':\calI_{\ell}\to\FF_{\ell^{2}}^{\times}$ are the level two fundamental characters. By the definition of Serre weights \cite[\S2]{Serre}, the Serre weight of $\rhobar'_{\ell}$ is equal to $2k$.
\item $\ell\not\equiv\pm1\mod8$. (See \cite[(4.8.7)]{Serre}.)
\end{itemize}
This set  $L$ is infinite. Strong Serre conjecture then implies that for each $\ell\in L$, $\rhobar'_{\ell}$ is modular of weight $2k$, level of the form $2^{e}3$ for some $0\leq e\leq 8$ (see the estimate of Artin conductors in \cite[Corollary to Proposition 9]{Serre} and \cite[Last paragraph of p.216]{Serre}; this is why wee need the third condition when defining $L$). We have $3$ in the level because $\GQp$ acts under $\rho_{\ell}$ by $J_{2}(-d/2)$. The finiteness of such modular forms then imply the modularity of the compatible system $\{\rho'_{\ell}\}$. For details, we refer to \cite[\S4.8]{Serre}.

In conclusion, there is a Hecke eigenform $f$ of weight $2k$, level $N_{f}=2^{e}3$ for some $0\leq e\leq8$ (and trivial nebentypus) such that the attached Galois representation $\rho_{f,\ell}$ is isomorphic to $\rho_{\ell}\chi_{\cyc,\ell}^{5-k}$ for every prime $\ell\in L$. Since $\rho_{\ell}\chi_{\cyc,\ell}^{5-k}$ form a Frobenius-compatible family, $\rho_{f,\ell}\cong \rho_{\ell}\chi_{\cyc,\ell}^{5-k}$ for all $\ell$.

\begin{remark} The same proof shows that the analog of Theorem \ref{th:Sym8} holds for $m^{6}_{2}(p)$, with the weight of the modular form now $2k$ with $1\leq k\leq 3$. We actually get more precise information at $p=2$ thanks to the fact $(\Sym^{6})^{\calI_{\infty}}\cong\sgn(-3)$ proved in Lemma \ref{l:local inv p=2}. Using this fact, the same arugment in \S\ref{sss:ram Kl2} allows us to show that $\rho_{\ell}|_{\Gal(\overline{\QQ}_{2}/\QQ_{2})}\cong J_{2}\otimes\sgn(-3)$, and in particular the modular form $f$ in question should have $\Gamma_{0}(6)$. Finally \eqref{m82} should change to
\begin{equation*}
-m^{6}_{2}(p)-1=p^{4-k}a_{f}(p) 
\end{equation*}
which holds for {\em all} primes $p$. 
\end{remark}

\subsection{$\Sym^{7}$ of $\Kl_{2}$}\label{ss:Sym7} Theorem \ref{th:intro} gives a Frobenius-compatible system of Galois representations
\begin{equation*}
\rho_{\ell}=\rho^{7}_{2,\ell}\chi_{\cyc,\ell}^{4}:\GQ\to\Og_{3}(\Ql)
\end{equation*}
with determinant $\left(\frac{\cdot}{105}\right)$ and ramified only at $3,5,7$ and $\ell$. By Theorem \ref{th:main}\eqref{motivic}, the Hodge-Tate weights of $\rho^{7}_{2,\ell}|_{\GQl}$ are in $\{-1,\cdots, -7\}$, hence the Hodge-Tate weights for $\rho_{\ell}|_{\GQl}$ are $(-k,0,k)$ for $0\leq k\leq 3$ independent of $\ell$ (see Proposition \ref{p:mot dn}(2)). 

We then define
\begin{equation*}
\rho'_{\ell}=\rho_{\ell}\left(\frac{\cdot}{105}\right):\GQ\to\SO_{3}(\Ql)\cong\PGL_{2}(\Ql).
\end{equation*}

For $p=3,5$ and $7$, Theorem \ref{th:main}\eqref{ram} implies that $\dim\Im(M_{\ell}^{\calI_{p}}\to M_{\ell,\calI_{p}})\geq\dim M^{7}_{2,!*,\FF_{p}}=2$. Since $\rho_{\ell}$ does ramify at $p$, this inequality must be an equality, the action of $\calI_{p}$ factors through its tame quotient, and a generator of $\calI^{t}_{p}$ maps under $\rho_{\ell}$ to a matrix conjugate to $\diag(1,1,-1)\in\Og_{3}(\Qlbar)$. Equivalently, a generator of $\calI^{t}_{p}$ maps under $\rho'_{\ell}$ to a matrix conjugate to $\zeta=\diag(-1,1)\in\PGL_{2}(\Qlbar)$. We fix a lifting  of $\Frob_{p}$ to $\GQp$ and let $\phi_{p}\in\PGL_{2}(\Qlbar)$ be its image under $\rho'_{\ell}$. Then the relation $\phi^{-1}_{p}\zeta\phi_{p}=\zeta^{p}=\zeta$ forces $\phi_{p}$ to lie in $N_{\PGL_{2}}(\zeta)=T\cup Tw$, where $T$ is the diagonal torus in $\PGL_{2}(\Qlbar)$ and $w=\left(\begin{array}{cc} 0 & 1 \\ 1 & 0\end{array}\right)$. 

We say $p$ is of {\em abelian type} if $\phi_{p}\in T$; we say $p$ is of {\em dihedral type} if $\phi_{p}\in Tw$.

\begin{lemma} $p=3$ and $7$ are of abelian type, and $p=5$ is dihedral type.
\end{lemma}
\begin{proof}
Whether $\phi_{p}\in T$ or $\phi_{p}\in Tw$ can be detected by computing the determinant of $\Frob_{p}$ on $M^{7}_{2,!*,\FF_{p}}$. In fact, under $\rho^{7}_{2,\ell}$, $\Frob_{p}$ acts on the two-dimensional orthogonal space $M^{\calI_{p}}_{\ell}$ with similitude $p^{4}$. If $\det(\Frob_{p}, M^{7}_{2,!*,\FF_{p}})=p^{8}$ then $\phi\in T$; if $\det(\Frob_{p}, M^{7}_{2,!*,\FF_{p}})=-p^{8}$ then $\phi\in Tw$. We then use Theorem \ref{th:det} to see that $\det(\Frob_{3}, M^{7}_{2,!*,\FF_{3}})=\det(\Frob_{7}, M^{7}_{2,!*,\FF_{7}})=p^{8}$ while $\det(\Frob_{5}, M^{7}_{2,!*,\FF_{5}})=-p^{8}$.
\end{proof}

\begin{lemma}\label{l:irred} For any prime $\ell>7$,   $\rho'_{\ell}$  is absolutely irreducible as a 3-dimensional orthogonal Galois representation.
\end{lemma}
\begin{proof}
If not, there should be a continuous character $\chi:\GQ\to\Qlbar^{\times}$ such that $\rho'_{\ell}$ is isomorphic to $\chi\oplus\chi^{-1}\oplus1$ up to semisimplification. By the Hodge-Tate property, we have $\chi=\chi_{\cyc}^{k}\chi_{0}$ where $\chi_{0}$ has finite order. Therefore for each prime $p\neq 3,5,7$ or $\ell$, we have
\begin{equation}\label{reducible}
\left(\frac{p}{105}\right)\frac{-m^{7}_{2}(p)-1}{p^{4}}=\Tr(\rho'_{\ell}(\Frob_{p}))=p^{k}\chi_{0}(p)+p^{-k}\chi_{0}(p)^{-1}+1.
\end{equation}
In particular, we have
\begin{equation*}
|m^{7}_{2}(p)+1|\geq p^{4}(p^{k}-p^{-k}-1)\geq p^{5}-p^{4}-p^{3}.
\end{equation*}
However, from \cite[Table 2.1]{Evans}, we see that neither $p=11$ nor $p=13$  (at least one of them is not equal to $\ell$) satisfies the above inequality. 

%We know that $\chi_{0}$ is ramified at $3,5,7$ and are quadratic there, and it is possibly ramified at $\ell$ and nowhere else. So we may write $\chi_{0}=\left(\frac{\cdot}{105}\right)\chi_{\ell}$ where $\chi_{\ell}$ is a Dirichlet character with conductor a power of $\ell$. For each prime $p\neq 3,5,7$ or $\ell$, we then have
%\begin{equation}\label{reducible}
%\left(\frac{p}{105}\right)\frac{-m^{7}_{2}(p)-1}{p^{4}}=\Tr(\rho'_{\ell}(\Frob_{p}))=p^{k}\left(\frac{p}{105}\right)\chi_{\ell}(p)+p^{-k}\left(\frac{p}{105}\right)\chi_{\ell}(p)^{-1}+1.
%\end{equation}
%In particular, the above number is rational, and hence $(p^{k}-p^{-k})\Im\chi_{\ell}(p)=0$. Since $k\neq0$, $\chi_{\ell}(p)$ is real for almost all $p$. Therefore $\chi_{\ell}$ is at most quadratic. But then we check from the numerics that neither $\chi_{\ell}=1$ nor $\chi_{\ell}=\left(\frac{\cdot}{\ell}\right)$ could make \eqref{reducible} possible.
\end{proof}

\begin{lemma}\label{l:not finite} For $\ell>7$, the Galois representation $\rho'_{\ell}$ does not have finite image.
\end{lemma}
\begin{proof} We need to eliminate the possibility that $\Im(\rho'_{\ell})$ is cyclic, dihedral, $A_{4}, S_{4}$ or $A_{5}$. Since $\rho'_{\ell}$ fits into a Frobenius-compatible system, it is unramified at $\ell$. 

By Lemma \ref{l:irred},  the image of $\rho'_{\ell}$ cannot be cyclic.

Suppose $\Im(\rho'_{\ell})=D_{2n}$ then we have a surjection $\ep:\GQ\surj D_{2n}\surj \{\pm1\}$ which is unramified away from $3,5,7$. Therefore for some $N|105$ we have $\ep=\left(\frac{\cdot}{N}\right)$. For those prime $p$ such that $\left(\frac{p}{N}\right)=-1$, $\Frob_{p}$ lies in the non-neutral component of $\Og_{2}\subset \SO_{3}\cong\PGL_{2}$, hence has trace $-1$. Therefore, whenever $\left(\frac{p}{N}\right)=-1$, we must have
\begin{equation*}
\left(\frac{p}{105}\right)\frac{-m^{7}_{2}(p)-1}{p^{4}}=-1.
\end{equation*}
Checking all possible $N$ and numerics of $m^{7}_{2}(p)$ from \cite[Table 2.1]{Evans} we can eliminate this possibility.

Suppose $\Im(\rho'_{\ell})=A_{4}, S_{4}$ or $A_{5}$, then $\Tr(\rho'_{\ell}(\Frob_{p}))$ can only take at most 5 different values. However, numerics from \cite[Table 2.1]{Evans} shows that $\left(\frac{p}{105}\right)\frac{-m^{7}_{2}(p)-1}{p^{4}}$ takes more than 5 values for varying $p$.
\end{proof}

\begin{cor}\label{c:distinct HT} The Hodge-Tate weights of $\rho'_{\ell}$ are not all zero.
\end{cor}
\begin{proof}
Since $\rho'_{\ell}$ is a subquotient of the cohomology of some fixed smooth projective variety $X$ over $\QQ$ (independent of $\ell$) by Theorem \ref{th:main}\eqref{motivic}, we may take $\ell$ such that $\rho'_{\ell}$ is crystalline. If the Hodge-Tate weights of $\rho'_{\ell}$ are all zero, it should be unramified at $\ell$. Then we know that $\rho'_{\ell}$ is only ramified at $3,5$ and $7$ with inertial image of order two. Hence for each $m$ the image of the reduction $\GQ\to\PGL_{2}(\ZZ/\ell^{m})$ is a finite Galois group $\Gal(K_{m}/\QQ)$ for $K_{m}$ of bounded discriminant independent of $m$. The inductive system $\varinjlim_{m}K_{m}$ should then stabilize and $\rho'_{\ell}$ must have finite image. But this is impossible by Lemma \ref{l:not finite}.
\end{proof}

\begin{cor}\label{c:odd} Let $c$ be a complex conjugation in $\GQ$, then $\rho'_{\ell}(c)\in\PGL_{2}(\Qlbar)$ is conjugate to $\diag(1,-1)$.
\end{cor}
\begin{proof} We only need to show that $\rho'_{\ell}(c)=\rho_{\ell}(c)\neq1$.

Consider the $\QQ$-Hodge structure $M^{7}_{2,\Hod}$ defined in Proposition \ref{p:mot dn}(2). By the comparison of singular and \'etale cohomology, there is an isomorphism $M^{7}_{2,\Hod}\otimes\Ql\isom M^{7}_{2,\ell}$ which intertwines the action of the complex conjugation on $M^{7}_{2,\Hod}$ (coming from the complex conjugation acting on the $\CC$-points of varieties defined over $\RR$), and the action of $c$ on $M^{7}_{2,\ell}$. By Proposition \ref{p:mot dn}(2), the Hodge-Tate weights of $M^{7}_{2,\Hod}$ are also $(4-k,4,4+k)$, i.e.,  the Hodge decomposition for $M^{7}_{2,\Hod}\otimes_{\QQ}\CC$ reads $H^{4+k,4-k}\oplus H^{4,4}\oplus H^{4-k,4+k}$, with each summand one-dimensional. We know from Corollary \ref{c:distinct HT} that $k>0$.

Since the Galois group $\Gal(\CC/\RR)$ acts on $M^{7}_{2,\Hod}\otimes_{\QQ}\CC$ (obtained by extending scalars from its action on $M^{7}_{2,\Hod}$) by permuting $H^{4+k,4-k}$ and $H^{4-k,4+k}$, we conclude that $\rho_{\ell}(c)\neq1$.
\end{proof}

\begin{lemma}\label{l:lifting} For sufficiently large prime $\ell$, there exists a continuous Galois representation
\begin{equation*}
\wt{\rho}_{\ell}:\GQ\to\GL_{2}(\Qlbar).
\end{equation*}
lifting the Galois representation $\rho'_{\ell}:\GQ\to\PGL_{2}(\Ql)\incl\PGL_{2}(\Qlbar)$.
Moreover, the lifting $\wt{\rho}_{\ell}$ may be chosen to satisfy the following properties:
\begin{enumerate}
\item It is unramified away from $3,5,7$ and $\ell$;
\item When $p=3$ or $7$, $\wt{\rho}_{\ell}|_{\calI_{p}}$ is tame, and maps a generator of $\calI^{t}_{p}$ to a matrix conjugate to $\diag(1,-1)$ in $\GL_{2}(\Qlbar)$;
\item When $p=5$, $\wt{\rho}_{\ell}|_{\calI_{p}}$ is tame, and maps a generator of $\calI^{t}_{p}$ to a matrix conjugate to $\diag(x,-x)$ in $\GL_{2}(\Qlbar)$,  where $x$ is a primitive $8\nth$ root of unity;
\item It is crystalline at $\ell$ with Hodge-Tate weights $(0,-k)$.
\end{enumerate}
\end{lemma}
\begin{proof}  
By a result of Patrikis \cite[Corollary 13.0.15]{Pat}, there exists a lifting $\wt{\rho}^{\dagger}$ which is ``geometric'' in the sense of Fontaine-Mazur, namely it is almost everywhere unramified and de Rham at $\ell$. Below we will define continuous characters $c_{p}:\ZZ_{p}^{\times}\to\Qlbar^{\times}$, one for each prime $p$, and trivial for almost all $p$. Ultimately the product of $\{c_{p}\}$ will give a character $c$ of $\GQ$ such that $\wt{\rho}^{\dagger}\otimes c^{-1}$ satisfies all the requirement in the lemma.

For $p\neq 3,5, 7$ or $\ell$, the restriction $\wt{\rho}^{\dagger}|_{\GQp}$ is abelian because $\wt{\rho}^{\dagger}|_{\calI_{p}}$ maps to scalar matrices in $\GL_{2}(E_{\l})$. Therefore $\wt{\rho}^{\dagger}|_{\calI_{p}}$ corresponds to a character $c_{p}:\ZZ_{p}^{\times}\to\Qlbar^{\times}$ by local class field theory. Since $\wt{\rho}^{\dagger}$ is ramified only at finitely many places, $c_{p}$ is trivial for almost all $p$.

For $p=3$ or $7$, $\rho'_{\ell}|_{\GQp}$ is of abelian type (i.e.,  its image lies in a torus), hence $\wt{\rho}^{\dagger}(\GQp)$ is also abelian and can be conjugated to $\diag(c_{p}, c'_{p})$. Here $c_{p}, c'_{p}:\GQp\to\Qlbar^{\times}$ are characters such that $c_{p}/c'_{p}$ restricted to $\calI_{p}$ has order two. 

For $p=5$,  $\rho'_{\ell}|_{\GQp}$ is of dihedral type, the restriction $\wt{\rho}^{\dagger}|_{\Gal(\Qpbar/\QQ_{p^{2}})}$
 is still abelian, and can be conjugated to $\diag(d_{1},d_{2})$ where $d_{1},d_{2}:\Gal(\Qpbar/\QQ_{p^{2}})\to\Qlbar^{\times}$ are characters. One easily sees that $d_{2}=d_{1}^{\sigma}$ where $\sigma\in\Gal(\QQ_{p^{2}}/\QQ_{p})$ is the nontrivial element. Also $d_{1}/d_{2}=d_{1}/d_{1}^{\sigma}$ when restricted to $\calI_{p}$ has order exactly two. Using class field theory we may write $d_{1}|_{\calI_{p}}:\ZZ^{\times}_{p^{2}}\to\Qlbar^{\times}$. Now $d_{1}/d_{1}^{\sigma}$ is trivial on $(\ZZ^{\times}_{p^{2}})^{2}$, i.e., $d_{1}|_{(\ZZ^{\times}_{p^{2}})^{2}}$ is invariant under $\Gal(\QQ_{p^{2}}/\QQ_{p})$. Identify $\ZZ^{\times}_{p^{2}}$ with $(1+p\ZZ_{p^{2}})\times\FF^{\times}_{p^{2}}$, we conclude that $d_{1}$ factors through
\begin{equation*}
\ZZ^{\times}_{p^{2}}\cong(1+p\ZZ_{p^{2}})\times\FF^{\times}_{p^{2}}\xrightarrow{\Nm, (-)^{(p+1)/2}}(1+p\ZZ_{p})\times \mu_{2(p-1)}(\Fpbar).
\end{equation*}
We denote the resulting character $(1+p\ZZ_{p})\times \mu_{2(p-1)}(\Fpbar)\to\Qlbar^{\times}$ by $\overline{d}$.
Since $d_{1}/d_{1}^{\sigma}$ is nontrivial, $d_{1}$ does not factor further through $(1+p\ZZ_{p})\times \FF_{p}^{\times}=\ZZ_{p}^{\times}$. Then there exists a character $c_{p}:\ZZ^{\times}_{p}=(1+p\ZZ_{p})\times \FF_{p}^{\times}\to\Qlbar^{\times}$ which is equal to $\overline{d}$ on  $(1+p\ZZ_{p})$. Then $d_{1}c_{p}^{-1}$ (which is now tame) maps a generator of $\calI^{t}_{p}$ to a primitive $2(p-1)=8\nth$ root of unity.

We suppose $\ell$ is sufficiently large so that $\rho'_{\ell}$ is crystalline. At $\ell$, the Hodge-Tate weights are $(a,a-k)$ for some $a\in\ZZ$. Since $\rho'_{\ell}|_{\GQl}$ admits a de Rham lifting, it also admits a crystalline lifting $\phi:\GQl\to\GL_{2}(\Qlbar)$, by a result of Conrad \cite[Corollary 6.7]{Conrad}. By making a Tate twist we may assume $\phi$ has Hodge-Tate weights $(0,-k)$. We then have $\wt{\rho}^{\dagger}_{\ell}|_{\GQl}=\phi\cdot d_{\ell}$ for some continuous character $d_{\ell}:\GQl^{ab}\to\Qlbar^{\times}$. The restriction of $d_{\ell}$ to the inertia group then gives a character $c_{\ell}:\ZZ_{\ell}^{\times}\to\Qlbar^{\times}$. 
%We have already seen that $k\neq0$ from Corollary \ref{c:distinct HT}, hence $1\leq k\leq 3$. 

Let $c:\GQ^{ab}\cong\prod_{p}\ZZ_{p}^{\times}\to\Qlbar^{\times}$ be the product of the $c_{p}$ defined above for all primes $p$. This makes sense because $c_{p}$ is trivial for almost all $p$. Also $c$ is continuous since each $c_{p}$ is. Then $\wt{\rho}_{\ell}:=\wt{\rho}^{\dagger}\otimes c^{-1}$ satisfies all the requirements.
\end{proof}

\subsubsection{Proof of Theorem \ref{th:Sym7}}\label{pf Sym7}
Fix a lifting $\wt{\rho}_{\ell}$ as in Lemma \ref{l:lifting} ($\ell$ is sufficiently large). 

The Artin conductor of $\wt{\rho}_{\ell}$ at $3,5$ and $7$ are $1,2$ and $1$ respectively according to the description of local behavior of $\wt{\rho}_{\ell}$ in Lemma \ref{l:lifting}.

From the description on the behavior of $\wt{\rho}_{\ell}$ at $p=3,5$ and $7$ again, we see that under $\wt{\rho}_{\ell}$ a generator of $\calI_{3}$ or $\calI_{7}$ gets mapped to $\diag(1,-1)$, with determinant $-1$. A generator of $\calI_{5}$ gets mapped to $\diag(x,-x)$, where $x$ is a primitive $8\nth$ root of unity, with determinant a primitive $4\nth$ root of unity. Therefore
\begin{equation*}
\det(\wt{\rho}_{\ell})\cong\left(\frac{\cdot}{21}\right)\ep_{5}\ep_{\ell}\chi_{\cyc,\ell}^{-k}.
\end{equation*}
Here $\ep_{5}:\GQ\surj(\ZZ/5\ZZ)^{\times}\isom\mu_{4}(\Qlbar)$ is a quartic Dirichlet character of conductor $5$; $\ep_{\ell}$ is a Dirichlet character of conductor a power of $\ell$.

Let $\rhowtbar_{\ell}$ be the reduction mod $\ell$ of $\wt{\rho}_{\ell}$, i.e., a continuous representation $\GQ\to\GL_{2}(\overline{\FF}_{\ell})$. Since $\wt{\rho}_{\ell}$ is crystalline at $\ell$ with Hodge-Tate weights $(0,-k)$, and $\ell$ is assumed to be sufficiently large, Fontaine-Laffaille theory implies that the Serre weight of $\rhowtbar_{\ell}$ is $k+1$ (as in the proof of Theorem \ref{th:Sym8}). Therefore the $\ell$-part of $\det(\rhowtbar_{\ell})$ should be $\omega_{\ell}^{-k}$ here $\omega_{\ell}$ is $\chi_{\cyc,\ell}$ mod $\ell$. Hence $\ep_{\ell}$ is trivial modulo $\ell$, and therefore admitting a square root $\ep^{1/2}_{\ell}$. Changing $\wt{\rho}_{\ell}$ to $\wt{\rho}_{\ell}\ep^{-1/2}_{\ell}$, our new lifting, which we still denote by $\wt{\rho}_{\ell}$ (and still satisfying the properties of Lemma \ref{l:lifting}), satisfies
\begin{equation}\label{det lifting}
\det(\wt{\rho}_{\ell})\cong\left(\frac{\cdot}{21}\right)\ep_{5}\chi_{\cyc,\ell}^{-k}.
\end{equation}

By Corollary \ref{c:odd}, a complex conjugation $c$ gets mapped to $\diag(1,-1)\in\PGL_{2}(\Qlbar)$ under $\rho'_{\ell}$, therefore $\wt{\rho}_{\ell}$ is odd. Since $\left(\frac{\cdot}{21}\right)\ep_{5}$ takes value $-1$ on $c$ as well, we conclude from \eqref{det lifting} that $k$ is even. Since $1\leq k\leq 3$, we get $k=2$. 

Summarizing our knowledge about $\rhowtbar_{\ell}$ for large $\ell$ at this point. It is absolutely irreducible (by Lemma \ref{l:irred}), odd, with Artin conductor $525=3\cdot 5^{2}\cdot 7$ (away from $\ell$ part), Serre weight equal to $k+1=3$ and determinant given by \eqref{det lifting} modulo $\ell$. 

We now appeal to Serre's argument for modularity \cite[\S4.8]{Serre} again. Although we do not know whether we can choose the liftings $\{\wt{\rho}_{\ell}\}$ to form a Frobenius-compatible family, we have the following weaker property.  The lifting $\wt{\rho}_{\ell}$ satisfying the condition \ref{det lifting} is unique up to twisting by a quadratic Dirichlet character. Hence for $p\neq 3,5$ or $7$, $\Tr(\wt{\rho}_{\ell}(\Frob_{p}))^{2}$ is independent of $\ell$. This property will play the role of Frobenius-compatibility in the argument below. Therefore, by Serre's conjecture, for sufficiently large $\ell$, $\rhowtbar_{\ell}$ comes from the mod $\ell$ Galois representation of a Hecke eigenform of weight 3, level 525 and nebentypus $\ep_{f}=\left(\frac{\cdot}{21}\right)\ep_{5}$. Since there are only finitely many possibilities for such eigenforms, we conclude that there is such a newform $f$ such that $\rhobar_{f,\ell}\cong\rhowtbar_{\ell}$ for infinitely many $\ell$. Therefore for each prime $p\neq 3,5$ or $7$, $\Tr(\wt{\rho}_{\ell}(\Frob_{p}))^{2}$ is congruent mod $\ell$ to the Fourier coefficient $a_{f}(p)^{2}$ for infinitely many $\ell$. Hence for each  prime $p\neq 3,5$ or $7$, we have
\begin{equation*}
\Tr(\wt{\rho}_{\ell}(\Frob_{p}))^{2}=a_{f}(p)^{2}.
\end{equation*}
This means that
\begin{equation}\label{trace rho'}
\Tr(\rho'_{\ell}(\Frob_{p}))=\frac{\Tr(\wt{\rho}_{\ell}(\Frob_{p}))^{2}}{\det(\wt{\rho}_{\ell}(\Frob_{p}))}-1=a_{f}(p)^{2}p^{-2}\ep_{f}(p)^{-1}-1.
\end{equation}
By Theorem \ref{th:intro}, we see that
\begin{equation*}
-m^{7}_{2}(p)=1+p^{4}\Tr(\rho_{\ell}(\Frob_{p}))=1+p^{4}\left(\frac{p}{105}\right)\Tr(\rho'_{\ell}(\Frob_{p})).
\end{equation*}
Plugging in \eqref{trace rho'}, we get \eqref{m72}.

\appendix

\section{Fourier transform}\label{app}
In the appendix, we collect some facts about homogeneous Fourier transform that we need in the main body of the paper. We follow Laumon's original paper \cite{Laumon}.

\subsection{Fourier-Deligne transform}\label{Fourier-Del} Let $S$ be a scheme of finite type over $k$. Let $\AA$ and $\AA^{\vee}$ both be the affine line $\AA^{1}_{S}$ over $S$. We recall the Fourier-Deligne transform 
\begin{equation*}
\Four_{\psi}: D^{b}_{c}(\AA,\Ql(\mu_{p}))\to D^{b}_{c}(\AA^{\vee},\Ql(\mu_{p})).
\end{equation*}
It is defined as follows. Consider the following diagram
\begin{equation*}
\xymatrix{ & \AA\times_{S}\AA^{\vee}\ar[rr]^{m}\ar[dl]^{\pr}\ar[dr]^{\pr^{\vee}}& & \AA^{1}\\
\AA & & \AA^{\vee}}
\end{equation*}
where $\pr,\pr^{\vee}$ are projections and $m$ is the multiplication map. For $\calF\in D^{b}_{c}(\AA,\Ql(\mu_{p}))$, its Fourier transform is defined as
\begin{equation*}
\Four_{\psi}(\calF)=\pr^{\vee}_{!}(\pr^{*}\calF\otimes m^{*}\AS_{\psi})[1].
\end{equation*}

\subsection{Laumon's homogeneous Fourier transform}\label{ss:HF}
We still let $\AA$ and $\AA^{\vee}$ be affine lines over $S$, except now $S$ is allowed to be any scheme over finite type over $\Zl$. The notation $\Gm$ now means the torus $\GmS$ over $S$, which acts on $\AA$ and  $\AA^{\vee}$. We shall consider the equivariant derived categories $D^b_{\Gm}(\AA,\Ql)$ and $D^b_{\Gm}(\AA^{\vee},\Ql)$. Laumon defines a homogeneous Fourier transform
\begin{equation*}
\HF: D^b_{\Gm}(\AA,\Ql)\to D^b_{\Gm}(\AA^{\vee},\Ql).
\end{equation*}
Its definition uses a similar diagram of stacks over $S$
\begin{equation*}
\xymatrix{ & [\AA/\Gm]\times_{S}[\AA^{\vee}/\Gm]\ar[rr]^{\overline{m}}\ar[dl]^{\bpr}\ar[dr]^{\bpr^{\vee}}& & [\AA^{1}_{S}/\Gm]\\
[\AA/\Gm] & & [\AA^{\vee}/\Gm]}
\end{equation*}
Let $\beta:S=[\Gm/\Gm]\incl[\AA^{1}_{S}/\Gm]$ be the open inclusion, and let $\Psi=\beta_{*}\Ql\in D^{b}_{\Gm}(\AA^{1}_{S},\Ql)$.
For $\calF\in D^{b}_{\Gm}(\AA,\Ql)$, its homogeneous Fourier transform is defined as
\begin{equation*}
\HF(\calF)=\bpr^{\vee}_{!}(\bpr^{*}\calF\otimes \overline{m}^{*}\Psi).
\end{equation*}

We summarize the major properties of the homogeneous Fourier transform.
\begin{theorem}\label{th:Fourier} Let $S$ be a scheme of finite type over $\Zl$. Let $\calF\in D^b_{\Gm}(\AA,\Ql)$.
\begin{enumerate}
\item\label{F:sp} Let $\calF_0$ and $\calF_1$ be restriction of $\calF$ to the sections $0_S$ and $1_S$ of the projection $\AA\to S$ (recall $\AA=\AA^{1}_{S}$). Then there is a specialization morphism
\begin{equation*}
\Sp(\calF):\calF_0\to\calF_1
\end{equation*}
functorial in $\calF$, such that the generic stalk $\HF(\calF)_{\eta}$ of the homogeneous Fourier transform $\HF(\calF)$ (i.e., the restriction of $\HF(\calF)$ to $[\Gm/\Gm]=S$) fits into a distinguished triangle
\begin{equation}\label{sp seq}
\calF_{0}\xrightarrow{\Sp(\calF)}\calF_{1}\to\HF(\calF)_{\eta}\to\calF_{0}[1].
\end{equation}
\item\label{FD=HF} Suppose $S$ is over $\FF_p$, then we have a canonical isomorphism in $D^{b}_{c}(\AA^{\vee},\Ql(\mu_{p}))$
\begin{equation*}
\Four_{\psi}(\calF\otimes\Ql(\mu_{p}))\cong\HF(\calF)\otimes\Ql(\mu_{p}).
\end{equation*}
\item\label{duality} There is an isomorphism of functors:
\begin{equation*}
\HF\circ\DD\cong(\DD\circ\HF)(-1).
\end{equation*}
%such that the composition
%\begin{equation*}
%\HF\cong\HF\circ\DD\circ\DD\xrightarrow{\delta\circ\DD}(\DD\circ\HF\circ\DD)(-1)\xrightarrow{\DD\circ\delta}(\DD\circ(\DD\circ\HF)(-1))(-1)\cong\HF.
%\end{equation*}
%is $-\id$.
\item\label{perv} The functor $\HF$ is $t$-exact with respect to the perverse $t$-structure on $D^b_{\Gm}(\AA,\Ql)$.
\item\label{conv} Let $*_{+}$ denote the convolution on $D^{b}_{\Gm}(\AA,\Ql)$ using the additive group structure. Then for $\calF_{1},\calF_{2}\in D^{b}_{\Gm}(\AA,\Ql)$, we have a canonical isomorphism
\begin{equation*}
\HF(\calF_{1}*_{+}\calF_{2})[1]\cong\HF(\calF_{1})\otimes\HF(\calF_{2}).
\end{equation*}
\end{enumerate}
\end{theorem}

\begin{proof}
(1) Let $h:\AA\to S$ be the structure morphism. By definition and proper base change, we have
\begin{equation}\label{b*}
\HF(\calF)_{\eta}=h_{!}(\beta_{*}\otimes\calF).
\end{equation}
Let $\alpha:\BB\Gm\incl[\AA/\Gm]$ be the closed embedding of the origin. Consider the distinguished triangle in $D^{b}_{\Gm}(\AA)$:
\begin{equation}\label{beta}
\beta_{!}\Ql\to\beta_{*}\Ql\to\alpha_{*}\alpha^{*}\beta_{*}\Ql\to\beta_{!}\Ql[1]
\end{equation}
Also let $g:S\to \BB\Gm$ be the quotient morphism, then by \cite[Lemme 1.4(ii)]{Laumon} we have an isomorphism in $D^{b}_{\Gm}(S,\Ql)$
\begin{equation*}
\alpha^{*}\beta_{*}\Ql\cong  g_{!}\Ql[1].
\end{equation*}
Therefore \eqref{beta}, after rotating the triangle, becomes
\begin{equation*}
\alpha_{*}g_{!}\Ql\to \beta_{!}\Ql\to\beta_{*}\Ql\to\alpha_{*}g_{!}\Ql[1].
\end{equation*}
Tensoring with $\calF$ and applying the functor $h_{!}$, we get
\begin{equation}\label{pre seq}
h_{!}(\alpha_{*}g_{!}\Ql\otimes\calF)\to h_{!}(\beta_{!}\Ql\otimes\calF)\to h_{!}(\beta_{*}\otimes\calF)\to.
\end{equation}
By projection formula
\begin{eqnarray}
\label{a!} h_{!}(\alpha_{*}g_{!}\Ql\otimes\calF)=h_{!}\alpha_{!}g_{!}(g^{*}\alpha^{*}\calF)=\calF_{0};\\
\label{b!} h_{!}(\beta_{!}\Ql\otimes\calF)=h_{!}\beta_{!}(\beta^{*}\calF)=\calF_{1}.
\end{eqnarray}
Plugging \eqref{a!},\eqref{b!} and \eqref{b*} into \eqref{pre seq}, we get the desired triangle \eqref{sp seq}, which also includes the definition of the functorial map $\Sp(\calF):\calF_{0}\to\calF_{1}$ as part of the triangle.

For \eqref{FD=HF}, see \cite[Th\'eor\`eme 2.2]{Laumon}.

For \eqref{duality}, see \cite[Th\'eor\`eme 4.1]{Laumon}.

For \eqref{perv}, see \cite[Th\'eor\`eme 4.2]{Laumon}.

\eqref{conv} is the special case of \cite[Lemme 1.7]{Laumon} in the case of the linear map $\AA^{2}\to\AA^{1}$ given by the addition of two coordinates.
\end{proof}

\section{Proof of Theorem \ref{th:Sym8}--Computational part \\ by Christelle Vincent }\label{app:Vincent}

In this Appendix we give a brief account of the computational part of the proof of Theorem \ref{th:Sym8}. This computation was carried out using the open source software Sage \cite{sage}, and part of the computation was carried out on the Sage Cloud \cite{sagecloud}. All Sage code written for this project is available online \cite{vincent}.

Fix a prime $p$. We first compute the numbers $m_2^8(p)$. In the notation of Section \ref{ss:intro Kl}, for each $a \in \mathbb{F}_p^{\times}$, we have
\begin{equation*}
\Kl_2(p;a) = - (\alpha_a + \beta_a),
\end{equation*}
for $\alpha_a$ and $\beta_a$ the eigenvalues of $\Frob_a$ acting on $(\Kl_2)_a$, and
\begin{equation*}
m_2^8(p) = \sum_{a \in \mathbb{F}_p^{\times}} \sum_{i=0}^8 \alpha_a^i \beta_a^{8-i}.
\end{equation*}

Sage already has a function to compute twisted Kloosterman sums. Since in our case the Dirichlet character is trivial, we modified the code to remove the dependence on the Dirichlet character. Even though it is possible to return the exact value of the Kloosterman sum, we compute it as an approximate complex number to speed up computations.

Using these values, we may compute, for $n \geq 1$, the sums
\begin{equation*}
S_n(p) = \sum_{a \in \mathbb{F}_p^{\times}} (\Kl_2(p;a))^n.
\end{equation*}
Since we only have an approximation for $\Kl_2(p;a)$, to recover the exact value of $S_n(p)$ we use the congruence 
\begin{equation*}
S_n(p)  \equiv (-1)^{n-1}(n-1)p \pmod{p^2}
\end{equation*}
proved in \cite{choievans}.

Using the $S_n(p)$'s and the fact that $\alpha_a \beta_a =p$, we have
\begin{equation*}
m_2^8(p) = S_8(p) - 7pS_6(p)+15p^2S_4(p)-10p^3S_2(p)+p^5-p^4+10p^3-15p^2+7p-1.
\end{equation*}

From there we can obtain the Fourier coefficients of the form $f$ which we are looking for:
\begin{equation*}
a_f(p) = -(m_2^8(p)+1+p^4)/p^2.
\end{equation*}

For each
\begin{equation*}
N \in \{ 3, 6, 12, 24, 48, 96, 192, 384, 768 \},
\end{equation*}
and $k = 2,4$, and each
\begin{equation*}
N \in \{ 3, 6, 12, 24, 48, 96, 192, 384 \},
\end{equation*}
and $k=6,8$ we compute the space of cuspidal new modular symbols of weight $k$ and level $\Gamma_0(N)$. We then decompose it into invariant subspaces under the action of the Hecke operators of index coprime to the level. Each of these spaces is simple, and we compute the eigenvalue of the Hecke operator $T_5$ on each space. If it is not equal to $a_f(5)=-66$, we can immediately discard the space. In only two cases is the eigenvalue equal to $-66$: When $N=6$ and $k=6$, and in one of the subspaces obtained when $N=48$ and $k=6$. In the second case, we check the eigenvalue of the Hecke operator $T_7$, and it is $-176$, which is not equal to $a_f(7)=176$. We can therefore discard this space too. 

When $N=768$ and $k=6$ or $k=8$, the spaces become too large and the computation becomes prohibitively long. So we instead compute the space of cuspidal new modular symbols over the field $\mathbb{F}_{13}$. In each case we then compute the characteristic polynomial of the Hecke operator $T_5$ and check whether the number $-66$ is a root of this polynomial. It is not, and therefore neither of these subspaces contains an eigenform with eigenvalue for $T_5$ congruent to $-66$ modulo $13$. We can therefore discard those spaces as well.

This leaves us with the unique new cusp form of weight $6$ and level $\Gamma_0(6)$. For good measure, we check that its Hecke eigenvalues agree with our list of $a_f(p)$, and they do for all primes $p$ such that $3 \leq p <1000$.

%%% reference %%%

\end{document}